\pdfoutput=1
\RequirePackage{ifpdf}
\ifpdf 
\documentclass[pdftex]{sigma}
\else
\documentclass{sigma}
\fi

\usepackage{array}
\usepackage{eqparbox}
\usepackage{lscape}
\usepackage{multirow}
\usepackage{tikz}
\usepackage{tikz-cd}
\usepackage{verbatim}

\numberwithin{equation}{section}

\newtheorem*{propositionA}{Proposition A}
\newtheorem*{theoremB}{Theorem B}
\newtheorem*{theoremC}{Theorem C}
\newtheorem*{theoremDminus}{Theorem D$_-$}
\newtheorem*{theoremDplus}{Theorem D$_+$}
\newtheorem*{theoremDzero}{Theorem D$_0$}

\newtheorem{Theorem}{Theorem}[section]
\newtheorem{Corollary}[Theorem]{Corollary}
\newtheorem{Lemma}[Theorem]{Lemma}
\newtheorem{Proposition}[Theorem]{Proposition}
\newtheorem{algorithm}[Theorem]{Algorithm}
 { \theoremstyle{definition}
\newtheorem{Definition}[Theorem]{Definition}
\newtheorem{Example}[Theorem]{Example}
\newtheorem{Remark}[Theorem]{Remark} }

\newlength{\splittingWidth}
\newcolumntype{C}[1]{>{\centering\arraybackslash$}p{#1}<{$}}

\newcommand{\bbA}{\mathbb{A}}

\newcommand{\bbC}{\mathbb{C}}
\newcommand{\bbE}{\mathbb{E}}
\newcommand{\bbH}{\mathbb{H}}
\newcommand{\bbI}{\mathbb{I}}
\newcommand{\bbJ}{\mathbb{J}}
\newcommand{\bbK}{\mathbb{K}}
\newcommand{\bbL}{\mathbb{L}}
\newcommand{\bbO}{\mathbb{O}}
\newcommand{\bbP}{\mathbb{P}}
\newcommand{\bbR}{\mathbb{R}}
\newcommand{\bbS}{\mathbb{S}}
\newcommand{\bbT}{\mathbb{T}}
\newcommand{\bbU}{\mathbb{U}}
\newcommand{\bbV}{\mathbb{V}}
\newcommand{\bbW}{\mathbb{W}}
\newcommand{\bbZ}{\mathbb{Z}}

\newcommand{\mbC}{\mathbf{C}}
\newcommand{\mbc}{\mathbf{c}}
\newcommand{\mbD}{\mathbf{D}}
\newcommand{\mbE}{\mathbf{E}}
\newcommand{\mbg}{\mathbf{g}}
\newcommand{\mbH}{\mathbf{H}}
\newcommand{\mbJ}{\mathbf{J}}
\newcommand{\mbK}{\mathbf{K}}
\newcommand{\mbL}{\mathbf{L}}

\newcommand{\mbS}{\mathbf{S}}

\newcommand{\mcA}{\mathcal{A}}
\newcommand{\mcD}{\mathcal{D}}
\newcommand{\mcE}{\mathcal{E}}
\newcommand{\mcF}{\mathcal{F}}
\newcommand{\mcG}{\mathcal{G}}
\newcommand{\mcH}{\mathcal{H}}
\newcommand{\mcL}{\mathcal{L}}
\newcommand{\mcM}{\mathcal{M}}
\newcommand{\mcO}{\mathcal{O}}
\newcommand{\mcU}{\mathcal{U}}
\newcommand{\mcV}{\mathcal{V}}
\newcommand{\mcW}{\mathcal{W}}
\newcommand{\mcX}{\mathcal{X}}
\newcommand{\mcY}{\mathcal{Y}}
\newcommand{\mcZ}{\mathcal{Z}}

\newcommand{\mfaut}{\mathfrak{aut}}
\newcommand{\mfg} {\mathfrak{g} }
\newcommand{\mfgl} {\mathfrak{gl} }

\newcommand{\mfp} {\mathfrak{p} }
\newcommand{\mfq} {\mathfrak{q} }

\newcommand{\mfsl} {\mathfrak{sl} }
\newcommand{\mfso} {\mathfrak{so} }
\newcommand{\mfsu} {\mathfrak{su} }

\newcommand{\sfP}{\mathsf{P}}

\DeclareMathOperator{\Ad} {Ad}
\DeclareMathOperator{\aEs} {aEs}
\DeclareMathOperator{\arsinh}{arsinh}

\DeclareMathOperator{\End} {End}
\DeclareMathOperator{\G} {G}
\DeclareMathOperator{\GL} {GL}

\DeclareMathOperator{\Hol} {Hol}
\DeclareMathOperator{\id} {id}
\DeclareMathOperator{\im} {im}
\DeclareMathOperator{\rank} {rank}
\DeclareMathOperator{\ReOp} {Re}
\DeclareMathOperator{\sech} {sech}
\DeclareMathOperator{\skewOp}{skew}
\DeclareMathOperator{\SL} {SL}
\DeclareMathOperator{\SO} {SO}
\DeclareMathOperator{\Sp} {Sp}
\DeclareMathOperator{\SU} {SU}

\DeclareMathOperator{\tr} {tr}
\DeclareMathOperator{\U} {U}

\renewcommand{\Re}{\ReOp}

\newcommand{\astPhi} {{\ast_{\Phi}}}
\newcommand{\barphi} {\bar\phi}
\newcommand{\hook} {\, \lrcorner \,}
\newcommand{\hatg} {\smash{\hat g}}
\newcommand{\hatK} {\smash{\hat K}}
\newcommand{\wtbbO} {\smash{\widetilde{\bbO}}}

\newcommand{\ul}[1]{\underline{#1}}
\newcommand{\wt}[1]{\widetilde{#1}}

\newcommand{\tractorT}[3]{
	\begin{pmatrix}
		{#3} \\
		{#2} \\
		{#1}
	\end{pmatrix}
}

\newcounter{tractorQcounter}

\newcommand{\tractorQ}[4]{
	\def\arraystretch{1.1}
	\begin{pmatrix}
		\multicolumn{3}{c}{#4} \\
		\eqmakebox[\thetractorQcounter]{$#2\!\!\!\!$} & | & \eqmakebox[\thetractorQcounter]{$\!\!\!\!#3$} \\
		\multicolumn{3}{c}{#1}
	\end{pmatrix}
	\stepcounter{tractorQcounter}
}

\newcommand{\flplus} 
{
\hspace{0.1cm}
\begin{tikzpicture}[baseline=-0.582ex]
 \draw [line width=0.24pt](-0.1129, 0) -- (0.1129, 0) -- (0, 0) -- (0, 0.1129) -- (0, -0.1129) arc (270:90:0.1129) -- (0, 0);
\end{tikzpicture}
\hspace{0.1cm}
}

\begin{document}

\allowdisplaybreaks

\newcommand{\arXivNumber}{1606.01069}

\renewcommand{\PaperNumber}{004}

\FirstPageHeading

\ShortArticleName{The Geometry of Almost Einstein $(2, 3, 5)$ Distributions}

\ArticleName{The Geometry of Almost Einstein\\ $\boldsymbol{(2, 3, 5)}$ Distributions}

\Author{Katja SAGERSCHNIG~$^\dag$ and Travis WILLSE~$^\ddag$}

\AuthorNameForHeading{K.~Sagerschnig and T.~Willse}

\Address{$^\dag$~Politecnico di Torino, Dipartimento di Scienze Matematiche,\\
\hphantom{$^\dag$}~Corso Duca degli Abruzzi 24, 10129 Torino, Italy}
\EmailD{\href{mailto:katja.sagerschnig@univie.ac.at}{katja.sagerschnig@univie.ac.at}}

\Address{$^\ddag$~Fakult\"{a}t f\"{u}r Mathematik, Universit\"{a}t Wien, Oskar-Morgenstern-Platz 1, 1090 Wien, Austria}
\EmailD{\href{mailto:travis.willse@univie.ac.at}{travis.willse@univie.ac.at}}

\ArticleDates{Received July 26, 2016, in f\/inal form January 13, 2017; Published online January 19, 2017}

\Abstract{We analyze the classic problem of existence of Einstein metrics in a given conformal structure for the class of conformal structures inducedf Nurowski's construction by (oriented) $(2, 3, 5)$ distributions. We characterize in two ways such conformal structures that admit an almost Einstein scale: First, they are precisely the oriented conformal structures $\mathbf{c}$ that are induced by at least two distinct oriented $(2, 3, 5)$ distributions; in this case there is a $1$-parameter family of such distributions that induce $\mathbf{c}$. Second, they are characterized by the existence of a holonomy reduction to ${\rm SU}(1, 2)$, ${\rm SL}(3, {\mathbb R})$, or a particular semidirect product ${\rm SL}(2, {\mathbb R}) \ltimes Q_+$, according to the sign of the Einstein constant of the corresponding metric. Via the curved orbit decomposition formalism such a reduction partitions the underlying manifold into several submanifolds and endows each ith a geometric structure. This establishes novel links between $(2, 3, 5)$ distributions and many other geometries~-- several classical geometries among them~-- including: Sasaki--Einstein geometry and its paracomplex and null-complex analogues in dimension~$5$; K\"{a}hler--Einstein geometry and its paracomplex and null-complex analogues, Fef\/ferman Lorentzian conformal structures, and para-Fef\/ferman neutral conformal structures in dimension $4$; CR geometry and the point geometry of second-order ordinary dif\/ferential equations in dimension $3$; and projective geometry in dimension $2$. We describe a generalized Fef\/ferman construction that builds from a $4$-dimensional K\"{a}hler--Einstein or para-K\"{a}hler--Einstein structure a family of $(2, 3, 5)$ distributions that induce the same (Einstein) conformal structure. We exploit some of these links to construct new examples, establishing the existence of nonf\/lat almost Einstein $(2, 3, 5)$ conformal structures for which the Einstein constant is positive and negative.}


\Classification{32Q20; 32V05; 53A30; 53A40; 53B35; 53C15; 53C25; 53C29; 53C55; 58A30}

\section{Introduction}

A (pseudo-)Riemannian metric $g_{ab}$ is said to be \textit{Einstein} if its Ricci curvature $R_{ab}$ is a multiple of~$g_{ab}$. The problem of determining whether a given conformal structure (locally) contains an Einstein metric has a rich history, and dates at least to Brinkmann's seminal investiga\-tions~\mbox{\cite{BrinkmannMapped, BrinkmannRiemann}} in the 1920s. Other signif\/icant contributions have been made by, among others, Hanntjes and Wrona~\cite{HaantjesWrona}, Sasaki~\cite{Sasaki}, Wong~\cite{Wong}, Yano~\cite{Yano}, Schouten~\cite{Schouten}, Szekeres~\cite{Szekeres}, Kozameh, Newman, and Tod~\cite{KNT}, Bailey, Eastwood, and Gover~\cite{BEG}, Fef\/ferman and Graham~\cite{FeffermanGraham}, and Gover and Nurowski~\cite{GoverNurowski}. Developments in this topic in the last quarter century in particular have stimulated substantial development both within conformal geometry and far beyond it.

In the watershed article \cite{BEG}, Bailey, Eastwood, and Gover showed that the existence of such a~metric in a conformal structure (here, and always in this article, of dimension $n \geq 3$) is governed by a second-order, conformally invariant linear dif\/ferential operator $\smash{\Theta_0^{\mcV}}$ that acts on sections of a natural line bundle $\mcE[1]$ (we denote by $\mcE[k]$ the $k$th power of $\mcE[1]$): Every conformal structure $(M, \mbc)$ is equipped with a canonical bilinear form $\mbg \in \Gamma(S^2 T^*M \otimes \mcE[2])$, and a nowhere-vanishing section $\sigma$ in the kernel of $\Theta_0^{\mcV}$ determines an Einstein metric $\sigma^{-2} \mbg$ in $\mbc$ and vice versa. Writing the dif\/ferential equation $\smash{\Theta_0^{\mcV}(\sigma) = 0}$ as a f\/irst-order system and prolonging once yields a closed system and hence determines a conformally invariant connection~$\nabla^{\mcV}$ on a natural vector bundle $\mcV$, called the (standard) \textit{tractor bundle}. (The conformal structure determines a parallel \textit{tractor metric} $H \in \Gamma(S^2 \mcV^*)$.) By construction, this establishes a bijective correspondence between Einstein metrics in $\mbc$ and parallel sections of this bundle satisfying a~genericity condition. This framework immediately suggests a natural relaxation of the Einstein condition: A~section of the kernel of~$\Theta_0^{\mcV}$ is called an \textit{almost Einstein scale}, and it determines an Einstein metric on the complement of its zero locus and becomes singular along that locus. A~conformal structure that admits a~nonzero almost Einstein scale is itself said to be \textit{almost Einstein}, and, somewhat abusively, the metric it determines is sometimes called an \textit{almost Einstein metric} on the original manifold. The generalization to the almost Einstein setting has substantial geometric consequences: The zero locus itself inherits a geometric structure, which can be realized as a natural limiting geometric structure of the metric on the complement. This arrangement leads to the notion of conformal compactif\/ication, which has received substantial attention in its own right, including in the physics literature~\cite{Susskind}.

\looseness=-1 This article investigates the problem of existence of almost Einstein scales, as well as the geo\-metric consequences of existence of such a scale, for a fascinating class of conformal structures that arise naturally from another geometric structure: A \textit{$(2, 3, 5)$ distribution} is a $2$-plane distribution $\mbD$ on a $5$-manifold which is maximally nonintegrable in the sense that $[\mbD, [\mbD, \mbD]] = TM$. This geometric structure has attracted substantial interest, especially in the last few decades, for numerous reasons: $(2, 3, 5)$ distributions are deeply connected to the exceptional simple Lie algebra of type $\G_2$ (in fact, the study of these distributions dates to 1893, when Cartan~\cite{CartanModel} and Engel~\cite{EngelModel} simultaneously realized that Lie algebra as the inf\/initesimal symmetry algebra of a~distribution of this type), they are the subject of Cartan's most involved application of his ce\-lebrated method of equivalence~\cite{CartanFiveVariables}, they comprise a f\/irst class of distributions with continuous local invariants, they arise na\-turally from a class of second-order Monge equations, they arise naturally from mechanical systems entailing one surface rolling on another without slipping or \mbox{twisting}~\cite{AgrachevSachkov, AnNurowski, BorMontgomery}, they can be used to construct pseudo-Riemannian metrics whose holonomy group is~$\G_2$~\cite{GrahamWillse, LeistnerNurowski}, they are natural compactifying structures for indef\/inite-signature nearly K\"{a}hler geometries in dimension~$6$~\cite{GPW}, and they comprise an interesting example of a broad class of so-called \textit{parabolic geometries}~\cite[Section~4.3.2]{CapSlovak}. (Here and henceforth, the symbol~$\G_2$ refers to the algebra automorphism group of the split octonions; this is a split real form of the complex simple Lie group of type $\G_2$.) For our purposes their most important feature is the natural construction, due to Nurowski \cite[Section~5]{Nurowski}, that associates to any $(2, 3, 5)$ distribution~$(M, \mbD)$ a~conformal structure $\mbc_{\mbD}$ of signature $(2, 3)$ on $M$, and it is these structures, which we call \textit{$(2, 3, 5)$ conformal structures}, whose almost Einstein geometry we investigate. For expository convenience, we restrict ourselves to the oriented setting: A~$(2, 3, 5)$ distribution $(M, \mbD)$ is \textit{oriented} if\/f $\mbD \to M$ is an oriented bundle; an orientation of $\mbD$ determines an orientation of~$M$ and vice versa.

The key ingredient in our analysis is that, like almost Einstein conformal structures, $(2, 3, 5)$ conformal structures can be characterized in terms of the holonomy group of the normal tractor connection, $\nabla^{\mcV}$ (for any oriented conformal structure of signature $(2, 3)$, $\nabla^{\mcV}$ has holonomy contained in $\SO(H) \cong \SO(3, 4)$): An oriented conformal structure $\mbc$ of signature $(2, 3)$ coincides with $\mbc_{\mbD}$ for some $(2, 3, 5)$ distibution $\mbD$ if\/f the holonomy group of $\nabla^{\mcV}$ is contained inside $\G_2$, or equivalently, if\/f there is a parallel tractor $3$-form, that is, a section $\Phi \in \Gamma(\Lambda^3 \mcV^*)$, compatible with the conformal structure in the sense that the pointwise stabilizer of $\Phi_p$ in $\GL(\mcV_p)$ (at any, equivalently, every point $p$) is isomorphic to $\G_2$ and is contained inside $\SO(H_p)$ \cite{HammerlSagerschnig, Nurowski}. While the construction $\mbD \rightsquigarrow \mbc_{\mbD}$ depends at each point on the $4$-jet of $\mbD$, the corresponding compatibility condition in the tractor setting is algebraic (pointwise), which reduces many of our considerations and arguments to properties of the algebra of $\G_2$.

With these facts in hand, it is immediate that whether an oriented conformal structure of signature $(2, 3)$ is both $(2, 3, 5)$ and almost Einstein is characterized by the admission of both a~compatible tractor $3$-form $\Phi$ and a (nonzero) parallel tractor $\bbS \in \Gamma(\mcV)$, which we may just as well frame as a reduction of the holonomy of $\nabla^{\mcV}$ to the $8$-dimensional common stabilizer~$S$ in~$\SO(3, 4)$ of a nonzero vector in the standard representation $\bbV$ of~$\SO(3, 4)$ and a $3$-form in $\Lambda^3 \bbV^*$ compatible with the conformal structure. The isomorphism type of $S$ depends on the causality type of $\bbS$: If the vector is spacelike, then $S \cong \SU(1, 2)$; if it is timelike, then $S \cong \SL(3, \bbR)$; if it is isotropic, then $S \cong \SL(2, \bbR) \ltimes Q_+$, where $Q_+ < \G_2$ is the connected, nilpotent subgroup of~$\G_2$ def\/ined via Sections~\ref{subsubsection:parabolic-geometry} and~\ref{subsubsection:235-distributions}.\footnote{Cf.~\cite[Corollary~2.4]{Kath}.}

\begin{propositionA}
An oriented conformal structure of signature $(2, 3)$ is both $(2, 3, 5)$ and almost Einstein iff it admits a holonomy reduction to the common stabilizer $S$ of a $3$-form in $\Lambda^3 \bbV^*$ compatible with the conformal structure and a nonzero vector in $\bbV$.
\end{propositionA}

Throughout this article, $S$, $\SU(1, 2)$ $\SL(3, \bbR)$, and $\SL(2, \bbR) \ltimes Q_+$ refer to the common stabilizer of the data described above~-- that is, to any subgroup in a particular conjugacy class in~$\SO(3, 4)$ (and not just to a subgroup of~$\SO(3, 4)$ of the respective isomorphism types).

Given an almost Einstein $(2, 3, 5)$ conformal structure, algebraically combining $\Phi$ and $\bbS$ yields other parallel tractor objects. The simplest of these is the contraction $\bbK := -\bbS \hook \Phi$, which we may identify with a skew endomorphism of the standard tractor bundle, $\mcV$. Underlying this endomorphism is a conformal Killing f\/ield $\xi$ of the induced conformal structure $\mbc_{\mbD}$ that does not preserve any distribution that induces that structure. Thus, if $\xi$ is complete (or alternatively, if we content ourselves with a suitable local statement) the images of $\mbD$ under the f\/low of $\xi$ comprise a $1$-parameter family of distinct $(2, 3, 5)$ distributions that all induce the same conformal structure. This suggests~-- and connects with the problem of existence of an almost Einstein scale~-- a natural question that we call the \textit{conformal isometry problem} for $(2, 3, 5)$ distributions: Given a~$(2, 3, 5)$ distribution~$(M, \mbD)$, what are the $(2, 3, 5)$ distributions~$\mbD'$ on~$M$ that induce the same conformal structure, that is, for which $\mbc_{\mbD'} = \mbc_{\mbD}$? Put another way, what are the f\/ibers of the map~$\mbD \rightsquigarrow \mbc_{\mbD}$? By our previous observation, working in the tractor setting essentially reduces this to an algebraic problem, which we resolve in Proposition~\ref{proposition:compatible-g2-structures} (and which extends to the split real form of~$\G_2$ an analogous result of Bryant \cite[Remark~4]{Bryant} for the compact real form). Translating this result to the tractor setting and then reinterpreting it in terms of the underlying data gives a complete description of all $(2, 3, 5)$ distributions $\mbD'$ that induce the conformal structure $\mbc_{\mbD}$. In order to formulate it, we note f\/irst that, given a f\/ixed oriented conformal structure $\mbc$ of signature $(2, 3)$, underlying any compatible parallel tractor $3$-form, and hence corresponding to a $(2, 3, 5)$ distribution $\mbD$, is a conformally weighted $2$-form $\phi \in \Gamma(\Lambda^2 T^*M \otimes \mcE[3])$, which in particular is a solution to the conformally invariant \textit{conformal Killing $2$-form equation}. The weighted $2$-forms that arise this way are called \textit{generic}. This solution turns out to satisfy $\phi \wedge \phi = 0$~-- so it is locally decomposable~-- but vanishes nowhere. Hence, it def\/ines an oriented $2$-plane distribution, and this distribution is $\mbD$.
\begin{theoremB}\label{theorem:conformally-isometric-distribution-parameterization}
Fix an oriented $(2, 3, 5)$ distribution $(M, \mbD)$, denote by $\phi$ the corresponding generic conformal Killing $2$-form, by $\Phi \in \Gamma(\Lambda^3 \mcV^*)$ the corresponding parallel tractor $3$-form, and by $H_{\Phi} \in \Gamma(S^2 \mcV^*)$ the parallel tractor metric associated to $\mbc_{\mbD}$.
\begin{enumerate}\itemsep=0pt
	\item[$1.$] Suppose $(M, \mbc_{\mbD})$ admits the nonzero almost Einstein scale $\sigma \in \Gamma(\mcE[1])$, and denote by $\bbS \in \Gamma(\mcV)$ the corresponding parallel tractor; by rescaling, we may assume that $\varepsilon := -H_{\Phi}(\bbS, \bbS) \in \{-1, 0, +1\}$. Then, for any $(\bar A, B) \in \bbR^2$ such that $-\varepsilon \bar A^2 + 2 \bar A + B^2 = 0$ $($there is a~$1$-parameter family of such pairs$)$ the weighted $2$-form
		\begin{gather}
			\phi'_{ab}
				:= \phi_{ab}
				+ \bar{A} \left[\tfrac{1}{5} \sigma^2 \left(\tfrac{1}{3} \phi_{ab, c}{}^c + \tfrac{2}{3} \phi_{c[a, b]}{}^c + \tfrac{1}{2} \phi_{c [a,}{}^c{}_{b]} + 4 \sfP^c{}_{[a} \phi_{b] c} \right) - \sigma \sigma^{,c} \phi_{[ca, b]} \right. \label{equation:family-conformal-Killing-2-forms}\\
 \left.\hphantom{\phi'_{ab}:= }{} - \tfrac{1}{2} \sigma \sigma_{, [a} \phi_{b] c,}{}^c - \tfrac{1}{5} \sigma \sigma_{,c}{}^c \phi_{ab} + 3 \sigma^{,c} \sigma_{,[c} \phi_{ab]} \right]
+ B [-\tfrac{1}{4} \sigma \phi^{cd,}{}_d \phi_{[ab, c]} + \tfrac{3}{4} \sigma^{,c} \phi_{[ab} \phi_{c]d,}{}^d]\nonumber
		\end{gather}
is a generic conformal Killing $2$-form, and the oriented $(2, 3, 5)$ distribution $\mbD'$ it determines induces the same oriented conformal structure that $\mbD$ does, that is $\mbc_{\mbD'} = \mbc_{\mbD}$.
	\item[$2.$] Conversely, all conformally isometric oriented $(2, 3, 5)$ distributions arise this way: If an oriented $(2, 3, 5)$ distribution $\mbD'$ satisfies $\mbc_{\mbD'} = \mbc_{\mbD}$ $($this condition is equality of \textit{oriented} conformal structures$)$, then there is an almost Einstein scale $\sigma$ of $\mbc_{\mbD}$ $($we may assume that the corresponding parallel tractor $\bbS$ satisfies $\varepsilon := -H_{\Phi}(\bbS, \bbS) \in \{-1, 0, +1\})$ and \mbox{$(\bar A, B) \in \bbR^2$} satisfying $-\varepsilon \bar A^2 + 2 \bar A + B^2 = 0$ such that the normal conformal Killing $2$-form $\phi'$ corresponding to $\mbD'$ is given by~\eqref{equation:family-conformal-Killing-2-forms}.
\end{enumerate}
\end{theoremB}

Herein, a comma $_,$ denotes the covariant derivative with respect to (any) representative $g \in \mbc_{\mbD}$, and $\sfP_{ab}$ denotes the Schouten tensor~\eqref{equation:definition-Schouten} of~$g$. We say that the distributions in the $1$-parameter family $\mcD$ determined by~$\mbD$ and $\sigma$ as in the theorem are \textit{related by $\sigma$}. Theorem~B is proved in Section~\ref{subsection:conformally-isometric-235-distributions}.

Dif\/ferent signs of the Einstein constant of the Einstein metric $\sigma^{-2} \mbg$, or equivalently, dif\/ferent causality types of the corresponding parallel tractor $\bbS$, determine families of distributions with dif\/ferent qualitative behaviors. Section \ref{subsubsection:parameterizations-isometric} gives simple parameterizations of the $1$-parameter families of conformally isometric distributions, and Section~\ref{subsubsection:recovering-Einstein-scale} gives an explicit algorithm for recovering an almost Einstein scale $\sigma$ of $\mbc_{\mbD}$ relating $\mbD$ and $\mbD'$ whose existence is guaranteed by Part~(2) of Theorem~B.

An immediate corollary of Theorem~B is a natural geometric characterization of almost Einstein oriented $(2, 3, 5)$ distributions:

\begin{theoremC}
The conformal structure $\mbc_{\mbD}$ induced by an oriented $(2, 3, 5)$ distribution $(M, \mbD)$ is almost Einstein iff there is a distribution $(M, \mbD')$, $\mbD' \neq \mbD$ such that $\mbc_{\mbD} = \mbc_{\mbD'}$.
\end{theoremC}

Now, f\/ix an oriented conformal structure $\mbc$ of signature $(2, 3)$ and a nonzero almost Einstein scale $\sigma$ of $\mbc$. The conformal Killing f\/ield $\xi$ of $\mbc$ determined together by $\sigma$ and a choice of distribution $\mbD$ in the $1$-parameter family $\mcD$ of oriented $(2, 3, 5)$ distributions inducing $\mbc$ and related by $\sigma$ turns out not to depend on the choice of~$\mbD$ (Proposition~\ref{proposition:conformal-Killing-field-distribution-family}), and we can ask for all of the geometric objects that (like $\xi$) are determined by~$\sigma$ and $\mcD$. The almost Einstein scale $\sigma$ alone partitions $M$ into three subsets, $M_+$, $\Sigma$, $M_-$, according to the sign $+$, $0$, $-$ of~$\sigma$ at each point. By construction, $\sigma$ determines an Einstein metric on the complement $M - \Sigma = M_+ \cup M_-$. If the Einstein metric determined by $\sigma$ is not Ricci-f\/lat (that is, if the parallel tractor $\bbS$ corresponding to $\sigma$ is nonisotropic), the boundary $\Sigma$ itself inherits a conformal structure $\mbc_{\Sigma}$ that is suitably compatible with and that can be regarded as a natural compactifying structure for $(M_{\pm}, g_{\pm})$ along $\partial M_{\pm} = \Sigma$~\cite{Gover}. Something similar but more involved occurs in the Ricci-f\/lat case.

This decomposition of $M$ according to the geometry of the object~$\sigma$~-- equivalently, the holonomy reduction of $\nabla^{\mcV}$ determined by the parallel standard tractor~$\bbS$~-- along with descriptions of the geometry induced on each subset in the decomposition, is formalized by the theory of \textit{curved orbit decompositions} \cite{CGH}. Here, the involved geometric structures are encoded as \textit{Cartan geometries} (Section~\ref{subsubsection:Cartan-geometry}) of an appropriate type, which are geometric structures modeled on appropriate homogeneous spaces $G / P$ endowed with $G$-invariant geometric structures, and the decomposition of $M$ in the presence of a holonomy reduction to a group $H \leq G$ is a natural generalization of the $H$-orbit decomposition of $G / P$; the subsets in the decomposition are accordingly termed the \textit{curved orbits} of the reduction. The curved orbits are parameterized by the intersections of $H$ and $P$ up to conjugacy in $G$, and $H$ together with these intersections determine the respective geometric structures on each curved orbit.

Section \ref{section:curved-orbit-decomposition} carries out this decomposition for the Cartan geometry canonically associated to $(M, \mbc)$ determined by $\sigma$ and $\mcD$, that is, by a holonomy reduction to the group $S$. Besides elucidating the geometry of almost Einstein $(2, 3, 5)$ conformal structures for its own sake, this serves three purposes: First, this documents an example of a curved orbit decomposition for which the decomposition is relatively involved. Second, and more importantly, we will see that several classical geometries occur in the curved orbit decompositions, establishing novel and nonobvious links between $(2, 3, 5)$ distributions and those structures. Third, we can then exploit these connections to give new methods for construction of almost Einstein $(2, 3, 5)$ conformal structures from classical geometries, and using these we produce, for the f\/irst time, examples both with negative (Example \ref{example:distinguished-rolling-distribution}) and positive (Example \ref{example:Dirichlet-Ricci-positive}) Einstein constants.

Dif\/ferent signs of the Einstein constant (equivalently, dif\/ferent causality types of the parallel tractor $\bbS$ corresponding to $\sigma$) lead to qualitatively dif\/ferent curved orbit decompositions, so we treat them separately. We say that an almost Einstein scale is \textit{Ricci-negative}, \textit{-positive}, or \textit{-flat} if the Einstein constant of the Einstein metric it determines is negative, positive, or zero, respectively. See also Appendix~\ref{appendix}, which summarizes the results here and records geometric characterizations of the curved orbits.

In the Ricci-negative case, the decomposition of a manifold into submanifolds is the same as that determined by $\sigma$ alone, but the family $\mcD$ determines additional structure on each closed orbit. (Herein, for readability we often suppress notation denoting restriction to a curved orbit.)
\begin{theoremDminus}\label{theorem:Dminus}
Let $(M, \mbc)$ be an oriented conformal structure of signature $(2, 3)$. A holonomy reduction of $\mbc$ to $\SU(1, 2)$ determines a $1$-parameter family $\mcD$ of oriented $(2, 3, 5)$ distributions related by a Ricci-negative almost Einstein scale such that $\mbc = \mbc_{\mbD}$ for all $\mbD \in \mcD$, as well as a~decomposition $M = M_5^+ \cup M_5^- \cup M_4$:
\begin{itemize}\itemsep=0pt
	\item $($Section~{\rm \ref{subsection:open-curved-orbits})} The orbits $M_5^{\pm}$ are open, and $M_5 := M_5^+ \cup M_5^-$ is equipped with a Ricci-negative Einstein metric $g := \sigma^{-2} \mbg\vert_{M_5}$. The pair $(-g, \xi)$ is a Sasaki structure $($see Section~{\rm \ref{subsubsection:vareps-Sasaki-structure})} on $M_5$. Locally, $M_5$ fibers along the integral curves of $\xi$, and the leaf space $L_4$ inherits a K\"{a}hler--Einstein structure $(\hatg, \hatK)$.
	\item $($Section~{\rm \ref{subsubsection:curved-orbit-negative-hypersurface})} The orbit $M_4$ is a smooth hypersurface, and inherits a Fefferman conformal structure $\mbc_{\mbS}$, which has signature $(1, 3)$: Locally, $\mbc_{\mbS}$ arises from the classical Fefferman construction {\rm \cite{CapGoverHolonomyCharacterization,Fefferman,LeitnerHolonomyCharacterization}}, which $($in this dimension$)$ canonically associates to any $3$-dimensional CR structure $(L_3, \mbH, \mbJ)$ a conformal structure on a circle bundle over $L_3$. Again in the local setting, the fibers of the fibration $M_4 \to L_3$ are the integral curves of~$\xi$.
\end{itemize}
\end{theoremDminus}

The Ricci-positive case is similar to the Ricci-negative case but entails $2$-dimensional curved orbits that have no analogue there.
\begin{theoremDplus}
Let $(M, \mbc)$ be an oriented conformal structure of signature $(2, 3)$. A holonomy reduction of $\mbc$ to $\SL(3, \bbR)$ determines a $1$-parameter family $\mcD$ of oriented $(2, 3, 5)$ distributions related by a Ricci-positive almost Einstein scale such that $\mbc = \mbc_{\mbD}$ for all $\mbD \in \mcD$, as well as a~decomposition $M = M_5^+ \cup M_5^- \cup M_4 \cup M_2^+ \cup M_2^-$:
\begin{itemize}\itemsep=0pt
	\item $($Section~{\rm \ref{subsection:open-curved-orbits})} The orbits $M_5^{\pm}$ are open, and $M_5 := M_5^+ \cup M_5^-$ is equipped with a Ricci-positive Einstein metric $g := \sigma^{-2} \mbg\vert_{M_5}$. The pair $(-g, \xi)$ is a para-Sasaki structure $($see Section~{\rm \ref{subsubsection:vareps-Sasaki-structure})} on~$M_5$. Locally, $M_5$ fibers along the integral curves of~$\xi$, and the leaf space~$L_4$ inherits a~para-K\"{a}hler--Einstein structure $(\hatg, \hatK)$.
	\item $($Section~{\rm \ref{subsubsection:curved-orbit-positive-hypersurface})} The orbit $M_4$ is a smooth hypersurface, and inherits a para-Fefferman conformal structure $\mbc_{\mbS}$, which has signature $(2, 2)$: Locally, $\mbc_{\mbS}$ arises from the paracomplex analogue of the classical Fefferman construction, which (in this dimension) canonically associates to any Legendrean contact structure $(L_3, \mbH_+ \oplus \mbH_-)$~-- or, locally equivalently, a point equivalence class of second-order ODEs $\ddot y = F(x, y, \dot y)$~-- a~conformal structure on a $\SO(1, 1)$-bundle over $L_3$. Again in the local setting, the fibers of the fibration $M_4 \to L_3$ are the integral curves of~$\xi$.
	\item $($Section~{\rm \ref{subsubsection:curved-orbit-M2pm})} The orbits $M_2^{\pm}$ are $2$-dimensional and inherit oriented projective structures.
\end{itemize}
\end{theoremDplus}

The descriptions of the geometric structures in the above two cases are complete in the sense that any other geometric data determined by the holonomy reduction to $S$ can be recovered from the indicated data. We do not claim the same for the descriptions in the Ricci-f\/lat case, which we can view as a sort of degenerate analogue of the other two cases.

\begin{theoremDzero}
Let $(M, \mbc)$ be an oriented conformal structure of signature $(2, 3)$. A holonomy reduction of $\mbc$ to $\SL(2, \bbR) \ltimes Q_+$ determines a $1$-parameter family $\mcD$ of oriented $(2, 3, 5)$ distributions related by a Ricci-flat almost Einstein scale such that $\mbc = \mbc_{\mbD}$ for all $\mbD \in \mcD$, as well as a decomposition $M = M_5^+ \cup M_5^- \cup M_4 \cup M_2 \cup M_0^+ \cup M_0^-$.
\begin{itemize}\itemsep=0pt
	\item $($Section~{\rm \ref{subsection:open-curved-orbits})} The orbits $M_5^{\pm}$ are open, and $M_5 := M_5^+ \cup M_5^-$ is equipped with a Ricci-flat metric $g := \sigma^{-2} \mbg\vert_{M_5}$. The pair $(-g, \xi)$ is a null-Sasaki structure $($see Section~{\rm \ref{subsubsection:vareps-Sasaki-structure})} on~$M_5$. Locally, $M_5$ fibers along the integral curves of $\xi$, and the leaf space $L_4$ inherits a~null-K\"{a}hler--Einstein structure $(\hatg, \hatK)$.
	\item $($Section~{\rm \ref{subsubsection:curved-orbit-flat-hypersurface})} The orbit $M_4$ is a smooth hypersurface, and it locally fibers over a $3$-mani\-fold $\wt L$ that carries a conformal structure $\mbc_{\wt L}$ of signature $(1, 2)$ and isotropic line field.
	\item $($Section~{\rm \ref{subsubsection:curved-orbit-M2})} The orbit $M_2$ has dimension $2$ and is equipped with a preferred line field.
	\item $($--$)$ The orbits $M_0^{\pm}$ consist of isolated points and so carry no intrinsic geometry.
\end{itemize}
\end{theoremDzero}

The statements of these theorems involve an expository choice that entails some subtle consequences: In each case, the holonomy reduction determines an almost Einstein scale $\sigma$ and $1$-parameter family $\mcD$ of oriented $(2, 3, 5)$ distributions, but this reduction does not distinguish a distribution within this family. Alternatively, we could specify for $\mbc$ an almost Einstein scale and a distribution $\mbD$ such that $\mbc = \mbc_{\mbD}$. Such a specif\/ication determines a holonomy reduction to $S$ as above, but the choice of a preferred $\mbD$ is additional data, and this is ref\/lected in the induced geometries on the curved orbits.

Proposition \ref{proposition:Einstein-Sasaki-to-235} gives a partial converse to the statements in Theorems $D_-$ and $D_+$ about the $\varepsilon$-Sasaki--Einstein structures $(-g, \xi)$ induced on the open orbits by the corresponding holonomy reductions: Any $\varepsilon$-Sasaki--Einstein structure $(-g, \xi)$ (here restricting to $\varepsilon = \pm 1$) determines around each point a $1$-parameter family of oriented $(2, 3, 5)$ distributions related by the almost Einstein scale for $[g]$ corresponding to $g$, and by construction the $\varepsilon$-Sasaki structure is the one induced by the corresponding holonomy reduction.

We also brief\/ly present a generalized Fef\/ferman construction that essentially inverts the projection $M_5 \to L_4$ along the leaf space f\/ibration in the non-Ricci-f\/lat cases, in a way that emphasizes the role of almost Einstein $(2, 3, 5)$ conformal structures (see Section~\ref{subsubsection:twistor-construction}). In particular, any non-Ricci-f\/lat $\varepsilon$-K\"{a}hler--Einstein metric of signature $(2, 2)$ gives rise to a $1$-parameter family of $(2, 3, 5)$ distributions. We treat this construction in more detail in an article currently in preparation \cite{SagerschnigWillseTwistor}.

As mentioned above we have for convenience formulated our results for oriented $(2, 3, 5)$ distributions and conformal structures, but all the results herein have analogues for unoriented distributions and many of our considerations are anyway local. Alternatively one could further restrict attention to space- and time-oriented conformal structures (see Remark \ref{remark:space-and-time-oriented}) or work with conformal spin structures, the latter of which would connect the considerations here more closely with those in \cite{HammerlSagerschnigSpinor}.

Finally, we mention brief\/ly one aspect of this geometry we do not discuss here, but which will be taken up in a shorter article currently in preparation: One can construct for any oriented $(2, 3, 5)$ distribution~$\mbD$ an invariant second-order linear dif\/ferential operator that acts on sections of $\mcE[1] \cong \Lambda^2 \mbD$ closely related to $\Theta_0^{\mcV}$~\cite{SagerschnigWillseOperator}. Its kernel can again be interpreted as the space of almost Einstein scales of $\mbc_{\mbD}$, but it is a simpler object than $\Theta_0^{\mcV}$, enough so that one can use it to construct new explicit examples of almost Einstein $(2, 3, 5)$ distributions. Among other things, the existence of such an operator emphasizes that almost Einstein geometry of the induced conformal structure is a fundamental feature of the geometry of $(2, 3, 5)$ distributions.

For simplicity of statements of results, we assume that all given manifolds are connected; we do not include this hypothesis explicitly in our statements of results.

We use both index-free and Penrose index notation throughout, according to convenience.

\section{Preliminaries}\label{section:preliminaries}

\subsection[epsilon-complex structures]{$\boldsymbol{\varepsilon}$-complex structures}

The \textit{$\varepsilon$-complex numbers}, $\varepsilon \in \{-1, 0, +1\}$, is the ring $\bbC_{\varepsilon}$ generated over $\bbR$ by the generator $i_{\varepsilon}$, which satisf\/ies precisely the relations generated by $i_{\varepsilon}^2 = \varepsilon$. An \textit{$\varepsilon$-complex structure} on a~real vector space $\bbW$ (necessarily of even dimension, say, $2m$) is an endomorphism $\bbK \in \End(\bbW)$ such that $\bbK^2 = \varepsilon \id_{\bbW}$;\footnote{In the case $\varepsilon = 0$, some references require additionally that a null-complex structure satisfy $\rank \bbK = m$~\cite{DunajskiPrzanowski} this anyway holds for the null-complex structures that appear in this article.}; if $\varepsilon = +1$, we further require that the $(\pm 1)$-eigenspaces both have dimension~$m$. This identif\/ies $\bbW$ with $\bbC_{\varepsilon}^m$ (as a free $\bbC_{\varepsilon}$-module) so that the action of $\bbK$ coincides with multiplication by $i_{\varepsilon}$, and the pair $(\bbW, \bbK)$ is an \textit{$\varepsilon$-complex vector space}.

One specializes the names of structures to particular values of $\varepsilon$ by omitting $(-1)$-, replacing $(+1)$- with the pref\/ix \textit{para-}, and replacing $0$- with the modif\/ier \textit{null}-. See \cite[Section~1]{SchulteHengesbach}.

\subsection[The group G2]{The group $\boldsymbol{\G_2}$}

\subsubsection{Split cross products in dimension 7}

The geometry studied in this article depends critically on the algebraic features of a so-called (split) cross product $\times$ on a $7$-dimensional real vector space $\bbV$. One can realize this explicitly using the algebra $\wtbbO$ of split octonions; we follow \cite[Section~2]{GPW}, and see also~\cite{Sagerschnig}. This is a~composition ($\bbR$-)algebra and so is equipped with a unit $1$ and a nondegenerate quadratic form $N$ multiplicative in the sense that $N(xy) = N(x) N(y)$ for all $x, y \in \wtbbO$. In particular $N(1) = 1$, and polarizing $N$ yields a nondegenerate symmetric bilinear form, which turns out for~$\wtbbO$ to have signature $(4, 4)$. So, the $7$-dimensional vector subspace $\bbI = \langle 1 \rangle^{\perp}$ of \textit{imaginary split octonions} inherits a nondegenerate symmetric bilinear form $H$ of signature $(3, 4)$, as well as a~map $\times\colon \bbI \times \bbI \to \bbI$ def\/ined by
\begin{gather*}
	x \times y := xy + H(x, y) 1 ;
\end{gather*}
this is just the orthogonal projection of $xy$ onto $\bbI$. This map is a (binary) cross product in the sense of \cite{BrownGray}, that is, it satisf\/ies
\begin{align}\label{equation:cross-product-totally-skew}
	H(x \times y, x) = 0
	\qquad \textrm{and} \qquad
	H(x \times y, x \times y) = H(x, x) H(y, y) - H(x, y)^2
\end{align}
for all $x, y \in \bbI$.

\begin{Definition}
We say that a bilinear map $\times\colon \bbV \times \bbV \to \bbV$ on a $7$-dimensional real vector space $\bbV$ is a \textit{split cross product} if\/f there is a linear isomorphism $A \colon \bbI \to \bbV$ such that $A(x \times y) = A(x) \times A(y)$.
\end{Definition}

A split cross product $\times$ determines a bilinear form
\begin{gather*}
	H_{\times}(x, y) := -\tfrac{1}{6} \tr(x \times (y \times \,\cdot\,)) ;
\end{gather*}
of signature $(3, 4)$ on the underlying vector space. For the split cross product $\times$ on $\bbI$, $H_{\times} = H$. We say that a cross product $\times$ is \textit{compatible} with a bilinear form if\/f$\times$ induces $H$, that is, if\/f \mbox{$H = H_{\times}$}. It follows from the alternativity identity $(x x) y = x (x y)$ satisf\/ied by the split octonions that
\begin{gather}\label{equation:iterated-cross-product}
	x \times (x \times y) := -H_{\times}(x, x) y + H_{\times}(x, y) x .
\end{gather}

By \eqref{equation:cross-product-totally-skew}, $\times$ is totally $H_{\times}$-skew, so lowering its upper index with $H_{\times}$ yields a $3$-form $\Phi \in \Lambda^3 \bbV^*$:
\begin{gather*}
	\Phi(x, y, z) := H_{\times}(x \times y, z) .
\end{gather*}
A $3$-form is said to be \textit{split-generic} if\/f it arises this way, and such forms comprise an open $\GL(\bbV)$-orbit under the standard action on $\Lambda^3 \bbV^*$. One can recover from any split-generic $3$-form the split cross product $\times$ that induces it.

A split cross product $\times$ also determines a nonzero volume form $\epsilon_{\times} \in \Lambda^7 \bbV^*$ for $H_{\times}$:
\begin{gather*}
	(\epsilon_{\times})_{ABCDEFG} := \tfrac{1}{42} {\Phi}_{K[AB} {\Phi}^K{}_{CD} {\Phi}_{EFG]} .
\end{gather*}
Thus, $\times$ determines an orientation $[\epsilon_{\times}]$ on $\bbV$ and Hodge star operators $\ast\colon \Lambda^k \bbV^* \to \Lambda^{7 - k} \bbV^*$.

If $\bbV$ is a vector space endowed with a bilinear form $H$ of signature $(p, q)$ and an orienta\-tion~$\Omega$, the subgroup of $\GL(\bbV)$ preserving the pair $(H, \Omega)$ is $\SO(p, q)$, so we refer to such a~pair as a~$\SO(p, q)$-structure on $\bbV$. We say that a cross product $\times$ on $\bbV$ is \textit{compatible} with an $\SO(p, q)$-structure $(H, \Omega)$ if\/f $\times$ induces $H$ and $\Omega$, that is, if\/f $H = H_{\times}$ and $\Omega = [\epsilon_{\times}]$.

A split-generic $3$-form $\Phi$ satisf\/ies various contraction identities, including \cite[equations~(2.8) and~(2.9)]{Bryant}:
\begin{gather}
	\Phi^E{}_{AB} \Phi_{ECD} = (\astPhi \Phi)_{ABCD} + H_{AC} H_{BD} - H_{AD} H_{BC}, \label{equation:contraction-Phi-Phi} \\
	\Phi^F{}_{AB} (\astPhi \Phi)_{FCDE}	= 3 (H_{A[C} \Phi_{DE]B} - H_{B[C} \Phi_{DE]A}). \label{equation:contraction-Phi-PhiStar}
\end{gather}

\subsubsection[The group G2]{The group $\boldsymbol{\G_2}$}\label{subsubsection:G2}

The (algebra) automorphism group of $\wtbbO$ is a connected, split real form of the complex Lie group of type $\G_2$, and so we denote it by $\G_2$. One can recover the algebra structure of $\wtbbO$ from $(\bbI, \times)$, so $\G_2$ is also the automorphism group of $\times$, and equivalently, the stabilizer subgroup in $\GL(\bbV)$ of a split-generic $3$-form on a $7$-dimensional real vector space $\bbV$. For much more about~$\G_2$, see~\cite{Kath}.

The action of $\G_2$ on $\bbV$ def\/ines the smallest nontrivial, irreducible representation of $\G_2$, which is sometimes called the \textit{standard representation}. This action stabilizes a unique split cross product, or equivalently, a unique split-generic $3$-form (up to a positive multiplicative constant). Thus, by a \textit{$\G_2$-structure} on a $7$-dimensional real vector space $\bbV$ we mean either (1)~a~representation of $\G_2$ on $\bbV$ isomorphic to the standard one, or, (2)~slightly abusively (on account of the above multiplicative ambiguity), a~split cross product $\times$ on $\bbV$, or equivalently, a~split-generic $3$-form $\Phi$ on $\bbV$.

Since a split cross product $\times$ on a vector space $\bbV$ determines algebraically both a bilinear form~$H_{\times}$ and an orientation $[\epsilon_{\times}]$ on $\bbV$, the induced actions of $\G_2$ preserve both, def\/ining a~natural embedding
\begin{gather*}
	\G_2 \hookrightarrow \SO(H_{\times}) \cong \SO(3, 4) .
\end{gather*}
Moreover, $H_{\times}$ realizes $\bbV$ as the standard representation of $\SO(3, 4)$, and its restriction to $\G_2$ is the standard representation $\times$ def\/ines. Like $\SO(3, 4)$, the $\G_2$-action on the ray projectivization of $\bbV$ has exactly three orbits, namely the sets of spacelike, isotropic, and timelike rays \cite[Theorem~3.1]{Wolf}.

It is convenient for our purposes to use the $\G_2$-structure on $\bbV$ def\/ined in a basis $(E_a)$ (with dual basis, say, $(e^a)$) via the $3$-form
\begin{gather}\label{equation:3-form-basis}
	\Phi := - e^{147} + \sqrt{2} e^{156} + \sqrt{2} e^{237} + e^{245} + e^{346},
\end{gather}
where $e^{a_1 \cdots a_k} := e^{a_1} \wedge \cdots \wedge e^{a_k}$ (cf.~\cite[equation~(23)]{HammerlSagerschnig}). With respect to the basis~$(E_a)$, the induced bilinear form $H_{\Phi}$ has matrix representation
\begin{gather}\label{equation:bilinear-form-basis}
	[H_{\Phi}] =
	\begin{pmatrix}
		0 & 0 & 0 & 0 & 1 \\
		0 & 0 & 0 & I_2 & 0 \\
		0 & 0 & -1 & 0 & 0 \\
		0 & I_2 & 0 & 0 & 0 \\
		1 & 0 & 0 & 0 & 0
	\end{pmatrix} ,
\end{gather}
where $I_2$ denotes the $2 \times 2$ identity matrix, and the induced volume form is
\begin{gather*}
	\epsilon_{\Phi} = -e^{1234567}.
\end{gather*}

Let $\mfg_2$ denote the Lie algebra of $\G_2$. Dif\/ferentiating the inclusion $\G_2 \hookrightarrow \GL(\bbV)$ yields a Lie algebra representation $\mfg_2 \hookrightarrow \mfgl(\bbV) \cong \End(\bbV)$, and with respect to the basis $(E_a)$ its elements are precisely those of the form
\begin{gather}\label{equation:g2-matrix-representation}
	\begin{pmatrix}
		\tr A & Z & s & W^{\top} & 0 \\
	 	 X & A & \sqrt{2} J Z^{\top} & \frac{s}{\sqrt{2}} J & -W \\
	 	 r & -\sqrt{2} X^{\top} J & 0 & -\sqrt{2} Z J & s \\
	 	 Y^{\top} & -\frac{r}{\sqrt{2}} J & \sqrt{2} J X & -A^{\top} & -Z^{\top} \\
	 	 0 & -Y & r & -X^{\top} & -\tr A
	\end{pmatrix} ,
\end{gather}
where $A \in \mfgl(2, \bbR)$, $W, X \in \bbR^2$, $Y, Z \in (\bbR^2)^*$, $r, s \in \bbR$, and $J := \begin{pmatrix}0 & -1\\1 & 0 \end{pmatrix}$.

\subsubsection[Some G2 representation theory]{Some $\boldsymbol{\G_2}$ representation theory}\label{subsubsection:g2-representation-theory}

Fix a $7$-dimensional real vector space $\bbV$ and a $\G_2$-structure $\Phi \in \Lambda^3 \bbV^*$. We brief\/ly record the decompositions of the $\G_2$-representations $\Lambda^2 \bbV^*$, $\Lambda^3 \bbV^*$, and $S^2 \bbV^*$ into irreducible subrepresentations; we use and extend the notation of \cite[Section~2.6]{Bryant}. In each case, $\Lambda^k_l$ denotes the irreducible subrepresentation of $\Lambda^k \bbV^*$ of dimension~$l$, which is unique up to isomorphism.

The representation $\Lambda^2 \bbV^* \cong \mfso(H_{\Phi})$ decomposes into irreducible subrepresentations as
\begin{gather*}
	\Lambda^2 \bbV^* = \Lambda^2_7 \oplus \Lambda^2_{14} ;
\end{gather*}
\looseness=-1 $\Lambda^2_7 \cong \bbV$ and $\Lambda^2_{14}$ is isomorphic to the adjoint representation $\mfg_2$. Def\/ine the map $\iota^2_7\colon \bbV \to \Lambda^2 \bbV^*$~by
\begin{gather}\label{equation:iota-2-7}
	\iota^2_7 \colon \ \bbS^A \mapsto \bbS^C \Phi_{CAB} ;
\end{gather}
by raising an index we can view $\iota^2_7$ as the map $\bbV \to \mfso(H_{\Phi})$ given by $\bbS \mapsto (\bbT \mapsto \bbT \times \bbS)$. It is evidently nontrivial, so by Schur's lemma its (isomorphic) image is $\Lambda^2_7$.

Conversely, consider the map $\pi^2_7 \colon \Lambda^2 \bbV^* \to \bbV$ def\/ined by
\begin{gather}\label{equation:pi-2-7}
	\pi^2_7 \colon \ \bbA_{AB} \mapsto \tfrac{1}{6} \bbA_{BC} \Phi^{BCA} .
\end{gather}
Raising indices gives a map $\Lambda^2 \bbV \to \bbV$, which up to the multiplicative constant, is the descent of $\times\colon \bbV \times \bbV \to \bbV$ via the wedge product. In particular it is nontrivial, so it has kernel $\Lambda^2_{14}$ and restricts to an isomorphism $\smash{\pi^2_7\vert_{\Lambda^2_7}\colon \Lambda^2_7 \to \bbV}$; we have chosen the coef\/f\/icient so that $\pi^2_7 \circ \iota^2_7 = \id_{\bbV}$ and $\smash{\iota^2_7 \circ \pi^2_7 \vert_{\Lambda^2_7} = \id_{\Lambda^2_7}}$. Since $\G_2$ is the stabilizer subgroup in $\SO(H_{\Phi})$ of $\Phi$, $\mfg_2$ is the annihilator in $\mfso(H_{\Phi}) \cong \Lambda^2 \bbV^*$ of $\Phi$. Expanding $\id_{\bbV} - \iota^2_7 \circ \pi^2_7$ using \eqref{equation:iota-2-7} and \eqref{equation:pi-2-7} and applying \eqref{equation:contraction-Phi-Phi} gives that under this identif\/ication, the corresponding projection $\pi^2_{14} \colon \mfso(\bbV) \cong \Lambda^2 \bbV^* \to \mfg_2$ is
\begin{gather*}
	\pi^2_{14} \colon \ \bbA^A{}_B \mapsto \tfrac{2}{3} \bbA^A{}_B - \tfrac{1}{6} (\astPhi \Phi)_D{}^{EA}{}_B \bbA^D{}_E .
\end{gather*}

The representation $\Lambda^3 \bbV^*$ decomposes into irreducible subrepresentations as
\begin{gather*}
	\Lambda^3 \bbV^* = \Lambda^3_1 \oplus \Lambda^3_7 \oplus \Lambda^3_{27} .
\end{gather*}
Here, $\Lambda^3_1$ is just the trivial representation spanned by $\Phi$, and the map $\pi^3_1 \colon \Lambda^3 \bbV^* \to \bbR$ def\/ined by
\begin{gather}\label{equation:pi-3-1}
	\pi^3_1\colon \ \Psi_{ABC} \mapsto \tfrac{1}{42} \Phi^{ABC} \Psi_{ABC}
\end{gather}
is a left inverse for the map $\bbR \stackrel{\cong}{\to} \Lambda^3_1 \hookrightarrow \Lambda^3 \bbV^*$ def\/ined by $a \mapsto a \Phi$.

The map $\iota^3_7 \colon \bbV \to \Lambda^3 \bbV^*$ def\/ined by
\begin{gather*}
	\iota^3_7 \colon \ \bbS^A \mapsto -\bbS^D (\astPhi \Phi)_{DABC} = [\ast_{\Phi}(\bbS \wedge \Phi)]_{ABC} .
\end{gather*}
is nonzero, so it def\/ines an isomorphism $\bbV \cong \Lambda^3_7$. The map $\pi^3_7 \colon \Lambda^3 \bbV^* \to \bbV$ def\/ined by
\begin{gather}\label{equation:pi-3-7}
	\pi^3_7 \colon \ \Psi_{ABC} \mapsto
		 \tfrac{1}{24} (\astPhi \Phi)^{BCDA} \Psi_{BCD}
		= \tfrac{1}{ 4} [\astPhi (\Phi \wedge \Psi)]^A
\end{gather}
is scaled so that $\pi^3_7 \circ \iota^3_7 = \id_{\bbV}$ and $\iota^3_7 \circ \pi^3_7 = \id_{\Lambda^3_7}$.

The $\G_2$-representation $S^2 \bbV^*$ decomposes into irreducible modules as $\bbR \oplus S^2_{\circ} \bbV^*$, namely, into its $H_{\Phi}$-trace and $H_{\Phi}$-tracefree components, respectively. The linear map $i\colon S^2 \bbV^* \to \Lambda^3 \bbV^*$ def\/ined by
\begin{gather}\label{equation:i}
	i \colon \ \bbA_{AB} \mapsto 6 \Phi^D{}_{[AB} \bbA_{C] D} .
\end{gather}
satisf\/ies $i(H_{\Phi}) = 6 \Phi$ and is nonzero on $S^2_{\circ} \bbV^*$, so $i \vert_{S^2_{\circ}}$ is an isomorphism $\smash{S^2_{\circ} \stackrel{\cong}{\to} \Lambda^3_{27}}$.

The projection $\pi^3_{27} \colon \Lambda^3 \bbV^* \to S^2_{\circ} \bbV^*$
\begin{gather}\label{equation:pi-3-27}
	\pi^3_{27} \colon \ \Psi_{ABC} \mapsto -\tfrac{1}{8} \astPhi[(\,\cdot \, \hook \Phi) \wedge (\,\cdot \, \hook \Phi) \wedge \Psi]_{AB} - \tfrac{3}{4} \pi^3_1(\Psi) H_{AB}
\end{gather}
is scaled so that $\pi^3_{27} \circ i \vert_{S^2_{\circ} \bbV^*} = \id_{S^2_{\circ} \bbV^*}$ and $i \circ \pi^3_{27}\vert_{\Lambda^3_{27}} = \id_{\Lambda^3_{27}}$.

\subsection{Cartan and parabolic geometry}

\subsubsection{Cartan geometry}
\label{subsubsection:Cartan-geometry}

In this subsubsection we follow \cite{CapSlovak, Sharpe}. Given a $P$-principal bundle $\pi \colon \mcG \to M$, we denote the (right) action of $\mcG \times P \to \mcG$ by $R^p(u) = u \cdot p$ for $u \in \mcG, p \in P$. For each $V \in \mfp$, the corresponding \textit{fundamental vector field} $\eta_V \in \Gamma(T\mcG)$ is $(\eta_V)_u := \partial_t \vert_0 [u \cdot \exp (tV)]$.

{\samepage \begin{Definition} For a Lie group $G$ and a closed subgroup $P$ (with respective Lie algebras~$\mfg$ and~$\mfp$), a \textit{Cartan geometry} of type $(G, P)$ on a manifold $M$ is a pair $(\mcG \to M, \omega)$, where $\mcG \to M$ is a~$P$-principal bundle and $\omega$ is a \textit{Cartan connection}, that is, a~section of $T^*\mcG \otimes \mfg$ satisfying
\begin{enumerate}\itemsep=0pt
	\item[1)] (right equivariance) $\omega_{u \cdot p}(T_u R^p \cdot \eta) = \Ad(p^{-1})(\omega_u(\eta))$ for all $u \in \mcG$, $p \in P$, $\eta \in T_u \mcG$,
	\item[2)] (reproduction of fundamental vector f\/ields) $\omega(\eta_V) = V$ for all $V \in \mfp$, and
	\item[3)] (absolute parallelism) $\omega_u \colon T_u \mcG \to \mfg$ is an isomorphism for all $u \in \mcG$.
\end{enumerate}
\end{Definition}
}

The \textit{(flat) model} of Cartan geometry of type $(G, P)$ is the pair $(G \to G / P, \omega_{\textrm{MC}})$, where $\omega_{\textrm{MC}}$ is the Maurer--Cartan form on $G$ def\/ined by $(\omega_{\textrm{MC}})_u := T_u L_{u^{-1}}$ (here $L_{u^{-1}} \colon G \to G$ denotes left multiplication by $u^{-1}$). This form satisf\/ies the identity $d\omega_{\textrm{MC}} + \tfrac{1}{2} [\omega_{\textrm{MC}}, \omega_{\textrm{MC}}] = 0$. We def\/ine the \textit{curvature $($form$)$} of a Cartan geometry $(\mcG, \omega)$ to be the section $\Omega := d\omega + \tfrac{1}{2} [\omega, \omega] \in \Gamma(\Lambda^2 T^* \mcG \otimes \mfg),$ and say that $(\mcG, \omega)$ is \textit{flat} if\/f $\Omega = 0$. This is the case if\/f around any point $u \in \mcG$ there is a local bundle isomorphism between $G$ and $\mcG$ that pulls back $\omega$ to $\omega_{\textrm{MC}}$.

One can show that the curvature $\Omega$ of any Cartan geometry $(\mcG \to M, \omega)$ is horizontal (it is annihilated by vertical vector f\/ields), so invoking the absolute parallelism and passing to the quotient def\/ines an equivalent object $\kappa \colon \mcG \to \Lambda^2 (\mfg / \mfp)^* \otimes \mfg$, which we also call the \textit{curvature}.

\subsubsection{Holonomy}\label{subsubsection:holonomy}

Given any Cartan geometry $(\mcG \to M, \omega)$ of type $(G, P)$, we can extend the Cartan connection $\omega$ to a unique principal connection $\smash{\hat\omega}$ on $\smash{\hat\mcG := \mcG \times_P G}$ characterized by (1) $G$-equivariance and (2) $\smash{\iota^* \hat\omega = \omega}$, where $\smash{\iota\colon \mcG \hookrightarrow \hat\mcG}$ is the natural inclusion $u \mapsto [u, e]$. Then, to any point $\smash{\hat u \in \hat\mcG}$ we can associate the holonomy group $\smash{\Hol_{\hat u}(\hat\omega) \leq G}$. Dif\/ferent choices of $\smash{\hat u}$ lead to conjugate subgroups of~$G$, so the conjugacy class $\smash{\Hol(\hat\omega)}$ thereof is independent of $\smash{\hat u}$, and we def\/ine the \textit{holonomy} of~$\omega$ (or just as well, of $(\mcG \to M, \omega)$) to be this class.

\subsubsection{Tractor geometry}\label{subsubsection:tractor-geometry}

Fix a pair $(G, P)$ as in Section~\ref{subsubsection:Cartan-geometry}, denote the Lie algebra of $G$ by $\mfg$, and f\/ix a $G$-representa\-tion~$\bbU$. Then, for any Cartan geometry $(\pi \colon \mcG \to M, \omega)$ of type $(G, P)$, we can form the associated \textit{tractor bundle} $\mcU := \mcG \times_P \bbU \to M$, which we can also view as the associated bundle $\smash{\hat\mcG \times_G \bbU \to M}$. Then, the principal connection $\smash{\hat\omega}$ on $\smash{\hat\mcG}$ determined by $\omega$ induces a vector bundle connection~$\nabla^{\mcU}$ on~$\mcU$.

Of distinguished importance is the \textit{adjoint tractor bundle} $\mcA := \mcG \times_P \mfg$. The canonical map $\Pi_0^{\mcA} \colon \mcA \to TM$ def\/ined by $(u, V) \mapsto T_u \pi \cdot \omega_u^{-1}(V)$ descends to a natural isomorphism $\smash{\mcG \times_P (\mfg / \mfp) \stackrel{\cong}{\to} TM}$, and via this identif\/ication $\Pi_0^{\mcA}$ is the bundle map associated to the canonical projection $\mfg \to \mfg / \mfp$.

Since the curvature $\Omega$ of $(\mcG, \omega)$ is also $P$-equivariant, we may regard it as a section $K \in \Gamma(\Lambda^2 T^*M \otimes \mcA)$, and again we call it the \textit{curvature} of $(\mcG \to M, \omega)$.

\subsubsection{Parabolic geometry}\label{subsubsection:parabolic-geometry}

In this article we will mostly (but not exclusively) work with geometries that can be realized as a special class of Cartan geometries that enjoy additional properties, most importantly suitable normalization conditions on $\omega$ that guarantee (subject to a usually satisf\/ied cohomological condition) a correspondence between Cartan geometries satisfying those conditions and geometric structures on the underlying manifold. We say that a Cartan geometry $(\mcG \to M, \omega)$ of type $(G, P)$ is a \textit{parabolic geometry} if\/f $G$ is semisimple and $P$ is a parabolic subgroup. For a detailed survey of parabolic geometry, including details of the below, see the standard reference \cite{CapSlovak}.

Recall that a parabolic subgroup $P < G$ determines a so-called $|k|$-grading on the Lie algebra~$\mfg$ of $G$: This is a vector space decomposition $\mfg = \mfg_{-k} \oplus \cdots \oplus \mfg_{+k}$ compatible with the Lie bracket in the sense that $[\mfg_a, \mfg_b] \subseteq \mfg_{a + b}$ and minimal in the sense that none of the summands~$\mfg_a$, $a = -k, \ldots, k$, is zero. The grading induces a $P$-invariant f\/iltration~$(\mfg^a)$ of~$\mfg$, where $\mfg^a := \mfg_a \oplus \cdots \oplus \mfg_{+k}$. In particular, $\mfp = \mfg^0 = \mfg_0 \oplus \cdots \oplus \mfg_{+k}$. We denote by $G_0 < P$ the subgroup of elements $p \in P$ for which $\Ad(g)$ preserves the grading $(\mfg_a)$ of $\mfg$, and by $P_+ < P$ the subgroup of elements $p \in P$ for which $\Ad(p) \in \End(\mfg)$ have homogeneity of degree $> 0$ with respect to the f\/iltration $(\mfg^a)$; in particular, the Lie algebra of $P_+$ is $\mfp_+ = \mfg^{+1} = \mfg_{+1} \oplus \cdots \oplus \mfg_{+k}$.

Since $\mfg$ is semisimple, its Killing form is nondegenerate, and it induces a $P$-equivariant identif\/ication $(\mfg / \mfp)^* \leftrightarrow \mfp_+$. Via this identif\/ication, for any $G$-representation $\bbU$ we may identify the Lie algebra homology $H_{\bullet}(\mfp_+, \bbU)$ with the chain complex
\begin{gather*}
	\cdots
		\to
	\Lambda^{i + 1} (\mfg / \mfp)^* \otimes \bbU
		\stackrel{\partial^{\ast}}{\to}
	\Lambda^ i (\mfg / \mfp)^* \otimes \bbU
		\to
	\cdots .
\end{gather*}
The \textit{Kostant codifferential} $\partial^*$ is $P$-equivariant, so it induces bundle maps $\partial^* \colon \Lambda^{i + 1} T^* M \otimes \mcU \to \Lambda^i T^* M \otimes \mcU$ between the associated bundles.

The normalization conditions for a parabolic geometry $(\mcG \to M, \omega)$ are that
\begin{enumerate}\itemsep=0pt
	\item[1)] (normality) the curvature $\kappa$ satisf\/ies $\partial^* \kappa = 0$, and
	\item[2)] (regularity) the curvature $\kappa$ satisf\/ies $\kappa(u)(\mfg^i, \mfg^j) \subseteq \mfg^{i + j + 1}$ for all $u \in \mcG$ and all $i$, $j$.
\end{enumerate}

Finally, tractor bundles associated to parabolic geometries inherit additional natural structure: Given a $G$-representation $\bbU$, $P$ determines a natural f\/iltration $(\bbU^a)$ of $\bbU$ by successive action of the nilpotent Lie subalgebra $\mfp_+ < \mfg$, namely
\begin{gather}\label{equation:general-representation-filtration}
	\bbU \supseteq \mfp_+ \cdot \bbU \supseteq \mfp_+ \cdot (\mfp_+ \cdot \bbU) \supseteq \cdots \supseteq \{ 0 \} .
\end{gather}
Since the f\/iltration $(\bbU^a)$ of $\bbU$ is $P$-invariant, it determines a bundle f\/iltration $(\mcU^a)$ of the tractor bundle $\mcU = \mcG \times_P \bbU$.

For the adjoint representation $\mfg$ itself, this f\/iltration (appropriately indexed) is just $(\mfg^a)$, and the images of the f\/iltrands $\mcA = \mcG \times_P \mfg^{-k} \supsetneq \cdots \supsetneq \mcG \times_P \mfg^{-1}$ under the projection $\Pi_0^{\mcA}$ comprise a canonical f\/iltration $TM = T^{-k} M \supsetneq \cdots \supsetneq T^{-1} M$ of the tangent bundle.

\subsubsection{Oriented conformal structures}\label{subsubsection:oriented-conformal-structures}

The group $\SO(p + 1, q + 1)$, $p + q \geq 3$, acts transitively on the space of isotropic rays in the standard representation $\bbV$, and the stabilizer subgroup $\bar P$ of such a ray is parabolic. There is an equivalence of categories between regular, normal parabolic geometries of type $(\SO(p + 1, q + 1), \bar P)$ and oriented conformal structures of signature $(p, q)$ \cite[Section~4.1.2]{CapSlovak}.
\begin{Definition}
A \textit{conformal structure} $(M, \mbc)$ is an equivalence class $\mbc$ of metrics on $M$, where we declare two metrics to be equivalent if one is a positive, smooth multiple of the other. The signature of $\mbc$ is the signature of any (equivalently, every) $g \in \mbc$, and we say that $(M, \mbc)$ is \textit{oriented} if\/f $M$ is oriented. The \textit{conformal holonomy} of an oriented conformal structure $\mbc$ is $\Hol(\mbc) := \Hol(\omega)$, where $\omega$ is the normal Cartan connection corresponding to $\mbc$.
\end{Definition}

We can choose a basis of $\bbV$ for which the nondegenerate, symmetric bilinear form $H$ preserved by $\SO(p + 1, q + 1)$ has block matrix representation
\begin{gather}\label{equation:bilinear-form-parabolic-adapted}
	\begin{pmatrix}
		0 & 0 & 1 \\
	 0 & \Sigma & 0 \\
	 1 & 0 & 0
	\end{pmatrix}.
\end{gather}
(With respect to the basis $(E_a)$, the matrix representation $[H_{\Phi}]$ \eqref{equation:bilinear-form-basis} of the bilinear form $H_{\Phi}$ determined by the explicit expression \eqref{equation:3-form-basis} for $\Phi$ has the form \eqref{equation:bilinear-form-parabolic-adapted}.) The Lie algebra $\mfso(p + 1,$ $q + 1)$ consists of exactly the elements
\begin{gather*}
	\begin{pmatrix}
		 b & Z & 0 \\
		 X & B & -\Sigma^{-1} Z^{\top} \\
		 0 & -X^{\top} \Sigma & -b
	\end{pmatrix} ,
\end{gather*}
where $B \in \mfso(\Sigma)$, $X \in \bbR^{p + q}$, $Z \in (\bbR^{p + q})^*$. The f\/irst element of the basis is isotropic, and if we take choose the preferred isotropic ray in $\bbV$ to be the one determined by that element, the corresponding Lie algebra grading on $\mfso(p + 1, q + 1)$ is the one def\/ined by the labeling
\begin{gather}\label{equation:conformal-grading}
	\begin{pmatrix}
		\mfg_0 & \mfg_{+1} & 0 \\
		\mfg_{-1} & \mfg_0 & \mfg_{+1} \\
		 0 & \mfg_{-1} & \mfg_0
	\end{pmatrix} .
\end{gather}

Since the grading on $\mfso(p + 1, q + 1)$ induced by $\bar P$ has the form $\mfg_{-1} \oplus \mfg_0 \oplus \mfg_{+1}$, any parabolic geometry of this type is regular. The normality condition coincides with Cartan's normalization condition for what is now called a Cartan geometry of this type \cite[Section~4.1.2]{CapSlovak}.

\subsubsection[Oriented (2,3,5) distributions]{Oriented $\boldsymbol{(2, 3, 5)}$ distributions}\label{subsubsection:235-distributions}

The group $\G_2$ acts transitively on the space of $H_{\Phi}$-isotropic rays in $\bbV$, and the stabilizer subgroup~$Q$ of such a ray is parabolic \cite{Sagerschnig}. The subgroup $Q$ is the intersection of $\G_2$ with the stabilizer subgroup $\bar P < \SO(3, 4)$ of the preferred isotropic ray in Section~\ref{subsubsection:oriented-conformal-structures}. In particular, the f\/irst basis element is isotropic, and if we again choose the preferred isotropic ray to be the one determined by that element, the corresponding Lie algebra grading on~$\mfg_2$ is the one def\/ined by the block decomposition \eqref{equation:g2-matrix-representation} and the labeling
\begin{gather*}
	\begin{pmatrix}
		\mfg_ 0 & \mfg_{+1} & \mfg_{+2} & \mfg_{+3} & 0 \\
		\mfg_{-1} & \mfg_ 0 & \mfg_{+1} & \mfg_{+2} & \mfg_{+3} \\
		\mfg_{-2} & \mfg_{-1} & 0 & \mfg_{+1} & \mfg_{+2} \\
		\mfg_{-3} & \mfg_{-2} & \mfg_{-1} & \mfg_ 0 & \mfg_{+1} \\
		 0 & \mfg_{-3} & \mfg_{-2} & \mfg_{-1} & \mfg_0
	\end{pmatrix} .
\end{gather*}

There is an equivalence of categories between regular, normal parabolic geometries of type $(\G_2, Q)$ and so-called oriented $(2, 3, 5)$ distributions \cite[Section~4.3.2]{CapSlovak}.

On a manifold $M$, def\/ine the bracket of distributions $\mathbf{E}, \mathbf{F} \subseteq TM$ to be the set $[\mathbf{E}, \mathbf{F}] := \{[\alpha, \beta]_x \colon x \in M; \alpha \in \Gamma(\mathbf{E}), \beta \in \Gamma(\mathbf{F})\} \subseteq TM$.

\begin{Definition}
A \textit{$(2, 3, 5)$ distribution} is a $2$-plane distribution $\mbD$ on a $5$-manifold $M$ that is maximally nonintegrable in the sense that (1) $[\mbD, \mbD]$ is a $3$-plane distribution, and (2) $[\mbD, [\mbD, \mbD]] = TM$. A $(2, 3, 5)$ distribution is \textit{oriented} if\/f the bundle $\mbD \to M$ is oriented.
\end{Definition}

An orientation of $\mbD$ determines an orientation of $M$ and vice versa. The appropriate restrictions of the Lie bracket of vector f\/ields descend to natural vector bundle isomorphisms $\smash{\mcL\colon \Lambda^2 \mbD \stackrel{\cong}{\to} [\mbD, \mbD] / \mbD}$ and $\smash{\mcL \colon \mbD \otimes ([\mbD, \mbD] / \mbD) \stackrel{\cong}{\to} TM / [\mbD, \mbD]}$; these are components of the \textit{Levi bracket}.

For a regular, normal parabolic geometry $(\mcG, \omega)$ of type $(\G_2, Q)$, the underlying $(2, 3, 5)$ distribution $\mbD$ is $T^{-1} M = \mcG \times_Q (\mfg^{-1} / \mfq)$, and $[\mbD, \mbD]$ is $T^{-2} M = \mcG \times_Q (\mfg^{-2} / \mfq)$.

\subsection{Conformal geometry}\label{subsection:conformal-tractor-geometry}

In this subsection we partly follow \cite{BEG}.

\subsubsection{Conformal density bundles}

A conformal structure $(M, \mbc)$ of signature $(p, q)$ (denote $n := p + q$), determines a family of natural \textit{$($conformal$)$ density bundles} on $M$: Denote by $\mcE[1]$ the positive $(2n)$th root of the canonically oriented line bundle $(\Lambda^n TM)^2$, and its respective $w$th integer powers by $\mcE[w]$; $\mcE := \mcE[0]$ is the trivial bundle with f\/iber $\bbR$, and there are natural identif\/ications $\mcE[w] \otimes \mcE[w'] \cong \mcE[w + w']$. Given any vector bundle $B \to M$, we denote $B[w] := B \otimes \mcE[w]$, and refer to the sections of $B[w]$ as \textit{sections of $B$ of conformal weight $w$}.

We may view $\mbc$ itself as the canonical \textit{conformal metric}, $\mbg_{ab} \in \Gamma(S^2 T^*M[2])$. Contraction with $\mbg_{ab}$ determines an isomorphism $TM \to T^*M[2]$, which we may use to raise and lower indices of objects on the tangent bundle at the cost of an adjustment of conformal weight. By construction, the Levi-Civita connection $\nabla^g$ of any metric $g \in \mbc$ preserves $\mbg_{ab}$ and its inverse, $\mbg^{ab} \in \Gamma(S^2 TM [-2])$.

We call a nowhere zero section $\tau \in \Gamma(\mcE[1])$ a \textit{scale} of $\mbc$. A scale determines \textit{trivializations} $\smash{B[w] \stackrel{\cong}{\to} B}$, $b \mapsto \ul b := \tau^{-w} b$, of all conformally weighted bundles, and in particular a representative metric $\tau^{-2} \mbg \in \mbc$.

\subsubsection{Conformal tractor calculus}
\label{conformal-tractor-calculus}

For an oriented conformal structure $(M, \mbc)$ of signature $(p, q)$, $n := p + q \geq 3$, the tractor bundle~$\mcV$ associated to the standard representation $\bbV$ of $\SO(p + 1, q + 1)$ is the \textit{standard tractor bundle}. It inherits from the normal parabolic geometry corresponding to~$\mbc$ a vector bundle connection~$\nabla^{\mcV}$. The $\SO(p + 1, q + 1)$-action preserves a canonical nondegenerate, symmetric bilinear form $H \in S^2 \bbV^*$ and a volume form $\epsilon \in \Lambda^{n + 2} \bbV^*$; these respectively induce on~$\mcV$ a~parallel \textit{tractor metric} $H \in \Gamma(S^2 \mcV^*)$ and parallel volume form $\epsilon \in \Gamma(\Lambda^{n + 2} \mcV^*)$.

Consulting the block structure \eqref{equation:conformal-grading} of $\bar \mfp_+ < \mfso(p + 1, q + 1)$ gives that the f\/iltration \eqref{equation:general-representation-filtration} of the standard representation $\bbV$ of $\SO(p + 1, q + 1)$ determined by $\bar P$ is
\begin{gather}\label{equation:conformal-standard-tractor-filtration}
	\left\{\tractorT{\ast}{\ast}{\ast}\right\}
		\supset
	\left\{\tractorT{ 0}{\ast}{\ast}\right\}
		\supset
	\left\{\tractorT{ 0}{ 0}{\ast}\right\}
		\supset
	\left\{\tractorT{ 0}{ 0}{ 0}\right\} .
\end{gather}
We may identify the composition series of the corresponding f\/iltration of $\mcV$ as
\begin{gather*}
	\mcV \cong \mcE[1], \flplus TM[-1], \flplus \mcE[-1] .
\end{gather*}
We denote elements and sections of $\mcV$ using uppercase Latin indices, $A, B, C, \ldots$, as $\bbS^A \in \Gamma(\mcV)$, and those of the dual bundle $\mcV^*$ with lower indices, as $\bbS_A \in \Gamma(\mcV^*)$; we freely raise and lower indices using $H$. The bundle inclusion $\mcE[-1] \hookrightarrow \mcV$ determines a canonical section $X^A \in \Gamma(\mcV[1])$.

Any scale $\tau$ determines an identif\/ication of $\mcV$ with the associated graded bundle determined by the above f\/iltration, that is, an isomorphism $\mcV \cong \mcE[1] \oplus TM[-1] \oplus \mcE[-1]$ \cite{BEG}. So, $\tau$ also determines (non-invariant, that is, scale-dependent) inclusions $TM[-1] \hookrightarrow \mcV$ and $\mcE[1] \hookrightarrow \mcV$, which we can respectively regard as sections $Z^A{}_a \in \Gamma(\mcV \otimes T^*M[1])$ and $Y^A \in \Gamma(\mcV[-1])$. So, for any choice of $\tau$ we can decompose a section $\bbS \in \Gamma(\mcV)$ uniquely as $\smash{\bbS^A \stackrel{\tau}{=} \sigma Y^A + \mu^a Z^A{}_a + \rho X^A}$, where the notation $\stackrel{\tau}{=}$ indicates that $Y^A$ and $Z^A{}_a$ are the inclusions determined by $\tau$. Reusing the notation of the f\/iltration of $\bbV$ we write
\begin{gather}\label{equation:standard-tractor-structure-splitting}
	\bbS^A \stackrel{\tau}{=} \tractorT{\sigma}{\mu^a}{\rho} .
\end{gather}
With respect to any scale $\tau$, the tractor metric has the form (cf.~\eqref{equation:bilinear-form-parabolic-adapted})
\begin{gather*}
	H_{AB}
		\stackrel{\tau}{=}
	\begin{pmatrix}
		0 & 0 & 1 \\
		0 & \mbg_{ab} & 0 \\
		1 & 0 & 0
	\end{pmatrix} .
\end{gather*}
In particular, the f\/iltration \eqref{equation:conformal-standard-tractor-filtration} of $\mcV$ is $\mcV \supset \langle X \rangle^{\perp} \supset \langle X \rangle \supset \{ 0 \}$.

The normal tractor connection $\nabla^{\mcV}$ on $\mcV$ is \cite{BEG}
\begin{gather*}
	\nabla^{\mcV}_b \tractorT{\sigma}{\mu^a}{\rho}
		\stackrel{\tau}{=}
		\tractorT
			{\sigma_{,b} - \mu_b}
			{\mu^a{}_{,b} + \sfP^a{}_b \sigma + \delta^a{}_b \rho}
			{\rho_{,b} - \sfP_{bc} \mu^c}
		\in
		\Gamma\left(\tractorT{\mcE[1]}{TM[-1]}{\mcE[-1]} \otimes T^*M\right) .
\end{gather*}
The subscript $_{,b}$ denotes the covariant derivative with respect to $g := \tau^{-2} \mbg$, and $\sfP_{ab}$ is the \textit{Schouten tensor} of $g$, which is a particular trace adjustment of the Ricci tensor $R_{ab}$:
\begin{gather}\label{equation:definition-Schouten}
	\sfP_{ab} := \frac{1}{n - 2} \left( R_{ab} - \frac{1}{2 (n - 1)} R^c{}_c g_{ab} \right) .
\end{gather}

A section $\bbA_{A_1 \cdots A_k}$ of the tractor bundle $\Lambda^k \mcV^*$ associated to the alternating representa\-tion~$\Lambda^k \bbV^*$ decomposes uniquely as
\begin{gather*}
\bbA_{A_1 \cdots A_k}\stackrel{\tau}{=}
	k \phi_{a_2 \cdots a_k} Y_{\smash{[}A_1} Z_{A_2}{}^{a_2} \cdots Z_{A_k\smash{]}}{}^{a_k}
		+ \chi_{a_1 \cdots a_k} Z_{\smash{[}A_1}{}^{a_1} \cdots Z_{A_k\smash{]}}{}^{a_k} \\
\hphantom{\bbA_{A_1 \cdots A_k}\stackrel{\tau}{=}}{} + k (k - 1) \theta_{a_3 \cdots a_k} Y_{\smash{[}A_1} X_{A_2} Z_{A_3}{}^{a_3} \cdots Z_{A_k\smash{]}}{}^{a_k} + k \psi_{a_2 \cdots a_k} X_{\smash{[}A_1} Z_{A_2}{}^{a_2} \cdots Z_{A_k\smash{]}}{}^{a_k},
\end{gather*}
which we write more compactly as
\begin{gather*}
\bbA_{A_1 \cdots A_k}\stackrel{\tau}{=}	\tractorQ{\phi_{a_2 \cdots a_k}}{\chi_{a_1 \cdots a_k}}{\theta_{a_3 \cdots a_k}}{\psi_{a_2 \cdots a_k}}
\in	\Gamma\tractorQ{\Lambda^{k - 1} T^*M [k]}{\Lambda^k T^*M [k]}{\Lambda^{k - 2} T^*M [k - 2]}{\Lambda^{k - 1} T^*M [k - 2]} .
\end{gather*}
The tractor connection $\nabla^{\mcV}$ induces a connection on $\Lambda^k \mcV^*$, and we denote this connection again by $\nabla^{\mcV}$.

In the special case $k = 2$, raising an index using $H$ gives $\Lambda^2 \bbV^* \cong \mfso(p + 1, q + 1)$, so we can identify $\Lambda^2 \mcV^* \cong \mcA$.
Any section $\bbA^A{}_B \in \Gamma(\mcA)$ decomposes uniquely as
\begin{gather*}
	\bbA^A{}_B = \xi^a \big(Y^A Z_{Ba} - Z^A{}_a Y_B\big) + \zeta^a{}_b Z^A{}_a Z_B{}^b \\
 \hphantom{\bbA^A{}_B =}{} + \alpha \big(Y^A X_B - X^A Y_B\big) + \nu_b \big(X^A Z_B{}^b - Z^{Ab} X_B\big) ,
\end{gather*}
which we write as
\begin{gather*}
	\bbA^A{}_B
		\stackrel{\tau}{=}
	\tractorQ{\xi^b}{\zeta^a{}_b}{\alpha}{\nu_b}
		\in
	\Gamma\tractorQ{TM}{\End_{\skewOp}(TM)}{\mcE}{T^* M} .
\end{gather*}

Finally, $\bar\mfp_+$ annihilates $\Lambda^{n + 2} \bbV^* \cong \bbR$, yielding a natural bundle isomorphism $\Lambda^{n + 2} \mcV^* \cong \Lambda^n T^*M [n]$. This identif\/ies the tractor volume form $\epsilon$ with the conformal volume form~$\epsilon_{\mbg}$ of~$\mbg$.

\subsubsection{Canonical quotients of conformal tractor bundles}

For any irreducible $\SO(p + 1, q + 1)$-representation $\bbU$, the canonical Lie algebra cohomology quotient map $\bbU \mapsto H_0 := H_0(\mfp_+, \bbU) = \bbU / (\bar{\mfp}_+ \cdot \bbU)$ is $\bar P$-invariant and so induces a canonical bundle quotient map $\Pi_0^{\mcU}\colon \mcU \to \mcH_0$ between the corresponding associated $\bar{P}$-bundles. (We reuse the notation $\Pi_0^{\mcU}$ for the induced map $\Gamma(\mcU) \to \Gamma(\mcH_0)$ on sections.) Given a section $\bbA \in \Gamma(\mcU)$, its image $\Pi_0^{\mcU}(\bbA) \in \Gamma(\mcH_0)$ is its \textit{projecting part}.

For the standard representation $\bbV$ this quotient map is $\Pi_0^{\mcV} \colon \mcV \to \mcE[1]$,
\begin{gather*}
	\Pi_0^{\mcV} \colon \ \tractorT{\sigma}{\ast}{\ast} \mapsto \sigma .
\end{gather*}
For the alternating representation $\Lambda^k \bbV^*$, the quotient map is $\smash{\Pi_0^{\Lambda^k \mcV^*}} \colon \Lambda^k \mcV^* \to \Lambda^{k - 1} T^*M [k]$,
\settowidth{\splittingWidth}{$\phi_{b_2 \cdots b_k}$}
\begin{gather}
\label{equation:projection-operator-alternating}
\def\arraystretch{1.1}
	\Pi_0^{\Lambda^k \mcV^*} \colon \
	\left(
		\begin{array}{C{0.2\splittingWidth}cC{0.2\splittingWidth}}
			\multicolumn{3}{c} \ast \\
			\ast & | & \ast \\
			\multicolumn{3}{c}{\phi_{a_2 \cdots a_k}}
		\end{array}
	\right)
		\mapsto
	 \phi_{a_2 \cdots a_k}
\end{gather}

For the adjoint representation $\mfso(p + 1, q + 1)$, the quotient map coincides with the map $\Pi_0^{\mcA} \colon \mcA \to TM$ def\/ined in Section~\ref{subsubsection:tractor-geometry}; in a splitting, it is
\begin{gather*}
	\Pi_0^{\mcA} \colon \ \tractorQ{\xi^a}{\ast}{\ast}{\ast} \mapsto \xi^a .
\end{gather*}

\subsubsection{Conformal BGG splitting operators}
\label{subsubsection:BGG-splitting-operators}

Conversely, for each irreducible $\SO(p + 1, q + 1)$-representation $\bbU$ there is a canonical dif\/ferential \textit{BGG splitting operator} $L_0^{\mcU}\colon \Gamma(\mcH_0) \to \Gamma(\mcU)$ characterized by the properties (1) $\Pi_0^{\mcU} \circ L_0^{\mcU} = \id_{\mcH_0}$ and (2) $\partial^{\ast} \circ \nabla^{\mcU} \circ L_0^{\mcU} = 0$ \cite{CalderbankDiemer,CSS}. The only property of the operators $L_0^{\mcU}$ we need here follows immediately from this characterization: If $\bbA \in \Gamma(\mcU)$ is $\nabla^{\mcU}$-parallel, then $L_0^{\mcU}(\Pi_0^{\mcU}(\bbA)) = \bbA$.

\subsection{Almost Einstein scales}
\label{subsection:almost-Einstein}

The BGG splitting operator $L_0^{\mcV} \colon \Gamma(\mcE[1]) \to \Gamma(\mcV)$ corresponding to the standard representation is \cite[equation~(114)]{Hammerl}
\begin{gather}\label{equation:splitting-operator-standard}
	L_0^{\mcV}\colon \ \sigma \mapsto \tractorT{\sigma}{\sigma^{,a}}{- \tfrac{1}{n} (\sigma_{,b}{}^b + \sfP^b{}_b \sigma)} .
\end{gather}

Computing gives
\begin{gather}\label{equation:nabla-L0-standard}
	\nabla^{\mcV}_b L_0^{\mcV}(\sigma)^A \stackrel{\tau}{=} \tractorT{0}{(\sigma_{,ab} + \sfP_{ab} \sigma)_{\circ}}{\ast} \in \Gamma\left(\tractorT{\mcE[1]}{T^*M[1]}{\mcE[-1]} \otimes T^* M\right) ,
\end{gather}
where $(T_{ab})_{\circ}$ denotes the tracefree part $\smash{T_{ab} - \tfrac{1}{n} T^c{}_c \mbg_{ab}}$ of the (possibly weighted) covariant $2$-tensor $T_{ab}$, and where $\ast$ is some third-order dif\/ferential expression in $\sigma$. Since the bottom component of $\nabla^{\mcV} L_0^{\mcV}(\sigma)$ is zero, the middle component, regarded as a (second-order) linear dif\/ferential operator $\Theta_0^{\mcV} \colon \Gamma(\mcE[1]) \to \Gamma(S^2 T^*M [1])$,
\begin{gather*}
	\Theta_0^{\mcV} \colon \ \sigma \mapsto (\sigma_{,ab} + \sfP_{ab} \sigma)_{\circ} ,
\end{gather*}
is conformally invariant. The operator $\Theta_0^{\mcV}$ is the \textit{first BGG operator} \cite{CSS} associated to the standard representation $\bbV$ for (oriented) conformal geometry.

We can readily interpret a solution $\sigma \in \ker \Theta_0^{\mcV}$ geometrically: If we restrict to the complement $M - \Sigma$ of the zero locus $\Sigma := \{x \in M \colon \sigma_x = 0\}$, we can work in the scale of the solution $\sigma$ itself: We have $\smash{\sigma \stackrel{\sigma}{=} 1}$ and hence $0 = \Theta_0^{\mcV}(\sigma) = \sfP_{\circ}$. This says simply says that the Schouten tensor, $\sfP$, of $g := \sigma^{-2} \mbg \vert_{M - \Sigma}$ is a multiple of $g$, and hence so is its Ricci tensor, that is, that $g$ is Einstein. This motivates the following def\/inition \cite{GoverAlmostEinstein}:

\begin{Definition}
An \textit{almost Einstein scale}\footnote{Our terminology follows that of the literature on almost Einstein scales, but this consistency entails a mild perversity, namely that, since they may vanish, almost Einstein scales need not be scales.} of an (oriented) conformal structure of dimension $n \geq 3$ is a solution $\sigma \in \Gamma(\mcE[1])$ of the operator $\Theta_0^{\mcV}$. A conformal structure is \textit{almost Einstein} if it admits a nonzero almost Einstein scale.
\end{Definition}

We denote the set $\ker \Theta_0^{\mcV}$ of almost Einstein scales of a given conformal structure $\mbc$ by $\aEs(\mbc)$. Since $\Theta_0^{\mcV}$ is linear, $\aEs(\mbc)$ is a vector subspace of $\Gamma(\mcE[1])$.

The vanishing of the component $\ast$ in \eqref{equation:nabla-L0-standard} turns out to be a dif\/ferential consequence of the vanishing of the middle component, $\Theta_0^{\mcV}(\sigma)$. So, $\nabla^{\mcV}$ is a prolongation connection for the opera\-tor~$\Theta_0^{\mcV}$:

\begin{Theorem}[{\cite[Section~2]{BEG}}] \label{theorem:almost-Einstein-bijection}
For any conformal structure $(M, \mbc)$, $\dim M \geq 3$, the restrictions of $L_0^{\mcV} \colon \Gamma(\mcE[1]) \to \Gamma(\mcV)$ and $\Pi_0^{\mcV} \colon \Gamma(\mcV) \to \Gamma(\mcE[1])$ comprise a natural bijective correspondence between almost Einstein scales and parallel standard tractors:
\begin{gather*}
	\aEs(\mbc)\mathrel{\mathop{\rightleftarrows}^{L_0^{\mcV}}_{\Pi_0^{\mcV}}}
	\big\{\text{$\nabla^{\mcV}$-parallel sections of $\mcV$}\big\} .
\end{gather*}
\end{Theorem}
In particular, if $\sigma$ is an almost Einstein scale and vanishes on some nonempty open set, then $\sigma = 0$. In fact, the zero locus $\Sigma$ of $\sigma$ turns out to be a smooth hypersurface~\cite{CGH}; see Example~\ref{example:curved-orbit-decomposition-almost-Einstein}.

We def\/ine the \textit{Einstein constant} of an almost Einstein scale $\sigma$ to be
\begin{gather}\label{equation:Einstein-constant}
	\lambda := -\tfrac{1}{2} H(L_0^{\mcV}(\sigma), L_0^{\mcV}(\sigma)) = \tfrac{1}{n} \sigma (\sigma_{,a}{}^a + \sfP^a{}_a \sigma) - \tfrac{1}{2} \sigma_{,a} \sigma^{,a} .
\end{gather}
This def\/inition is motivated by the following computation: On $M - \Sigma$ the Schouten tensor of the representative metric $g := \sigma^{-2} \mbg \vert_{M - \Sigma} \in \mbc\vert_{M - \Sigma}$ determined by the scale $\sigma\vert_{M - \Sigma}$ is $\sfP = \lambda g$. Thus, the Ricci tensor of $g$ is $R_{ab} = 2 (n - 1) \lambda g_{ab}$, so we say that $\sigma$ (or the metric $g$ it induces) is \textit{Ricci-negative}, \textit{-flat}, or \textit{-positive} respectively if\/f $\lambda < 0$, $\lambda = 0$, or $\lambda > 0$.\footnote{The def\/inition here of \textit{Einstein constant} is consistent with some of the literature on almost Einstein conformal structures, but elsewhere this term is sometimes used for the quantity $2 (n - 1) \lambda$.}

\subsection[Conformal Killing f\/ields and (k-1)-forms]{Conformal Killing f\/ields and $\boldsymbol{(k - 1)}$-forms}\label{subsection:conformal-Killing}

The BGG splitting operator $L_0^{\Lambda^k \mcV^*} \colon \Gamma(\Lambda^{k - 1} T^*M [k]) \to \Gamma(\Lambda^k \mcV^*)$ determined by the alternating representation $\Lambda^k \bbV^*$, $1 < k < n + 1$, is \cite[equation~(134)]{Hammerl}
\settowidth{\splittingWidth}{$\left(\!\!\!\!{\def\arraystretch{1.1}\begin{array}{@{}c@{}}
			\!\!\!\tfrac{1}{n} \big[ {-}\tfrac{1}{k} \omega_{a_2 \ldots a_k, b}{}^b + \tfrac{k - 1}{k} \omega_{b[a_3 \cdots a_k, a_2]}{}^b + \tfrac{k - 1}{n - k + 2} \omega_{b[a_3 \cdots a_k,}{}^b{}_{a_2]} \\
			+ 2 (k - 1) \sfP^b{}_{[a_2} \omega_{|b| a_3 \cdots a_k]} - \sfP^b{}_b \omega_{a_2 \cdots a_k} \big]
			\end{array}}\!\!\!\!\right)$}
\begin{gather}\label{equation:splitting-operator-alternating}
\def\arraystretch{1.1}
	L_0^{\Lambda^k \mcV^*} \colon \ \phi_{a_2 \cdots a_k}
		\mapsto
	\left(\!\!
		\begin{array}{C{0.46\splittingWidth}cC{0.46\splittingWidth}}
			\multicolumn{3}{c}{
			\left({\def\arraystretch{1.1}\begin{array}{@{}c@{}}
			\tfrac{1}{n} \big[ {}-\tfrac{1}{k} \phi_{a_2 \ldots a_k, b}{}^b + \tfrac{k - 1}{k} \phi_{b[a_3 \cdots a_k, a_2]}{}^b + \tfrac{k - 1}{n - k + 2} \phi_{b[a_3 \cdots a_k,}{}^b{}_{a_2]} \\
			+ 2 (k - 1) \sfP^b{}_{[a_2} \phi_{|b| a_3 \cdots a_k]} - \sfP^b{}_b \phi_{a_2 \cdots a_k} \big]
			\end{array}}\right)
			} \\
			{\phi_{[a_2 \cdots a_k, a_1]}} & | & {- \tfrac{1}{n - k + 2} \phi_{b [a_3 \cdots a_k a_2],}{}^b} \\
			\multicolumn{3}{c}{\phi_{a_2 \cdots a_k}}
		\end{array}\!\!
	\right) \! .\!\!\!\!
\end{gather}
Proceeding as in Section~\ref{subsection:almost-Einstein}, we f\/ind that
\settowidth{\splittingWidth}{$\phi_{a_2 \cdots a_k, b} - \phi_{[a_2 \cdots a_k, b]} - \tfrac{k - 1}{n - k + 2} \mbg_{b [a_2} \phi_{|c| a_3 \cdots a_k]},{}^c$}
\begin{gather*}
	\nabla^{\mcV}_b L_0^{\Lambda^k \mcV^*}(\phi)_{A_1 \cdots A_k}
		=
	\left(
		\begin{array}{C{0.44\splittingWidth}cC{0.44\splittingWidth}}
			\multicolumn{3}{c}{\ast} \\
			\ast & | & \ast \\
			\multicolumn{3}{c}{\phi_{a_2 \cdots a_k, b} - \phi_{[a_2 \cdots a_k, b]} - \tfrac{k - 1}{n - k + 2} \mbg_{b [a_2} \phi_{|c| a_3 \cdots a_k],}{}^c}
		\end{array}
	\right) ,
\end{gather*}
where each $\ast$ denotes some dif\/ferential expression in $\sigma$. The bottom component def\/ines an inva\-riant conformal dif\/ferential operator $\smash{\Theta_0^{\Lambda^k \mcV^*} \colon \Gamma(\Lambda^{k - 1} T^*M [k]) \to \Gamma(\Lambda^{k - 1} T^*M \odot T^*M [k])}$ (here $\odot$ denotes the Cartan product) and elements of its kernel are called \textit{conformal Killing $(k - 1)$-forms}~\cite{Semmelmann}. Unlike in the case of almost Einstein scales, vanishing of $\smash{\Theta_0^{\Lambda^k \mcV^*}(\phi)}$ does not in general imply the vanishing of the remaining components $\ast$; if they do vanish, that is, if $\smash{\nabla^{\mcV} L_0^{\Lambda^k \mcV^*}(\phi) = 0}$, $\phi$ is called a~\textit{normal} conformal Killing $(k - 1)$-form \cite[Section~6.2]{Hammerl}, \cite{LeitnerNormalConformalKillingForms}.

The BGG splitting operator $L_0^{\mcA} \colon \Gamma(TM) \to \Gamma(\mcA)$ for the adjoint representation $\mfso(p + 1, q + 1)$ is \cite[equation~(119)]{Hammerl}
\settowidth{\splittingWidth}{$\tfrac{1}{n} \left( -\tfrac{1}{2} \xi_{b, c}{}^c + \tfrac{1}{2} \xi^c{}_{,bc} + \tfrac{1}{n} \xi^c{}_{,c b} + 2 \sfP_{bc} \xi^c - \sfP^c{}_c \xi_b\right)$}
\begin{gather*}
\def\arraystretch{1.1}
	L_0^{\Lambda^k \mcV^*} \colon \ \xi^b
		\mapsto
	\left(\!\!
		\begin{array}{C{0.46\splittingWidth}cC{0.46\splittingWidth}}
			\multicolumn{3}{c}{
			\tfrac{1}{n} \left( -\tfrac{1}{2} \xi_{b, c}{}^c + \tfrac{1}{2} \xi^c{}_{,bc} + \tfrac{1}{n} \xi^c{}_{,c b} + 2 \sfP_{bc} \xi^c - \sfP^c{}_c \xi_b\right)
			} \\
			\tfrac{1}{2}(-\xi^a{}_{,b} + \xi_{b,}{}^a) & | & -\tfrac{1}{n} \xi^c{}_{,c} \\
			\multicolumn{3}{c}{\xi^b}
		\end{array}\!\!
	\right) .
\end{gather*}
So viewed, $\Theta_0^{\mcA}$ is the map $\Gamma(TM) \to \Gamma(S^2_{\circ} T^*M [2])$, $\smash{\xi^a \mapsto (\xi_{(a, b)})_{\circ}} = (\mcL_{\xi} \mbg)_{ab}$. Thus, the solutions of $\ker \Theta_0^{\mcA}$ are precisely the vector f\/ields whose f\/low preserves $\mbc$, and so these are called \textit{conformal Killing fields}. If $\nabla^{\mcV} L_0^{\mcA}(\xi) = 0$, we say $\xi$ is a \textit{normal} conformal Killing f\/ield.

\subsection[(2,3,5) conformal structures]{$\boldsymbol{(2, 3, 5)}$ conformal structures}\label{subsection:235-conformal-structure}

About a decade ago, Nurowski observed the following:
\begin{Theorem}
A $(2, 3, 5)$ distribution $(M, \mbD)$ canonically determines a conformal structure $\mbc_{\mbD}$ of signature $(2, 3)$ on~$M$.
\end{Theorem}
This construction has since been recognized as a special case of a \textit{Fefferman construction}, so named because it likewise generalizes a classical construction of Fef\/ferman that canonically assigns to any nondegenerate hypersurface-type CR structure on a manifold~$N$ a conformal structure on a natural circle bundle over~$N$~\cite{Fefferman}. In fact, this latter construction arises in our setting, too; see Section~\ref{subsubsection:curved-orbit-negative-hypersurface}.

We use the following terminology:
\begin{Definition}
A conformal structure $\mbc$ is a \textit{$(2, 3, 5)$ conformal structure} if\/f $\mbc = \mbc_{\mbD}$ for some $(2, 3, 5)$ distribution $\mbD$.
\end{Definition}

An oriented $(2, 3, 5)$ distribution $\mbD$ determines an orientation of $TM$, and hence $\mbc_{\mbD}$ is oriented (henceforth, that symbol refers to an oriented conformal structure).

Because we will need some of the ingredients anyway, we brief\/ly sketch a construction of $\mbc_{\mbD}$ using the framework of parabolic geometry: Fix an oriented $(2, 3, 5)$ distribution $(M, \mbD)$, and per Section~\ref{subsubsection:235-distributions} let $(\mcG \to M, \omega)$ be the corresponding regular, normal parabolic geometry of type $(\G_2, Q)$. Form the extended bundle $\bar\mcG := \mcG \times_Q \bar P$, and let $\bar\omega$ denote the Cartan connection equivariantly extending $\omega$ to $\bar\mcG$. By construction $(\bar\mcG, \bar\omega)$ is a parabolic geometry of type $(\SO(3, 4), \bar P)$ (for which $\bar\omega$ turns out to be normal, see \cite[Proposition 4]{HammerlSagerschnig}), and hence def\/ines an oriented conformal structure on $M$.

For any $(2, 3, 5)$ distribution $(M, \mbD)$ and for any representation $\bbU$ of $\SO(p + 1, q + 1)$, we may identify the associated tractor bundle $\mcG \times_Q \bbU$ (here regarding $\bbU$ as a $Q$-representation) with the conformal tractor bundle $\bar\mcG \times_{\bar P} \bbU$, and so denote both of these bundles by $\mcU$. Since $\bar\omega$ is itself normal, the (normal) tractor connections that $\omega$ and $\bar\omega$ induce on $\mcU$ coincide.

\subsubsection[Holonomy characterization of oriented (2,3,5) conformal structures]{Holonomy characterization of oriented $\boldsymbol{(2, 3, 5)}$ conformal structures}

An oriented $(2, 3, 5)$-distribution $\mbD$ corresponds to a regular, normal parabolic geometry $(\mcG, \omega)$ of type $(\G_2, Q)$. In particular, this determines on the tractor bundle $\mcV = \mcG \times_Q \bbV$ a $\G_2$-structure $\Phi \in \Gamma(\Lambda^3 \mcV^*)$ parallel with respect to the induced normal connection on $\mcV$, and again we may identify $\mcV$ and the normal connection thereon with the standard conformal tractor bundle $\smash{\bar\mcG \times_{\bar P} \bbV}$ of $\mbc_{\mbD}$ and the normal conformal tractor connection. The $\G_2$-structure determines f\/iberwise a bilinear form $H_{\Phi} \in \Gamma(S^2 \mcV^*)$. Since this construction is algebraic, $H_{\Phi}$ is parallel, and by construction it coincides with the conformal tractor metric on $\mcV$ determined by $\mbc_{\mbD}$.

Conversely, if an oriented, signature $(2, 3)$ conformal structure $\mbc$ admits a parallel tractor $\G_2$-structure $\Phi$ whose restriction to each f\/iber $\mcV_x$ is compatible with the restriction $H_x$ of the tractor metric (in which case we simply say that $\Phi$ is compatible with~$H$), the distribution $\mbD$ underlying $\Phi$ satisf\/ies $\mbc = \mbc_{\mbD}$. This recovers a correspondence stated in the original work of Nurowski~\cite{Nurowski} and worked out in detail in~\cite{HammerlSagerschnig}:
\begin{Theorem}\label{theorem:2-3-5-holonomy-characterization}
An oriented conformal structure $(M, \mbc)$ $($necessarily of signature $(2, 3))$ is induced by some $(2, 3, 5)$ distribution $\mbD$ $($that is, $\mbc = \mbc_{\mbD})$ iff the normal conformal tractor connection admits a holonomy reduction to~$\G_2$, or equivalently, iff $\mbc$ admits a parallel tractor $\G_2$-structure~$\Phi$ compatible with the tractor metric~$H$.
\end{Theorem}

\subsubsection[The conformal tractor decomposition of the tractor G2-structure]{The conformal tractor decomposition of the tractor $\boldsymbol{\G_2}$-structure}

Fix an oriented $(2, 3, 5)$ distribution $(M, \mbD)$, let $\Phi \in \Gamma(\Lambda^3 \mcV^*)$ denote the corresponding parallel tractor $\G_2$-structure, and denote its components with respect to any scale $\tau$ of the induced conformal structure~$\mbc_{\mbD}$ according to
\begin{gather}\label{equation:G2-structure-splitting}
	\Phi_{ABC}
		\stackrel{\tau}{=}
	\tractorQ{\phi_{bc}}{\chi_{abc}}{\theta_c}{\psi_{bc}}
		\in
	\Gamma\tractorQ{\Lambda^2 T^*M [3]}{\Lambda^3 T^*M [3]}{T^*M [1]}{\Lambda^2 T^*M [1]} .
\end{gather}

In the language of Section~\ref{subsection:conformal-Killing}, $\smash{\phi = \Pi_0^{\Lambda^3 \mcV^*}(\Phi)}$ is a normal conformal Killing $2$-form, and $\smash{\Phi = L_0^{\Lambda^3 \mcV^*}(\phi)}$. An argument analogous to that in the proof of Proposition \ref{proposition:identites-g2-structure-components}(5) below shows that $\phi$ is locally decomposable, and Proposition \ref{proposition:identites-g2-structure-components}(8) shows that it vanishes nowhere, so the (weighted) bivector f\/ield $\phi^{ab} \in \Gamma(\Lambda^2 TM [-1])$ determines a $2$-plane distribution on $M$, and this is precisely $\mbD$ \cite{HammerlSagerschnig}.

We collect for later some useful geometric facts about $\mbD$ and encode them in algebraic identities in the tractor components $\phi, \chi, \theta, \psi$. Parts (1) and (2) of the Proposition \ref{proposition:identites-g2-structure-components} are well-known features of $(2, 3, 5)$ distributions.

\begin{Proposition}
\label{proposition:identites-g2-structure-components}
Let $(M, \mbD)$ be an oriented $(2, 3, 5)$ distribution, let $\Phi \in \Gamma(\Lambda^3 \mcV^*)$ denote the corresponding parallel tractor $\G_2$-structure, and denote its components with respect to an arbitrary scale $\tau$ as in~\eqref{equation:G2-structure-splitting}. Then:
\begin{enumerate}\itemsep=0pt
	\item[$1.$] 
The distribution $\mbD$ is totally $\mbc_{\mbD}$-isotropic; equivalently, $\phi^{ac} \phi_{cb} = 0$.
	\item[$2.$] 
The annihilator of $\phi_{ab}$ (in $TM$) is $[\mbD, \mbD]$, and hence $\mbD^{\perp} = [\mbD, \mbD]$ $($here, $\mbD^{\perp}$ is the subbundle of $TM$ orthogonal to $\mbD$ with respect to $\mbc_{\mbD})$; equivalently, $\phi^{bc} \chi_{bca} = 0$.
	\item[$3.$] 
The weighted vector field $\theta^b \in \Gamma(TM[-1])$ is a section of $[\mbD, \mbD][-1]$, or equivalently, the line field $\mbL$ that $\theta$ determines $($which depends on $\tau)$ is orthogonal to $\mbD$; equivalently, $\theta^b \phi_{ba} = 0$.
	\item[$4.$] 
The weighted vector field $\theta^b$ satisfies $\theta_b \theta^b = -1$. In particular, the line field $\mbL$ is timelike.
	\item[$5.$]
 Like $\phi$, the weighted $2$-form $\psi$ is locally decomposable, that is, $(\psi \wedge \psi)_{abcd} = 6 \psi_{[ab} \psi_{cd]} = 0$. Since $($by equation \eqref{eq-item:conformal-volume-form-identity}$)$ it vanishes nowhere, it determines a $2$-plane distribution $\mbE$ $($which depends on $\tau)$.
	\item[$6.$] The distribution $\mbE$ is totally $\mbc_{\mbD}$-isotropic; equivalently, $\psi^{ac} \psi_{cb} = 0$.
	\item[$7.$] The line field $\mbL$ is orthogonal to $\mbE$; equivalently, $\theta^b \psi_{ba} = 0$.
	\item[$8.$] 
The $($weighted$)$ conformal volume form $\epsilon_{\mbg} \in \Gamma(\Lambda^5 T^*M [5])$ satisfies
	\begin{gather}\label{eq-item:conformal-volume-form-identity}
		(\epsilon_{\mbg})_{abcde} \stackrel{\tau}{=} \tfrac{1}{2} (\phi \wedge \theta \wedge \psi)_{abcde} = 15 \phi_{[ab} \theta_c \psi_{de]} .
	\end{gather}
\end{enumerate}
In particular, {\rm (8)} implies that $\mbD$, $\mbL$, and $\mbE$ are pairwise transverse and hence span~$TM$. Moreover, {\rm (2)} and {\rm (3)} imply that $\mbD \oplus \mbL = [\mbD, \mbD]$ and so $\mbD \oplus \mbL \oplus \mbE$ is a splitting of the canonical filtration $\mbD \subset [\mbD, \mbD] \subset TM$.\footnote{The splitting $\mbD \oplus \mbL \oplus \mbE$ determined by $\tau$ is a special case of a general feature of parabolic geometry, in which a choice of Weyl structure yields a splitting of the canonical f\/iltration of the tangent bundle of the underlying structure \cite[Section~5.1]{CapSlovak}.}
\end{Proposition}

It is possible to give abstract proofs of the identities in Proposition~\ref{proposition:identites-g2-structure-components}, but it is much faster to use frames of the standard tractor bundle suitably adapted to the parallel tractor $\G_2$-structure~$\Phi$.

\begin{proof}[Proof of Proposition \ref{proposition:identites-g2-structure-components}] Call a local frame $(E_a)$ of $\mcV$ \textit{adapted} to $\Phi$ if\/f (1) $E_1$ is a local section of the line subbundle $\langle X \rangle$ determined by $X$, and (2) the representation of $\Phi$ in the dual coframe $(e^a)$ is given by \eqref{equation:3-form-basis}; it follows from \cite[Theorem 3.1]{Wolf} that such a local frame exists in some neighborhood of any point in~$M$.

Any adapted local frame determines a (local) choice of scale: Since $X \in \Gamma(\mcV[1])$, we have $\tau := e^7(X) \in \mcE[1]$, and by construction it vanishes nowhere. Then, since $\langle X \rangle^{\perp} = \langle E_1, \ldots, E_6 \rangle$, the (weighted) vector f\/ields $F_a := E_a + \langle E_1 \rangle$, $a = 2, \ldots, 6$ comprise a frame of $\langle X \rangle^{\perp} / \langle X \rangle$ which by Section~\ref{conformal-tractor-calculus} is canonically isomorphic to $TM[-1]$. Trivializing these frame f\/ields (by multiplying by $\tau$) yields a local frame $(\ul F_2, \ldots, \ul F_6)$ of $TM$; denote the dual coframe by $(f^2, \ldots, f^6)$. One can read immediately from \eqref{equation:3-form-basis} that in an adapted local frame, (the trivialized) components of~$\Phi$ are
\begin{gather*}
	\ul \phi \stackrel{\tau}{=} \sqrt{2} f^5 \wedge f^6, \qquad
	\ul \chi \stackrel{\tau}{=} f^2 \wedge f^4 \wedge f^5 + f^3 \wedge f^4 \wedge f^6, \qquad
	\ul\theta \stackrel{\tau}{=} f^4, \qquad
	\ul \psi \stackrel{\tau}{=} \sqrt{2} f^2 \wedge f^3 ,
\end{gather*}
and consulting the form of equation \eqref{equation:bilinear-form-basis} gives that the (trivialized) conformal metric is
\begin{gather*}
	\ul\mbg \stackrel{\tau}{=} f^2 f^5 + f^3 f^6 - \big(f^4\big)^2 .
\end{gather*}
In an adapted frame, $\epsilon_{\Phi}$ is given by $-e^1 \wedge \cdots \wedge e^7$, so the (trivialized) conformal volume form is $\ul\epsilon_{\mbg} = f^2 \wedge \cdots \wedge f^6$.

All of the identities follow immediately from computing in this frame. For example, to compute (1), we see that raising indices gives $\smash{\ul\phi^{\sharp \sharp} = \sqrt{2} \ul F_2 \wedge \ul F_3}$, and that contracting an index of this bivector f\/ield with $\ul\phi = \sqrt{2} f^5 \wedge f^6$ yields $0$.

It remains to show that the geometric assertions are equivalent to the corresponding identities; these are nearly immediate for all but the f\/irst two parts. For both parts, pick a local frame $(\alpha, \beta)$ of $\mbD$ around an arbitrary point; by scaling we may assume that $\ul \phi^{\sharp\sharp} = \alpha \wedge \beta$.
\begin{enumerate}\itemsep=0pt
	\item The identity implies that the trace over the second and third indices of the tensor product $\smash{\ul\phi^{\sharp\sharp} \otimes \ul\phi = (\alpha \wedge \beta) \otimes (\alpha^{\flat} \wedge \beta^{\flat})}$ is zero, or, expanding, that
		\begin{gather*}
			0 = -\ul\mbg(\alpha, \alpha) \beta \otimes \beta^{\flat} + \ul\mbg(\alpha, \beta) \alpha \otimes \beta^{\flat} + \ul\mbg(\alpha, \beta) \beta \otimes \alpha^{\flat} - \ul\mbg(\beta, \beta) \beta \otimes \beta^{\flat} .
		\end{gather*}
		Since $\alpha$, $\beta$ are linearly independent, the four coef\/f\/icients on the right-hand side vanish separately, but up to sign these are the components of the restriction of $\mbc_{\mbD}$ to $\mbD$ in the given frame.
	\item By Part~(1), $\phi(\alpha, \,\cdot\,) = \phi(\beta, \,\cdot\,) = 0$. Any local section $\eta \in \Gamma([\mbD, \mbD])$ can be written as $\eta = A \alpha + B \beta + C [\alpha, \beta]$ for some smooth functions $A$, $B$, $C$, giving $\phi(\eta, \gamma) = C \phi([\alpha, \beta], \gamma)$. The invariant formula for the exterior derivative of a $2$-form then gives $\phi([\alpha, \beta], \gamma) = -d\phi(\alpha, \beta, \gamma)$. Now, in the chosen scale, $d\phi = \chi$ so
		\begin{gather*}
			-[d\phi(\alpha, \beta, \,\cdot\,)]_a = -\chi_{bca} \alpha^b \beta^c = -\tfrac{1}{2} \chi_{bca} \cdot 2 \alpha^{[b} \beta^{c]} = -\tfrac{1}{2} \chi_{bca} \phi^{bc} . \tag*{\qed}
		\end{gather*}
\end{enumerate}\renewcommand{\qed}{}
\end{proof}

Since the tractor Hodge star operator $\ast_{\Phi}$ is algebraic, $\ast_{\Phi} \Phi \in \Gamma(\Lambda^4 \mcV^*)$ is parallel. We can express its components with respect to a scale $\tau$ in terms of those of $\Phi$ and the weighted Hodge star operators $\ast \colon \Lambda^l T^*M [w] \to \Lambda^{5 - l} T^*M [7 - w]$ determined by $\mbg$ \cite{LeitnerNormalConformalKillingForms}:
\begin{gather}\label{equation:Hodge-star-3-form}
	(\ast_{\Phi} \Phi)_{ABCD}
		\stackrel{\tau}{=}
			\tractorQ
				{ - (\ast \phi)_{ fgh}}
				{ - (\ast \theta)_{efgh}}
				{ (\ast \chi)_{ gh}}
				{ (\ast \psi)_{ fgh}}
		\in
			\Gamma\tractorQ
				{\Lambda^3 T^*M [4]}
				{\Lambda^4 T^*M [4]}
				{\Lambda^2 T^*M [2]}
				{\Lambda^3 T^*M [2]} .
\end{gather}
Computing in an adapted frame as in the proof of Proposition \ref{proposition:identites-g2-structure-components} yields some useful identities relating the components of $\Phi$ and their images under~$\ast$:
\begin{alignat}{3}
&	(\ast \phi)_{ fgh} = 3 \phi_{[fg} \theta_{h]}, 	\qquad 	&& (\ast \chi)_{ gh} = \theta^i \chi_{igh}, &\nonumber\\
&	(\ast \theta)_{efgh} = -3 \phi_{[ef} \psi_{gh]}, \qquad	&& (\ast \psi)_{ fgh} = 3 \psi_{[fg} \theta_{h]} . &\label{equation:Hodge-star-3-form-components}
\end{alignat}

\section[The global geometry of almost Einstein (2,3,5) distributions]{The global geometry of almost Einstein $\boldsymbol{(2, 3, 5)}$ distributions}
\label{section:global-geometry}

In this section we investigate the global geometry of $(2, 3, 5)$ distributions $(M, \mbD)$ that induce almost Einstein conformal structures $\mbc_{\mbD}$; naturally, we call such distributions themselves \textit{almost Einstein}.

Almost Einstein $(2, 3, 5)$ distributions are special among $(2, 3, 5)$ conformal structures: In a sense that can be made precise \cite[Theorem 1.2, Proposition 5.1]{GrahamWillse}, for a generic $(2, 3, 5)$ distribution $\mbD$ the holonomy of $\mbc_{\mbD}$ is equal to $\G_2$ and hence $\mbc_{\mbD}$ admits no nonzero almost Einstein scales.

Via the identif\/ication of the standard tractor bundles of $\mbD$ and $\mbc_{\mbD}$, Theorem \ref{theorem:almost-Einstein-bijection} gives that an oriented $(2, 3, 5)$ distribution is almost Einstein if\/f its standard tractor bundle~$\mcV$ admits a~nonzero parallel standard tractor $\bbS \in \Gamma(\mcV)$, or equivalently, if\/f it admits a holonomy reduction from~$\G_2$ to the stabilizer subgroup $S$ of a~nonzero vector in the standard representation~$\bbV$.

\subsection[Distinguishing a vector in the standard representation V of G2]{Distinguishing a vector in the standard representation $\boldsymbol{\bbV}$ of $\boldsymbol{\G_2}$}

In this subsection, let $\bbV$ denote the standard representation of $\G_2$ and $\Phi \in \Lambda^3 \bbV^*$ the corresponding $3$-form. We establish some of the algebraic consequences of f\/ixing a nonzero vector~$\bbS \in \bbV$.

\subsubsection{Stabilizer subgroups}

Recall from the introduction that the stabilizer group in $\G_2$ of $\bbS \in \bbV$ is as follows:

\begin{Proposition}\label{proposition:stabilizer-subgroup}
The stabilizer subgroup of a nonzero vector $\bbS$ in the standard representa\-tion~$\bbV$ of~$\G_2$ is isomorphic to:
\begin{enumerate}\itemsep=0pt
	\item[$1)$] $\SU(1, 2)$, if $\bbS$ is spacelike,
	\item[$2)$] $\SL(3, \bbR)$, if $\bbS$ is timelike, and
	\item[$3)$] $\SL(2, \bbR) \ltimes Q_+$, where $Q_+ < \G_2$ is the connected, nilpotent subgroup of $\G_2$ defined via Sections~{\rm \ref{subsubsection:parabolic-geometry}} and~{\rm \ref{subsubsection:235-distributions}}, if~$\bbS$ is isotropic.
\end{enumerate}
\end{Proposition}

\subsubsection[An varepsilon-Hermitian structure]{An $\boldsymbol{\varepsilon}$-Hermitian structure}\label{subsubsection:vareps-hermitian-structure}

Contracting a nonzero vector $\bbS \in \bbV$ with $\Phi$ determines an endomorphism:
\begin{gather}\label{equation:definition-K}
	\bbK^A{}_B := -\iota^2_7(\bbS)^A{}_B = -\bbS^C \Phi_C{}^A{}_B \in \mfso(3, 4) .
\end{gather}
We can identify $\bbK$ with the map $\bbT \mapsto \bbS \times \bbT$, so if we scale $\bbS$ so that $\varepsilon := -H_{\Phi}(\bbS, \bbS) \in \{-1, 0, 1\}$, identity \eqref{equation:iterated-cross-product} becomes
\begin{gather}\label{equation:K-squared-identity}
	\bbK^2 = \varepsilon {\id_{\bbV}} + \bbS \otimes \bbS^{\flat} .
\end{gather}
By skewness, $H_{AC} \bbS^A \bbK^C{}_B = -\bbS^A \bbS^D \bbK_{DAB} = 0$, so the image of $\bbK$ is contained in $\bbW := \langle \bbS \rangle^{\perp}$, and hence we can regard $\bbK\vert_{\bbW}$ as an endomorphism of $\bbW$, which by abuse of notation we also denote~$\bbK$. Restricting~\eqref{equation:K-squared-identity} to $\bbW$ gives that this latter endomorphism is an $\varepsilon$-complex structure on that bundle: $\bbK^2 = \varepsilon \id_{\bbW}$. Thus, $(H_{\Phi}\vert_{\bbW}, \bbK)$ is an \textit{$\varepsilon$-Hermitian structure on~$\bbW$}: this is a~pair~$(g, \bbK)$, where $g \in S^2 \bbW^*$ is a symmetric, nondegenerate, bilinear form and $\bbK$ is an $\varepsilon$-complex structure on $\bbW$ compatible in the sense that $g(\,\cdot\,, \bbK \,\cdot\,)$ is skew-symmetric. If~$\bbK$ is complex, $g$ has signature $(2p, 2q)$ for some integers $p$, $q$; if $\bbK$ is paracomplex, $g$ has signature $(m, m)$.

\subsubsection{Induced splittings and f\/iltrations}

If $\bbS$ is nonisotropic, it determines an orthogonal decomposition $\bbV = \bbW \oplus \langle \bbS \rangle$. If $\bbS$ is isotropic, it determines a f\/iltration $(\bbV^a_{\bbS})$ \cite[Proposition 2.5]{GPW}:
\begin{gather}\label{equation:isotropic-filtration}
	\begin{array}{@{}cccccccccccc@{}}
		^{-2} & & ^{-1} & & ^0 & & ^{+1} & & ^{+2} & & ^{+3} \\
		\eqmakebox[F]{$\bbV$} & \supset & \eqmakebox[F]{$\bbW$} & \supset & \eqmakebox[F]{$\im \bbK$} & \supset & \eqmakebox[F]{$\ker \bbK$} & \supset & \eqmakebox[F]{$\langle \bbS \rangle$} & \supset & \eqmakebox[F]{$\{ 0 \}$} \\
		_7 & & _6 & & _4 & & _3 & & _1 & & _0
	\end{array}
\end{gather}
The number above each f\/iltrand is its f\/iltration index $a$ (which are canonical only up to addition of a given integer to each index) and the number below its dimension. Moreover, $\im \bbK = (\ker \bbK)^{\perp}$ (so~$\ker \bbK$ is totally isotropic). If we take~$Q$ to be the stabilizer subgroup of the ray spanned by~$\bbS$, then the f\/iltration is $Q$-invariant, and checking the (representation-theoretic) weights of~$\bbV$ as a $Q$-representation shows that it coincides with the f\/iltration \eqref{equation:general-representation-filtration} determined by~$Q$. The map~$\bbK$ satisf\/ies $\bbK(\bbV_{\bbS}^a) = \bbV_{\bbS}^{a + 2}$, where we set $\bbV_{\bbS}^a = 0$ for all $a > 2$.

\subsubsection{The family of stabilized 3-forms}\label{subsubsection:stabilized-3-forms}

For nonzero $\bbS \in \bbV$, elementary linear algebra gives that the subspace of $3$-forms in $\Lambda^3 \bbV^*$ f\/ixed by the stabilizer subgroup $S$ of $\bbS$ has dimension $3$ and contains
\begin{gather}
	\Phi_I := \bbS \hook (\bbS^{\flat} \wedge \Phi)	\in \Lambda^3_1 \oplus \Lambda^3_{27}, \label{equation:definition-Phi-I} \\
	\Phi_J := -\iota^3_7(\bbS) = -\astPhi\big(\bbS^{\flat} \wedge \Phi\big) = \bbS \hook \astPhi \Phi\in \Lambda^3_7, \label{equation:definition-Phi-J} \\
	\Phi_K := \bbS^{\flat} \wedge (\bbS \hook \Phi)	\in \Lambda^3_1 \oplus \Lambda^3_{27} . \label{equation:definition-Phi-K}
\end{gather}
The containment $\Phi_K \in \Lambda^3_1 \oplus \Lambda^3_{27}$ follows from the fact that $\Phi_K = \tfrac{1}{2} i(\bbS^{\flat} \circ \bbS^{\flat})$, where $i$ is the $\G_2$-invariant map def\/ined in \eqref{equation:i}. The containment $\Phi_I \in \Lambda^3_1 \oplus \Lambda^3_{27}$ follows from that containment, the identity
\begin{gather}\label{equation:Phi-I-plus-Phi-K}
	\Phi_I + \Phi_K = \bbS \hook \big(\bbS^{\flat} \wedge \Phi\big) + \bbS^{\flat} \wedge (\bbS \hook \Phi) = H_{\Phi}(\bbS, \bbS) \Phi ,
\end{gather}
and the fact that $\Phi \in \Lambda^3_1$. It follows immediately from the def\/initions that
\begin{gather}\label{equation:contraction-S-PhiIJK}
	\bbS \hook \Phi_I = \bbS \hook \Phi_J = 0
		\qquad \textrm{and} \qquad
	\bbS \hook \Phi_K = H_{\Phi}(\bbS, \bbS) \bbS \hook \Phi .
\end{gather}
Since $\bbS$ annihilates $\Phi_I$ but not $\Phi$, the containments in \eqref{equation:definition-Phi-I}, \eqref{equation:definition-Phi-J} show that $\{\Phi, \Phi_I, \Phi_J\}$ is a~basis of the subspace of stabilized $3$-forms. If $H_{\Phi}(\bbS, \bbS) \neq 0$, then \eqref{equation:Phi-I-plus-Phi-K} implies that $\{\Phi_I, \Phi_J, \Phi_K\}$ is also a basis of that space. If $H_{\Phi}(\bbS, \bbS) = 0$ then $\Phi_K = -\Phi_I$. It is convenient to abuse notation and denote by $\Phi_I, \Phi_J$ the pullbacks to $\bbW$ of the $3$-forms of the same names via the inclusion $\bbW \hookrightarrow \bbV$.

For nonisotropic $\bbS$, def\/ine $\bbW^{1, 0} \subset \bbW \otimes_{\bbR} \bbC_{\varepsilon}$ to be the $(+i_{\varepsilon}$)-eigenspace of (the extension of)~$\bbK$, and an \textit{$\varepsilon$-complex volume form} to be an element of $\Lambda^m_{\bbC_{\varepsilon}} := \Lambda^m (\bbW^{1, 0})^*$.

\begin{Proposition}
\label{proposition:vareps-complex-volume-forms}
Suppose $\varepsilon := -H_{\Phi}(\bbS, \bbS) \in \{\pm 1\}$. For each $(A, B)$ such that $A^2 - \varepsilon B^2 = 1$,
\begin{gather*}
	\Psi_{(A, B)} := [A \Phi_I + \varepsilon B \Phi_J]	+ i_{\varepsilon} [B \Phi_I + A \Phi_J]	\in \Gamma\big(\Lambda^3_{\bbC_{\varepsilon}} \bbW\big)
\end{gather*}
is an $\varepsilon$-complex volume form for the $\varepsilon$-Hermitian structure $(H_{\Phi} \vert_{\bbW}, \bbK)$ on $\bbW$.
\end{Proposition}

\begin{Proposition}\label{proposition:epsilon-complex-volume-form-g2-structure}
Suppose $\bbV'$ is a $7$-dimensional real vector space and $H \in S^2 (\bbV')^*$ is a symmetric bilinear form of signature $(3, 4)$. Now, fix a vector $\bbS \in \bbV'$ such that $-\varepsilon := H(\bbS, \bbS) \in \{\pm 1\}$, denote $\bbW := \langle \bbS \rangle^{\perp}$, fix an $\varepsilon$-complex structure $\bbK \in \End(\bbW)$ such that $(H\vert_{\bbW}, \bbK)$ is a Hermitian structure on $\bbW$, and fix a compatible $\varepsilon$-complex volume form $\Psi \in \Lambda^3_{\bbC_{\varepsilon}} \bbW^*$ satisfying the normalization condition
\begin{gather*}
	\Psi \wedge \bar\Psi = -\tfrac{4}{3}i_{\varepsilon}\bbK \wedge \bbK\wedge \bbK.
\end{gather*}
Then, the $3$-form
\begin{gather*}
	\Re \Psi + \varepsilon \bbS^{\flat} \wedge \bbK \in \Lambda^3 (\bbV')^*
\end{gather*}
is a $\G_2$-structure on $\bbV'$ compatible with $H$. Here, $\Re \Psi$ and $\bbK$ are regarded as objects on $\bbV'$ via the decomposition $\bbV' = \bbW \oplus \langle \bbS \rangle$.
\end{Proposition}

This proposition can be derived, for example, from \cite[Proposition~1.12]{CLSS}, since, using the terminology of the article, $(\Re \Psi, \bbK)$ is a~compatible and normalized pair of stable forms.

\subsection[The canonical conformal Killing field xi]{The canonical conformal Killing f\/ield $\boldsymbol{\xi}$}
\label{subsection:conformal-Killing-field}

For this subsection, f\/ix an oriented $(2, 3, 5)$ distribution $\mbD$, let $\Phi \in \Gamma(\Lambda^3 \mcV^*)$ denote the corresponding parallel tractor $\G_2$-structure, and denote its components with respect to an arbitrary scale $\tau$ as in \eqref{equation:G2-structure-splitting}; in particular, $\phi := \Pi_0^{\Lambda^3 \mcV^*}(\Phi)$ is the underlying normal conformal Killing $2$-form. Also, f\/ix a nonzero almost Einstein scale $\sigma \in \Gamma(\mcE[1])$ of $\mbc_{\mbD}$, denote the cor\-respon\-ding parallel standard tractor by $\bbS := L_0^{\mcV}(\sigma)$, and denote its components with respect to $\tau$ as in~\eqref{equation:standard-tractor-structure-splitting}. By scaling, we assume that $-\varepsilon := H_{\Phi}(\bbS, \bbS) \in \{-1, 0, +1\}$.

We view the adjoint tractor
\begin{gather*}
	\bbK^A{}_B := -\bbS^C \Phi_C{}^A{}_B \in \Gamma(\mcA) .
\end{gather*}
as a bundle endomorphism of $\mcV$ (cf.~\eqref{equation:definition-K}), and computing gives that the components of $\bbK$ with respect to $\tau$ are
\begin{gather}\label{equation:components-of-K}
\bbK^A{}_B\stackrel{\tau}{=}	\tractorQ{\sigma \theta^a + \mu_c \phi^{ca}}{-\sigma \psi^a{}_b - \mu_c \chi^{ca}{}_b - \rho \phi^a{}_b}{\mu^c \theta_c}{\mu^c \psi_{ca} - \rho \theta_a} .
\end{gather}

We denote the projecting part of $\bbK^A{}_B$ by
\begin{gather}\label{equation:definition-xi}
	\xi^a := \Pi_0^{\mcA}(\bbK)^a = \sigma \theta^a + \mu_b \phi^{ba} \in \Gamma(TM) ;
\end{gather}
because $\bbK$ is parallel, $\xi$ is a normal conformal Killing f\/ield for $\mbc_{\mbD}$. By~\eqref{equation:K-squared-identity} $\bbK$ is not identically zero and hence neither is $\xi$. This immediately gives a simple geometric obstruction~-- nonexistence of a conformal Killing f\/ield~-- for the existence of an almost Einstein scale for an oriented $(2, 3, 5)$ conformal structure.

By construction, $\xi = \iota_7(\sigma)$, where $\iota_7$ is the manifestly invariant dif\/ferential operator $\iota_7 := \Pi_0^{\mcA} \circ (-\iota^2_7) \circ L_0^{\mcV} \colon \Gamma(\mcE[1]) \to \Gamma(TM)$. Here, $\iota^2_7$ is the bundle map $\mcV \to \Lambda^2 \mcV^*$ associated to the algebraic map~\eqref{equation:iota-2-7} of the same name, and we have implicitly raised an index with $H_{\Phi}$. Computing gives $\xi^a = \iota_7(\sigma)^a = -\phi^{ab} \sigma_{,b} + \tfrac{1}{4} \phi^{ab}{}_{,b} \sigma$.\footnote{This formula corrects a sign error in \cite[equation~(41)]{HammerlSagerschnig}, and~\eqref{equation:g2-conformal-Killing-field-decomposition} below corrects a corresponding sign error in equation~(40) of that reference.}

\begin{Proposition}\label{proposition:containment-xi-DD}
Given an oriented $(2, 3, 5)$ distribution $\mbD$, let $\phi$ denote the corresponding normal conformal Killing $2$-form, and suppose the induced conformal class $\mbc_{\mbD}$ admits an almost Einstein scale $\sigma$. The corresponding vector field $\xi := \iota_7(\sigma)$ is a section of $[\mbD, \mbD]$.
\end{Proposition}
\begin{proof}
By \eqref{equation:definition-xi}, $\phi_{ba} \xi^b = \phi_{ba} (\sigma \theta^b + \mu_c \phi^{cb}) = \sigma \phi_{ba} \theta^b + \mu_c \phi^{cb} \phi_{bc}$, but the f\/irst and second term vanish respectively by Proposition~\ref{proposition:identites-g2-structure-components}(1),(3). Thus, $\xi \in \ker \phi$, which by Part~(2) of that proposition is $[\mbD, \mbD]$.
\end{proof}

On the set $M_{\xi} := \{x \in M \colon \xi_x \neq 0\}$, $\xi$ spans a canonical line f\/ield
\begin{gather*}
\mbL := \langle \xi \rangle\vert_{M_{\xi}} ,
\end{gather*} and by Proposition \ref{proposition:containment-xi-DD}, $\mbL$ is a subbundle of $[\mbD, \mbD]\vert_{M_{\xi}}$. Henceforth we often suppress the restriction notation $\vert_{M_{\xi}}$. We will see in Proposition~\ref{proposition:coincidence-line-fields} that~$\mbL$ coincides with the line f\/ield of the same name determined via Proposition~\ref{proposition:identites-g2-structure-components} by the preferred scale~$\sigma$ (on the complement of its zero locus).

\subsection[Characterization of conformal Killing fields induced by almost Einstein scales]{Characterization of conformal Killing f\/ields induced \\ by almost Einstein scales}\label{subsection:symmetry-decomposition}

Hammerl and Sagerschnig showed that for any oriented $(2, 3, 5)$ distribution $\mbD$, the Lie algebra $\mfaut(\mbc_{\mbD})$ of conformal Killing f\/ields of the induced conformal structure $\mbc_{\mbD}$ admits a natural (vector space) decomposition, corresponding to the $\G_2$-module decomposition $\mfso(3, 4) \cong \mfg_2 \oplus \bbV$ into irreducible submodules, that encodes features of the geometry of the underlying distribution.

Given an oriented $(2, 3, 5)$ distribution $M$, a vector f\/ield $\eta \in \Gamma(TM)$ is an \textit{infinitesimal symmetry} of $\mbD$ if\/f $\mbD$ is invariant under the f\/low of $\eta$, and the inf\/initesimal symmetries of $\mbD$ comprise a Lie algebra $\mfaut(\mbD)$ under the usual Lie bracket of vector f\/ields. The construction $\mbD \rightsquigarrow \mbc_{\mbD}$ is functorial, so $\mfaut(\mbD) \subseteq \mfaut(\mbc_{\mbD})$.

By construction, the map $\pi_7$ given in \eqref{equation:pi-7} below is a left inverse for $\iota_7$, so in particular $\iota_7$ is injective.

\begin{Theorem}[{\cite[Theorem~B]{HammerlSagerschnig}}] \label{theorem:conformal-Killing-field-decomposition}
If $(M, \mbD)$ is an oriented $(2, 3, 5)$ distribution, the Lie algebra $\mfaut(\mbc_{\mbD})$ of conformal Killing fields of the induced conformal structure $\mbc_{\mbD}$ admits a natural $($vector space$)$ decomposition
\begin{gather}\label{equation:g2-conformal-Killing-field-decomposition}
	\mfaut(\mbc_{\mbD}) = \mfaut(\mbD) \oplus \iota_7(\aEs(\mbc_{\mbD}))
\end{gather}
and hence an isomorphism $\mfaut(\mbc_{\mbD}) \cong \mfaut(\mbD) \oplus \aEs(\mbc_\mbD)$.

The projection $\mfaut(\mbc_{\mbD}) \to \aEs(\mbc_{\mbD})$ is $($the restriction of$)$ the invariant differential operator $\pi_7 := \Pi_0^{\mcV} \circ (-\pi^2_7) \circ L_0^{\mcA} \colon \Gamma(TM) \to \Gamma(\mcE[1])$, which is given by
\begin{gather}\label{equation:pi-7}
	\pi_7 \colon \ \eta^a \mapsto \tfrac{1}{6} \phi^{ab} \eta_{a, b} - \tfrac{1}{12} \phi_{ab,}{}^b \eta^a .
\end{gather}

The canonical projection $\mfaut(\mbc_{\mbD}) \to \mfaut(\mbD)$ is $($the restriction of$)$ the invariant differential operator $\pi_{14} := \id_{\Gamma(TM)} - \iota_7 \circ \pi_7 \colon \Gamma(TM) \to \Gamma(TM)$.
\end{Theorem}

The map $\pi^2_7$ is the bundle map $\mcA \cong \Lambda^2 \mcV^* \to \mcV^*$ associated to the algebraic map~\eqref{equation:pi-2-7} of the same name.

The conformal Killing f\/ields in the distinguished subspace $\iota_7(\aEs(\mbc_{\mbD}))$, that is, those corresponding to almost Einstein scales, admit a simple geometric characterization:

\begin{Proposition}
Let $(M, \mbD)$ be an oriented $(2, 3, 5)$ distribution. Then, a conformal Killing field of $\mbc_{\mbD}$ is in the subspace $\iota_7(\aEs(\mbc_{\mbD}))$ iff it is a section of $[\mbD, \mbD]$.

Hence, the indicated restrictions of $\iota_7$ and $\pi_7$ comprise a natural bijective correspondence
\begin{gather*}
	\aEs(\mbc_{\mbD}) \mathrel{\mathop{\rightleftarrows}^{\iota_7}_{\pi_7}} \mfaut(\mbc_{\mbD}) \cap \Gamma([\mbD, \mbD]) .
\end{gather*}
\end{Proposition}

\begin{proof}
Let $q_{-3}$ denote the canonical projection $TM \to TM / [\mbD, \mbD]$ and the map on sections it induces. It follows from a general fact about inf\/initesimal symmetries of parabolic geometries \cite{Cap} that an inf\/initesimal symmetry $\xi$ of $\mbD$ can be recovered from its image $q_{-3}(\xi) \in \Gamma(TM / [\mbD, \mbD])$ via a natural linear dif\/ferential operator $\Gamma(TM / [\mbD, \mbD]) \to \Gamma(TM)$; in particular, if $q_{-3}(\xi) = 0$ then $\xi = 0$, so $\mfaut(\mbD)$ intersects trivially with $\ker q_{-3}$. On the other hand, Proposition \ref{proposition:containment-xi-DD} gives that the image of $\iota_7$ is contained in $\ker q_{-3} = [\mbD, \mbD]$. The claim now follows from the decomposition in Theorem \ref{theorem:conformal-Killing-field-decomposition}.
\end{proof}

\subsection[The weighted endomorphisms $I$, $J$, $K$]{The weighted endomorphisms $\boldsymbol{I}$, $\boldsymbol{J}$, $\boldsymbol{K}$}\label{subsection:I-J-K}

Since they are algebraic combinations of parallel tractors, the $3$-forms $\Phi_I, \Phi_J, \Phi_K \in \Gamma(\Lambda^3 \mcV^*)$ respectively def\/ined pointwise by \eqref{equation:definition-Phi-I}, \eqref{equation:definition-Phi-J}, \eqref{equation:definition-Phi-K} are themselves parallel. Thus, their respective projecting parts, $I_{ab} := \Pi_0^{\Lambda^3 \mcV^*}(\Phi_I)_{ab}$, $J_{ab} := \Pi_0^{\Lambda^3 \mcV^*}(\Phi_J)_{ab}$, $K_{ab} := \Pi_0^{\Lambda^3 \mcV^*}(\Phi_K)_{ab}$, are normal conformal Killing $2$-forms.

The def\/initions of $\Phi_I$, $\Phi_J$, $\Phi_K$, together with \eqref{equation:Hodge-star-3-form} and \eqref{equation:Hodge-star-3-form-components} give
\begin{gather}
	I_{ab} =-\sigma^2 \psi_{ab} - \sigma \mu^c \chi_{cab}- 2 \sigma \mu_{[a} \theta_{b]} + \sigma \rho \phi_{ab}+ 3 \mu^c \mu_{[c} \phi_{ab]}m,
				\label{equation:I} \\
	J_{ab} =	-\sigma \theta^c \chi_{cab} + 3 \mu^c \phi_{[ca} \theta_{b]},	\label{equation:J}\\
	K_{ab} =	\sigma^2 \psi_{ab} + \sigma \mu^c \chi_{cab}+ 2 \sigma \mu_{[a} \theta_{b]} + \sigma \rho \phi_{ab}
				- 2 \mu^c \mu_{[a} \phi_{b] c}.			\label{equation:K}
\end{gather}
Using the splitting operators $L_0^{\mcV}$ \eqref{equation:splitting-operator-standard} and $L_0^{\Lambda^3 \mcV^*}$ \eqref{equation:splitting-operator-alternating}, we can write these normal conformal Killing $2$-forms as dif\/ferential expressions in $\phi$ and $\sigma$:
\begin{gather}
I_{ab}= \tfrac{1}{5} \sigma^2 \left(\tfrac{1}{3} \phi_{ab, c}{}^c + \tfrac{2}{3} \phi_{c[a, b]}{}^c + \tfrac{1}{2} \phi_{c [a,}{}^c{}_{b]} + 4 \sfP^c{}_{[a} \phi_{b] c} \right) \nonumber \\
\hphantom{I_{ab}=}{} - \sigma \sigma^{,c} \phi_{[ca, b]} - \tfrac{1}{2} \sigma \sigma_{, [a} \phi_{b] c,}{}^c - \tfrac{1}{5} \sigma \sigma_{,c}{}^c \phi_{ab} + 3 \sigma^{,c} \sigma_{,[c} \phi_{ab]}, \label{equation:I-sigma-phi} \\
	J_{ab}
		= -\tfrac{1}{4} \sigma \phi^{cd,}{}_d \phi_{[ab, c]} + \tfrac{3}{4} \sigma^{,c} \phi_{[ab} \phi_{c]d,}{}^d, \label{equation:J-sigma-phi} \\
K_{ab}= -\tfrac{1}{5} \sigma^2 \left(\tfrac{1}{3} \phi_{ab, c}{}^c + \tfrac{2}{3} \phi_{c [a, b]}{}^c + \tfrac{1}{2} \phi_{c [a,}{}^c{}_{b]} + 4 \sfP^c{}_{[a} \phi_{b] c} + 2 \sfP^c{}_c \phi_{ab}\right) \nonumber \\
\hphantom{K_{ab}=}{} + \sigma \sigma^{,c} \phi_{[ab, c]} + \tfrac{1}{2} \sigma \sigma_{,[a} \phi_{b] c,}{}^c - \tfrac{1}{5} \sigma \sigma_{, c}{}^c \phi_{ab} - 2 \sigma^{,c} \sigma_{, [a} \phi_{b] c} . \label{equation:K-sigma-phi}
\end{gather}

Raising indices gives weighted $\mbg$-skew endomorphisms $I^a{}_b, J^a{}_b, K^a{}_b \in \Gamma(\End_{\skewOp}(TM)[1])$.

\subsection{The (local) leaf space}\label{subsection:local-leaf-space}

Let $L$ denote the space of integral curves of $\xi$ in $M_{\xi} := \{x \in M \colon \xi_x \neq 0\}$, and denote by $\pi_L\colon M_{\xi} \to L$ the projection that maps a point to the integral curve through it. Since $\xi$ vanishes nowhere, around any point in $M_{\xi}$ there is a neighborhood such that the restriction of~$\pi_L$ thereto is a trivial f\/ibration over a smooth $4$-manifold; henceforth in this subsection, we will assume we have replaced $M_{\xi}$ with such a neighborhood.

\subsubsection{Descent of the canonical objects}\label{subsubsection:descent}

Some of the objects we have already constructed on $M$ descend to $L$ via the projection $\pi_L$. One can determine which do by computing the Lie derivatives of the various tensorial objects with respect to the generating vector f\/ield $\xi$, but again it turns out to be much more ef\/f\/icient to compute derivatives in the tractor setting. Since any conformal tractor bundle $\mcU \to M$ is a~natural bundle in the category of conformal manifolds, one may pull back any section $\bbA \in \Gamma(\mcU)$ by the f\/low $\Xi_t$ of $\xi$ and def\/ine the Lie derivative $\mcL_{\xi} \bbA$ to be $\mcL_{\xi} \bbA := \partial_t\vert_0 \Xi_t^* \bbA$ \cite{KMS}. Since the tractor projection $\Pi_0^{\mcU}$ is associated to a canonical vector space projection, it commutes with the Lie derivative. We exploit the following identity:

\begin{Lemma}[{\cite[Appendix~A.3]{CurryGover}}]\label{lemma:Lie-derivative-tractor}
Let $(M, \mbc)$ be a conformal structure of signature $(p, q)$, \mbox{$p + q \geq 3$}, let $\bbU$ be a $\SO(p + 1, q + 1)$-representation, and denote by $\mcU \to M$ the tractor bundle it induces. If $\xi \in \Gamma(TM)$ is a conformal Killing field for $\mbc$, and $\bbA \in \Gamma(\mcU)$, then $\mcL_{\xi} \bbA = \nabla_{\xi} \bbA - L_0^{\mcA}(\xi) \cdot \bbA ,$ where $\cdot$ denotes the action on sections induced by the action $\mfso(p + 1, q + 1) \times \bbU \to \bbU$. In particular, if $\bbA$ is parallel, then
\begin{gather}\label{equation:Lie-derivative-parallel-tractor}
	\mcL_{\xi} \bbA = - L_0^{\mcA}(\xi) \cdot \bbA.
\end{gather}
\end{Lemma}

\begin{Proposition}\label{proposition:descent}
Suppose $(M, \mbD)$ is an oriented $(2, 3, 5)$ distribution and let $\phi \in \Gamma(\Lambda^2 T^*M [3])$ denote the corresponding normal conformal Killing $2$-form. Suppose moreover that the conformal structure $\mbc_{\mbD}$ induced by $\mbD$ admits a nonzero almost Einstein scale $\sigma \in \Gamma(\mcE[1])$, let $\xi \in \Gamma(TM)$ denote the corresponding normal conformal Killing field, and let $I$, $J$, $K$ denote the normal conformal Killing $2$-forms defined in Section~{\rm \ref{subsection:I-J-K}}. Then,
\begin{gather}
	 \mcL_{\xi} \sigma = 0 ,\qquad (\mcL_{\xi} \phi)_{ab} = 3 J_{ab},\qquad	(\mcL_{\xi} I)_{ab} = -3 \varepsilon J_{ab},\nonumber\\
	(\mcL_{\xi} J)_{ab} = 3 I_{ab},\qquad	(\mcL_{\xi} K)_{ab} = 0 .\label{equation:Lie-derivatives-objects}
\end{gather}
As usual, we scale $\sigma$ so that $\varepsilon := -H_{\Phi}(\bbS, \bbS) \in \{-1, 0, +1\}$, where $\bbS^A := L_0^{\mcV}(\sigma)^A$ is the parallel standard tractor corresponding to $\sigma$.

In particular, $\sigma$ and $K$ descend via $\pi_L \colon M_{\xi} \to N$ to well-defined objects $\hat{\sigma}$ and $\hat{K}$, but~$\phi$ and~$J$ do not descend, and when $\varepsilon \neq 0$ neither does $I$ (recall that when $\varepsilon = 0$, $I = -K$).
\end{Proposition}

\begin{proof}
As usual, denote $\smash{\bbK := L_0^{\mcA}(\xi)}$, $\smash{\Phi := L_0^{\Lambda^3 \mcV^*}(\phi)}$, $\smash{\Phi_I := L_0^{\Lambda^3 \mcV^*}(I)}$, $\smash{\Phi_J := L_0^{\Lambda^3 \mcV^*}(J)}$, and $\smash{\Phi_K := L_0^{\Lambda^3 \mcV^*}(K)}$. Since $\xi$ is a conformal Killing f\/ield, by \eqref{equation:Lie-derivative-parallel-tractor} $(\mcL_{\xi} \bbS)^A = -(L_0^{\mcA}(\xi) \cdot \bbS)^A = -\bbK^A{}_B \bbS^B = -\bbS^C \Phi_C{}^A{}_B \bbS^B = 0$. Applying $\Pi_0^{\mcV}$ yields $\mcL_{\xi} \sigma = \mcL_{\xi} \Pi_0^{\mcV}(\bbS) = \Pi_0^{\mcV}(\mcL_{\xi} \bbS) = \Pi_0^{\mcV}(0) = 0$.

The proofs for $J$ and $\phi$ are similar, and use the identities \eqref{equation:contraction-Phi-Phi}, \eqref{equation:contraction-Phi-PhiStar}.

By def\/inition, $\mcL_{\xi} \Phi_K = \mcL_{\xi} [\bbS^{\flat} \wedge (\bbS \hook \Phi)]$, and since $\mcL_{\xi} \bbS = 0$, we have $\mcL_{\xi} \Phi_K = \bbS^{\flat} \wedge (\bbS \hook \mcL_{\xi} \Phi) = \bbS^{\flat} \wedge [\bbS \hook (3 \Phi_J)]$, but by \eqref{equation:definition-Phi-J} this is $3 \bbS^{\flat} \wedge [\bbS \hook (\bbS \hook \astPhi \Phi)],$ which is zero by symmetry. App\-lying~$\smash{\Pi_0^{\Lambda^3 \mcV^*}}$ gives $\mcL_{\xi} K = 0$.

Finally, \eqref{equation:Phi-I-plus-Phi-K} gives $\mcL_{\xi} \Phi_I = \mcL_{\xi} (-\varepsilon \Phi - \Phi_K) = -\varepsilon \mcL_{\xi} \Phi - \mcL_{\xi} \Phi_K = -\varepsilon (3 \Phi_J) - (0) = -3 \varepsilon \Phi_J$, and applying $\smash{\Pi_0^{\Lambda^3 \mcV^*}}$ gives $\mcL_{\xi} I = -3 \varepsilon J$.
\end{proof}

\section{The conformal isometry problem}
\label{section:conformally-isometric}

In this section we consider the problem of determining when two distributions $(M, \mbD)$ and $(M, \mbD')$ induce the same oriented conformal structure; we say two such distributions are \textit{conformally isometric}. This problem turns out to be intimately related to existence of a nonzero almost Einstein scale for $\mbc_{\mbD}$.

Approaching this question at the level of underlying structures is prima facie dif\/f\/icult: The value of the conformal structure $\mbc_{\mbD}$ at a point $x \in M$ induced by a $(2, 3, 5)$ distribution~$(M, \mbD)$ depends on the $4$-jet of $\mbD$ at $x$ \cite[equation~(54)]{Nurowski} (or, essentially equivalently, multiple prolongations and normalizations), so analyzing directly the dependence of $\mbc_{\mbD}$ on $\mbD$ involves apprehending high-order dif\/ferential expressions that turn out to be cumbersome.

We have seen that in the tractor bundle setting, however, this construction is essentially algebraic: At each point, the parallel tractor $\G_2$-structure $\Phi \in \Gamma(\Lambda^3 \mcV^*)$ determined by an oriented $(2, 3, 5)$ distribution $(M, \mbD)$ determines the parallel tractor bilinear form $H_{\Phi} \in \Gamma(S^2 \mcV^*)$ and orientation $[\epsilon_{\Phi}]$ canonically associated to the oriented conformal structure $\mbc_{\mbD}$. So, the problem of determining the distributions $(M, \mbD')$ such that $\mbc_{\mbD'} = \mbc_{\mbD}$ amounts to the corresponding algebraic problem of identifying for a $\G_2$-structure $\Phi$ on a $7$-dimensional real vector space $\bbV$ the $\G_2$-structures $\Phi'$ on $\bbV$ such that $(H_{\Phi'}, [\epsilon_{\Phi'}]) = (H_{\Phi}, [\epsilon_{\Phi}])$. We solve this algebraic problem in Section~\ref{subsection:compatible-G2-structures} and then transfer the result to the setting of parallel sections of conformal tractor bundles to resolve the conformal isometry problem in Section~\ref{subsection:conformally-isometric-235-distributions}.

\subsection[The space of G2-structures compatible with an SO(3,4)-structure]{The space of $\boldsymbol{\G_2}$-structures compatible with an $\boldsymbol{\SO(3,4)}$-structure}\label{subsection:compatible-G2-structures}

In this subsection, which consists entirely of linear algebra, we characterize explicitly the space of $\G_2$-structures compatible with a given $\SO(3, 4)$-structure on a $7$-dimensional real vector space~$\bbV$, or more precisely, the $\SO(3, 4)$-structure determined by a reference $\G_2$-structure $\Phi$. This characterization is essentially equivalent to that in \cite[Remark~4]{Bryant} for the analogous inclusion of the compact real form of $\G_2$ into $\SO(7, \bbR)$. The following proposition can be readily verif\/ied by computing in an adapted frame. (Computer assistance proved particularly useful in this verif\/ication.)

\begin{Proposition}\label{proposition:compatible-g2-structures}
Let $\bbV$ be a $7$-dimensional real vector space and fix a $\G_2$-structure $\Phi \in \Lambda^3 \bbV^*$.
\begin{enumerate}\itemsep=0pt
	\item[$1.$] Fix a nonzero vector $\bbS \in \bbV$; by rescaling we may assume that $\varepsilon := - H_{\Phi}(\bbS, \bbS) \in \{{-}1{,} 0{,} {+}1\}$. For any $(\bar A, B) \in \bbR^2$ such that $-\varepsilon \bar A^2 + 2 \bar A + B^2 = 0$ $($there is a $1$-parameter family of such pairs$)$ the $3$-form
		\begin{gather}\label{equation:family-G2-structures}
			\Phi' := \Phi + \bar A \Phi_I + B \Phi_J \in \Lambda^3 \bbV^*
		\end{gather}
		is a $\G_2$-structure compatible with the $\SO(3, 4)$-structure $(H_{\Phi}, [\epsilon_{\Phi}])$, that is, $(H_{\Phi'}, [\epsilon_{\Phi'}]) = (H_{\Phi}, [\epsilon_{\Phi}])$.
	\item[$2.$] Conversely, all compatible $\G_2$-structures arise this way: If a $\G_2$-structure $\Phi'$ on $\bbV$ satisfies $(H_{\Phi}, [\epsilon_{\Phi}]) = (H_{\Phi'}, [\epsilon_{\Phi'}])$, there is a vector $\bbS \in \bbV$ $($we may assume that $\varepsilon := -H_{\Phi}(\bbS, \bbS) \in \{-1, 0, +1\})$ and $(\bar A, B) \in \bbR^2$ satisfying $-\varepsilon \bar A^2 + 2 \bar A + B^2 = 0$ such that $\Phi'$ is given by~\eqref{equation:family-G2-structures}.
\end{enumerate}
\end{Proposition}

\begin{Remark}\label{remark:determined-by-S}
Let $\bbV$ be the standard representation of $\SO(3, 4)$, and let $S$ denote the intersection of (a copy of) $\G_2$ in $\SO(3, 4)$ and the stabilizer subgroup in $\SO(3, 4)$ of a nonzero vector $\bbS \in \bbV$. By Section~\ref{subsubsection:stabilized-3-forms} the space of $3$-forms in $\Lambda^3 \bbV^*$ stabilized by $S$ is $\langle \Phi, \Phi_I, \Phi_J \rangle$. By Proposition~\ref{proposition:compatible-g2-structures}, this determines a $1$-parameter family of copies of $\G_2$ containing $S$ and contained in $\SO(3, 4)$, or equivalently, a $1$-parameter family $\mcF$ of $\G_2$-structures $\Phi'$ compatible with the $\SO(3, 4)$-structure, but $S$ does not distinguish a $\G_2$-structure in this family.
\end{Remark}

\begin{Remark}\label{remark:homogeneous-space-compatible-G2-structures}
We may identify the space of $\G_2$-structures that induce a particular $\SO(3, 4)$-structure with the homogeneous space $\SO(3, 4) / \G_2$. Since $\SO(3, 4)$ has two components but $\G_2$ is connected, this homogeneous space has two components; the $\G_2$-structures in one determine the opposite space and time orientations as those in the other. We can identify one component with the projectivization $\bbP(\bbS^{3, 4})$ of the cone of spacelike elements in $\bbR^{4, 4}$ and the other as the projectivization $\bbP(\bbS^{4, 3})$ of the cone of timelike elements, and the homogeneous space $\SO(3, 4) / \G_2$ (the union of these projectivizations) as the complement of the neutral null quadric in $\bbP(\bbR^{4, 4})$ \cite[Theorem 2.1]{Kath}.
\end{Remark}

Henceforth denote by $\mcF[\Phi; \bbS]$ the $1$-parameter family of $\G_2$-structures compatible with the $\SO(3, 4)$-structure $(H_{\Phi}, [\epsilon_{\Phi}])$ def\/ined by Proposition \ref{proposition:compatible-g2-structures}. By construction, if $\Phi' \in \mcF[\Phi; \bbS]$, then $\mcF[\Phi'; \bbS] = \mcF[\Phi; \bbS]$.

\begin{Proposition}\label{proposition:compatible-3-form-Xi}
For any $\G_2$-structure $\Phi'\! \in\! \mcF[\Phi; \bbS]$, the endomorphism $(\bbK')^A{}_B := -\bbS^C (\Phi')_C{}^A{}_B\!$ coincides with $\bbK^A{}_B := -\bbS^C \Phi_C{}^A{}_B$.
\end{Proposition}

\begin{proof}Proposition \ref{proposition:compatible-g2-structures} gives that $\Phi' = \Phi + \bar A \Phi_I + B \Phi_J$ for some constants $\bar A, B$, so \eqref{equation:contraction-S-PhiIJK} gives that $\bbK' := -\bbS \hook \Phi' = -\bbS \hook (\Phi + \bar A \Phi_I + B \Phi_J) = -\bbS \hook \Phi = \bbK$.
\end{proof}

We can readily parameterize the families $\mcF[\Phi; \bbS]$ of $\G_2$-structures. It is convenient henceforth to split cases according to the causality type of $\bbS$, that is, according to $\varepsilon$. If $\varepsilon \neq 0$, then $\Phi = -\varepsilon (\Phi_I + \Phi_K)$, so in terms of $A := \bar A - \varepsilon$ and $B$, $\Phi' = A \Phi_I + B \Phi_J - \varepsilon \Phi_K$ and the condition on the coef\/f\/icients is $A^2 - \varepsilon B^2 = 1$.

If $\varepsilon = -1$, then $A^2 + B^2 = 1$, and so we can parameterize $\mcF[\Phi; \bbS]$ by{\samepage
\begin{gather}\label{equation:parameterization-Phi-upsilon}
	\Phi_{\upsilon} := (\cos \upsilon) \Phi_I + (\sin \upsilon) \Phi_J + \Phi_K .
\end{gather}
The parameterization descends to a bijection $\bbR / 2\pi \bbZ \cong \bbS^1 \leftrightarrow \mcF[\Phi; \bbS]$, and $\Phi_0 = \Phi$.}

If $\varepsilon = +1$, then $A^2 - B^2 = 1$, and so we can parameterize $\mcF[\Phi; \bbS]$ by
\begin{gather}\label{equation:parameterization-Phi-t}
	\Phi_t^{\mp} := (\mp \cosh t) \Phi_I + (\sinh t) \Phi_J - \Phi_K
\end{gather}
and $\Phi_0^- = \Phi$.

 If $\varepsilon = 0$, the compatibility conditions simplify to $2 \bar A + B^2 = 0$, so we can parameterize $\mcF[\Phi; \bbS]$ by
\begin{gather}\label{equation:parameterization-Phi-s}
	\Phi_s := \Phi - \tfrac{1}{2} s^2 \Phi_I + s \Phi_J
\end{gather}
and $\Phi_0 = \Phi$.

Each of the above parameterizations $\Phi_u$ satisf\/ies $\smash{\left.\frac{d}{du}\right\vert_0 \Phi_u = \Phi_J}$, and the parameterizations in the latter two cases are bijective. In the nonisotropic cases, we can encode the compatible $\G_2$-structures ef\/f\/iciently in terms of the $\varepsilon$-complex volume forms in Proposition \ref{proposition:vareps-complex-volume-forms}.

\begin{Proposition}
Suppose $\bbS$ is nonisotropic, and let $\Psi_{(A, B)} \in \Gamma(\Lambda^3_{\bbC_{\varepsilon}} \bbW)$, $A^2 - \varepsilon B^2 = 1$, denote the corresponding $1$-parameter family of $\varepsilon$-complex volume forms defined pointwise in Proposition~{\rm \ref{proposition:vareps-complex-volume-forms}}. Then, $\mcF[\Phi; \bbS]$ consists of the $\G_2$-structures
\begin{gather*}
	\Phi_{(A, B)} := \Pi_{\bbW}^* \Re \Psi_{(A, B)} + \varepsilon \Phi_K ,
\end{gather*}
$A^2 - \varepsilon B^2 = 1$. $($Here, $\Pi_{\bbW}$ is the orthogonal projection $\bbV \to \bbW.)$
\end{Proposition}
\begin{proof}
This follows immediately from the appearance of the condition $A^2 - \varepsilon B^2 = 1$ in the discussion after the proof of Proposition~\ref{proposition:compatible-3-form-Xi} and the form of $\Psi_{(A, B)}$.
\end{proof}

\subsection[Conformally isometric (2,3,5) distributions]{Conformally isometric $\boldsymbol{(2, 3, 5)}$ distributions}\label{subsection:conformally-isometric-235-distributions}

We now transfer the results of Proposition~\ref{proposition:compatible-g2-structures} to the level of parallel sections of conformal tractor bundles, thereby proving Theorem~B:

\begin{proof}[Proof of Theorem~B]
The oriented $(2, 3, 5)$ distributions $\mbD'$ conformally isometric to $\mbD$ are precisely those for which the corresponding parallel tractor $3$-form $\Phi$ is compatible with the parallel tractor metric $H_{\Phi}$ and orientation $[\epsilon_{\Phi}]$. Transferring the content of Proposition~\ref{proposition:compatible-g2-structures} to the tractor setting gives that these are precisely the $3$-forms $\Phi' := \Phi + \bar A \Phi_I + B \Phi_J \in \Gamma(\Lambda^3 \mcV^*)$, and applying $\Pi_0^{\Lambda^3 \mcV^*}$ \eqref{equation:projection-operator-alternating} gives that the underlying normal conformal Killing $2$-forms are $\phi' := \phi + \bar A I + B J \in \Gamma(\Lambda^2 T^* M[3])$. Substituting for~$I$,~$J$ respectively using \eqref{equation:I-sigma-phi} and~\eqref{equation:J-sigma-phi} yields the formula~\eqref{equation:family-conformal-Killing-2-forms}.
\end{proof}

We reuse the notation $\mcF[\Phi; \bbS]$ for the $1$-parameter family of parallel tractor $\G_2$-structures def\/ined pointwise by~\eqref{equation:family-G2-structures} by a parallel tractor $3$-form $\Phi$ and a parallel, nonzero standard tractor~$\bbS$. By analogy, we denote by $\mcD[\mbD; \sigma]$ the $1$-parameter family of conformally isometric oriented $(2, 3, 5)$ distributions determined by $\mbD$ and $\sigma$ as in Theorem~B; we say that the distributions in the family are \textit{related by $\sigma$}. Again, if $\mbD' \in \mcD[\mbD; \sigma]$, then by construction $\mcD[\mbD'; \sigma] = \mcD[\mbD; \sigma]$. Henceforth, $\mcD$ denotes a family $\mcD[\mbD; \sigma]$ for some $\mbD$ and $\sigma$.

\begin{Proposition}\label{proposition:conformal-Killing-field-distribution-family}\looseness=-1
Let $\mbD$ be an oriented $(2, 3, 5)$ distribution and $\sigma$ an almost Einstein scale of $\mbc_{\mbD}$. For any $\mbD' \in \mcD[\mbD; \sigma]$, the conformal Killing fields~$\xi$ and~$\xi'$ respectively determined by $(\mbD, \sigma)$ and $(\mbD', \sigma)$ coincide. In particular, $\xi$, the line $\mbL \subset TM \vert_{M_{\xi}}$ its restriction spans, and its orthogonal hyperplane field $\mbC := \mbL^{\perp} \subset TM \vert_{M_{\xi}}$ depend only on the family $\mcD[\mbD; \sigma]$ and not on~$\mbD$.
\end{Proposition}

\begin{proof}Let $\Phi, \Phi'$ denote the parallel tractor $\G_2$-structures corresponding respectively to $\mbD, \mbD'$. Translating Proposition~\ref{proposition:compatible-3-form-Xi} to the tractor bundle setting gives that $\bbK' := -L_0^{\mcV}(\sigma) \hook \Phi'$ and $\bbK := -L_0^{\mcV}(\sigma) \hook \Phi$ coincide, and hence $\xi' = \Pi_0^{\mcA}(\bbK') = \Pi_0^{\mcA}(\bbK) = \xi$.
\end{proof}

\begin{Corollary}\label{corollary:containment-family-hyperplane-distribution}
Let $\mcD$ be a $1$-parameter family of conformally isometric oriented $(2, 3, 5)$ distributions related by an almost Einstein scale. Every distribution $\mbD \in \mcD$ satisfies $\mbD \subset \mbC$.
\end{Corollary}

\subsubsection{Parameterizations of conformally isometric distributions}\label{subsubsection:parameterizations-isometric}

Now, given an oriented $(2, 3, 5)$ distribution $\mbD$ and a nonzero almost Einstein scale $\sigma$ of $\mbc_{\mbD}$, we can explicitly parameterize the family $\mcD[\mbD; \sigma]$ they determine by passing to the projecting parts $\smash{\phi' := \Pi_0^{\Lambda^3 \mcV^*}(\Phi')}$ of the corresponding parallel tractor $\G_2$-structures $\Phi' \in \mcF[\Phi; \bbS]$. To do so, it is convenient to split cases according to the sign of the Einstein constant \eqref{equation:Einstein-constant} of $\sigma$. As usual we denote by $\Phi \in \Gamma(\Lambda^3 \mcV^*)$ the parallel tractor $\G_2$-structure corresponding to~$\mbD$ and scale~$\sigma$ (by a~constant) so that $\bbS := L_0^{\mcV}(\sigma)$ satisf\/ies $\varepsilon := -H_{\Phi}(\bbS, \bbS) \in \{-1, 0, +1\}$.

For $\varepsilon = -1$, $\mcD[\mbD; \bbS]$ consists of the distributions $\mbD_{\upsilon}$ corresponding to the normal conformal Killing $2$-forms
\begin{gather}\label{equation:phi-parameterization-Ricci-negative}
	\phi_{\upsilon}
		:= \Pi_0^{\Lambda^3 \mcV^*}(\Phi_{\upsilon}) = (\cos \upsilon) I + (\sin \upsilon) J + K .
\end{gather}
As for the corresponding family $\mcF[\Phi; \bbS]$ of parallel tractor $\G_2$-structures, this parameterization descends to a bijection $\bbR / 2\pi \bbZ \cong \bbS^1 \leftrightarrow \mcD[\mbD; \sigma]$.

For $\varepsilon = +1$, $\mcD[\mbD; \bbS]$ consists of the distributions $\mbD_t^{\mp}$ corresponding to
\begin{gather}\label{equation:phi-parameterization-Ricci-positive}
	\phi_t^{\mp}
		:= \Pi_0^{\Lambda^3 \mcV^*}(\Phi_t^{\mp}) = (\mp \cosh t) I + (\sinh t) J - K .
\end{gather}
Each value of the parameter $(\pm, t)$ corresponds to a distinct distribution.

For $\varepsilon = 0$, $\mcD[\mbD; \bbS]$ consists of the distributions $\mbD_{\upsilon}$ corresponding to
\begin{gather}\label{equation:phi-parameterization-Ricci-flat}
	\phi_s
		:= \Pi_0^{\Lambda^3 \mcV^*}(\Phi_s) = \phi - \tfrac{1}{2} s^2 I + s J .
\end{gather}
Each value of the parameter $s$ corresponds to a distinct distribution.

These parameterizations are distinguished: Locally they agree (up to an overall constant) with the f\/low of the distinguished conformal Killing f\/ield $\xi$ determined by~$\mbD$ and~$\sigma$.
\begin{Proposition}
Let $\mbD$ be an oriented $(2, 3, 5)$ distribution and $\sigma$ an almost Einstein scale of~$\mbc_{\mbD}$. Denote $\xi := \iota_7(\sigma)$ and denote its flow by $\Xi_{\bullet}$. Then, for each $x \in M$ there is a neighbor\-hood~$U$ of~$x$ and an interval~$T$ containing $0$ such that:
\begin{enumerate}\itemsep=0pt
	\item[$1)$] $(T \Xi_{\upsilon / 3}) \cdot \mbD \vert_U = \mbD_{\upsilon}\vert_U$ for all $\upsilon \in T$, if $\sigma$ is Ricci-negative,
	\item[$2)$] $(T \Xi_{t / 3}) \cdot \mbD \vert_U = \mbD_t^-\vert_U$ and $(T \Xi_{t / 3}) \cdot \mbD_0^+ \vert_U = \mbD_t^+\vert_U$ for all $t \in T$, if $\sigma$ is Ricci-positive, and
	\item[$3)$] $(T \Xi_{s / 3}) \cdot \mbD \vert_U = \mbD_s\vert_U$ for all $s \in T$, if~$\sigma$ is Ricci-flat.
\end{enumerate}
\end{Proposition}
\begin{proof}
In the Ricci-negative case this follows immediately from the facts that the normal conformal Killing $2$-form $\phi$ corresponding to $\mbD$ satisf\/ies $\mcL_{\xi} \phi = 3 J$~\eqref{equation:Lie-derivatives-objects} and that the $1$-parameter family of normal conformal Killing $2$-forms $\phi_{\upsilon}$ corresponding to the distributions $\mbD_{\upsilon}$ satisfy $\smash{\left.\frac{d}{d\upsilon}\right\vert_0 \phi_{\upsilon} = J}$. The other cases are analogous.
\end{proof}

\subsubsection{Additional induced distributions}
\label{subsubsection:additional-distributions}

An almost Einstein scale for an oriented $(2, 3, 5)$ distribution $(M, \mbD)$ naturally determines one or more additional $2$-plane distributions on $M$, depending on the sign of the Einstein constant.

\begin{Definition}
We say that two oriented $(2, 3, 5)$ distributions $\mbD, \mbD'$ in a given $1$-parameter family $\mcD$ of conformally isometric oriented $(2, 3, 5)$ distributions are \textit{antipodal} if\/f (1)~they are distinct, and (2)~their respective corresponding parallel tractor $\G_2$-structures, $\Phi$, $\Phi'$, together satisfy $\Phi \wedge \Phi' = 0$.
\end{Definition}

This condition is visibly symmetric in $\mbD, \mbD'$. Note that rearranging~\eqref{equation:pi-3-7} and passing to the tractor bundle setting gives that $\Phi \wedge \Phi' = 4 \astPhi \pi^3_7(\Phi')$, where $\pi^3_7$ denotes the bundle map $\Lambda^3 \mcV^* \to \mcV$ associated to the algebraic map of the same name in that equation.

\begin{Proposition}\label{proposition:antipodal-distributions}
Let $\mcD$ be a $1$-parameter family of conformally isometric oriented $(2, 3, 5)$ distributions related by an almost Einstein scale and fix $\mbD \in \mcD$.
\begin{itemize}\itemsep=0pt
	\item If the almost Einstein scale determining $\mcD$ is non-Ricci-flat, there is precisely one distribution $\mbE$ antipodal to~$\mbD$.
	\item If the almost Einstein scale determining $\mcD$ is Ricci-flat, there are no distributions antipodal to~$\mbD$.
\end{itemize}
\end{Proposition}
\begin{proof}Let $\Phi$ denote the parallel tractor $\G_2$-structure corresponding to $\mbD$ and let $\bbS$ the parallel standard tractor corresponding to the almost Einstein scale. Any $\mbD' \in \mcD$ corresponds to a~compatible parallel tractor $\G_2$-structure $\Phi' \in \mcF[\Phi; \bbS]$ and by Proposition~\ref{proposition:compatible-g2-structures} we can write $\Phi' = \Phi + \bar A \Phi_I + B \Phi_J$, so $\Phi \wedge \Phi' = 4 \astPhi \pi^3_7(\Phi + \bar A \Phi_I + B \Phi_J)$. Since $\Phi \in \Lambda^3_1$, $\Phi_I \in \Lambda^3_7$, and $\Phi_J \in \Lambda^3_1 \oplus \Lambda^3_{27}$ (see~\eqref{equation:definition-Phi-I}, \eqref{equation:definition-Phi-J}), where $\Lambda^3_{\bullet} \subset \Lambda^3 \mcV^*$ denote the subbundles associated to the $\G_2$-representations $\Lambda^3_{\bullet} \subset \Lambda^3 \bbV^*$, Schur's Lemma implies that $\pi^3_7(\Phi) = \pi^3_7(\Phi_I) = 0$. On the other hand, $\pi^3_7(\Phi_J) = \bbS \neq 0$, so $\Phi \wedge \Phi' = 4 B \astPhi \bbS$, which is zero if\/f $B = 0$.

If $\varepsilon = -1$, then in the parameterization $\{\Phi_{\upsilon}\}$ \eqref{equation:parameterization-Phi-upsilon}, the coef\/f\/icient of $\Phi_J$ is $\sin \upsilon$, and this vanishes only for $\Phi_0 = \Phi = \Phi_I + \Phi_K$ and $\Phi_{\pi} = -\Phi_I + \Phi_K$, corresponding to the distributions $\mbD = \mbD_0$ and $\mbE := \mbD_{\pi}$.

If $\varepsilon = +1$, then in the parameterization $\{\Phi_t^{\mp}\}$ \eqref{equation:parameterization-Phi-t}, the coef\/f\/icient is $\sinh t$, and this vanishes only for $\Phi_0^+ = \Phi_0 = \Phi_I - \Phi_K$ and $\Phi_0^- = -\Phi_I - \Phi_K$, corresponding to the distributions $\mbD = \mbD_0^-$ and $\mbE := \mbD_0^+$.

Finally, if $\varepsilon = 0$, then in the parameterization $\{\Phi_s\}$ \eqref{equation:parameterization-Phi-s}, the coef\/f\/icient is $s$, and this vanishes only for $\Phi_0 = \Phi$ itself, corresponding to $\mbD = \mbD_0$.
\end{proof}

Though in the Ricci-f\/lat case there are no antipodal distributions (see Proposition \ref{proposition:antipodal-distributions}), there is in that case a suitable replacement: Given an oriented $(2, 3, 5)$ distribution $(M, \mbD)$ and an almost Einstein scale $\sigma$ of $\mbc_{\mbD}$, the family $2 s^{-2} \phi_s$ converges to the normal conformal Killing $2$-form $\phi_{\infty} = -I = K$ as $s \to \pm \infty$. By continuity, $\phi_{\infty}$ is decomposable and hence def\/ines on the set where $\phi_{\infty}$ does not vanish a distinguished $2$-plane distribution $\mbE$ called the \textit{null-complementary distribution} (for $\mbD$ and $\sigma$), and the distribution so def\/ined is the same for every $\mbD \in \mcD[\mbD; \sigma]$. (This is analogous to the notion of antipodal distribution in that both antipodal and null-complementary distributions are spanned by the decomposable conformal Killing form $-I - \varepsilon K$.) Corollary \ref{corollary:null-complementary-distribution-set-of-definition} gives a precise description of the set on which $\phi_{\infty}$ does not vanish and hence on which $\mbE$ is def\/ined; this set turns out to be the complement of a set that (if nonempty) has codimension $\geq 3$. Corollary \ref{corollary:null-complementary-intgrable} below shows that $\mbE$ is integrable (and hence not a $(2, 3, 5)$ distribution).

\begin{Remark}\label{remark:space-and-time-oriented}
Since $\G_2$ is connected, it is contained in the connected component $\SO_+(3, 4)$ of the identity of $\SO(3, 4)$, and hence a $\G_2$-structure determines space- and time-orientations on the underlying vector space. If we replace $\SO(3, 4)$ with $\SO_+(3, 4)$ in the description of the construction $\mbc \rightsquigarrow \mbc_{\mbD}$ in Section~\ref{subsection:235-conformal-structure}, the construction assigns to an oriented $(2, 3, 5)$ distribu\-tion~$\mbD$ the (oriented) conformal structure $\mbc_{\mbD}$ along with space and time orientations.

Suppose $\mbc_{\mbD}$ admits a nonzero almost Einstein scale $\sigma$. If $\sigma$ is Ricci-negative or Ricci-f\/lat, then the family $\mcD := \mcD[\mbD; \sigma]$ (parameterized respectively as in~\eqref{equation:phi-parameterization-Ricci-negative} or~\eqref{equation:phi-parameterization-Ricci-flat}) is connected, so the space and time orientations of~$\mbc_{\mbD}$ determined by the distributions in~$\mcD$ all coincide.

\looseness=-1 If instead $\sigma$ is Ricci-positive, then $\mcD$ consists of two connected components. Again, by connectness, the distributions $\mbD_t^-$, which comprise the component containing $\mbD_0^- = \mbD$ (in the notation of Section~\ref{subsubsection:parameterizations-isometric}) all determine the same space and time orientations of $\mbc_{\mbD}$, but the distributions $\mbD_t^+$, which comprise the other component and which include the antipodal distribution $\mbD_0^+ = \mbE$, determine the space and time orientations opposite those determined by~$\mbD$.
\end{Remark}

In the case that $\sigma$ is Ricci-positive, $\mbD$ and $\sigma$ determine two additional distinguished distributions: In the notation of Section~\ref{subsubsection:parameterizations-isometric}, the family $(\sech t) \phi_t^{\mp}$ converges to the normal conformal Killing $2$-form $\mp I \pm' J$ as $t \to \pm' \infty$. By continuity $\phi_{\mp \infty} := \pm I + J$ are decomposable and hence determine distributions $\mbD_{\mp \infty}$ on the sets where they respectively do not vanish. Proposition~\ref{proposition:vanishing-phi-pm-infinity} below describes precisely these sets (their complements, if nonempty, have codimension~$3$). By construction $\mbD_{\mp \infty}$ depend only on the family $\mcD[\mbD; \sigma]$ and not $\mbD$ itself. Computing in an adapted frame shows that the corresponding parallel tractor $3$-forms $\smash{\Phi_{\mp \infty} := L_0^{\Lambda^3 \mcV^*}(\phi_{\mp \infty})}$ are not generic (they both annihilate $L_0^{\mcV}(\sigma)$, that is, they are not $\G_2$-structures, and hence the distribu\-tions~$\mbD_{\mp \infty}$ are not $(2, 3, 5)$ distributions).

\subsubsection{Recovering the Einstein scale relating conformally isometric distributions}\label{subsubsection:recovering-Einstein-scale}

Given two distinct, oriented $(2, 3, 5)$ distributions $\mbD$, $\mbD'$ for which $\mbc_{\mbD} = \mbc_{\mbD'}$, we can reconstruct explicitly an almost Einstein scale $\sigma \in \Gamma(\mcE[1])$ of $\mbc_{\mbD}$ for which $\mbD' \in \mcD[\mbD; \sigma]$. If we require that the corresponding parallel standard tractor $\bbS := L_0^{\mcV}(\sigma)$ satisf\/ies $\varepsilon := -H_{\Phi}(\bbS, \bbS) \in \{-1, 0, +1\}$, then $\mbD$ and $\mbD'$ together determine $\varepsilon$. If $\varepsilon \in \{\pm 1\}$, $\sigma$ is determined up to sign. If $\varepsilon = 0$, then we may choose $\sigma$ so that $\mbD' = \mbD_1$, where the right-hand side refers to the parameterization in Section~\ref{subsubsection:parameterizations-isometric}, and this additional condition determines $\sigma$. The maps $\pi^3_1 \colon \Lambda^3 \mcV^* \to \mcE$, \mbox{$\pi^3_7 \colon \Lambda^3 \mcV^* \to \mcV$}, $\pi^3_{27} \colon \Lambda^3 \mcV^* \to S^2_{\circ} \mcV^*$ denote the bundle maps respectively associated to \eqref{equation:pi-3-1}, \eqref{equation:pi-3-7}, \eqref{equation:pi-3-27}.

\begin{algorithm}\label{algorithm:recovery}
\textit{Input}: Fix distinct, oriented $(2, 3, 5)$ distributions $\mbD$, $\mbD'$ on $M$ such that \mbox{$\mbc_{\mbD} = \mbc_{\mbD'}$}.

Let $\Phi, \Phi'$ denote the parallel tractor $\G_2$-structures respectively corresponding to $\mbD$, $\mbD'$, and define $\bbT := \pi^3_7(\Phi')$. In each case, $\sigma := \Pi_0^{\mcV}(\bbS)$.
\begin{itemize}\itemsep=0pt
	\item If $H_{\Phi}(\bbT, \bbT) < 0$, set $s := \sqrt{-H_{\Phi}(\bbT, \bbT)}$, so that $\bbS := s^{-1} \bbT$ satisfies $-H_{\Phi}(\bbS, \bbS) = +1$. Then, $\phi' = \phi_{\arsinh s} = \mp\sqrt{s^2 + 1} I + s J - K$, where $\mp$ is the negative of the sign of $\pi^3_1(\Phi')$.
	\item If $H_{\Phi}(\bbT, \bbT) > 0$, set $s := -\sqrt{H_{\Phi}(\bbT, \bbT)}$, so that $\bbS := s^{-1} \bbT$ satisfies $-H_{\Phi}(\bbS, \bbS) = -1$. Then, $\phi' = \phi_{\upsilon} = c I + s J + K$, where $c := \tfrac{1}{4} [7 \pi^3_1(\Phi') - 3]$ and $\upsilon$ is an angle that satisfies $\cos \upsilon = c$, $\sin \upsilon = s$.
	\item If $H_{\Phi}(\bbT, \bbT) = 0$ but $\bbT \neq 0$, set $\bbS := -\tfrac{1}{4} \bbT$, giving $\phi' = \phi_1 = \phi - \tfrac{1}{2} I + J$.
	\item If $\bbT = 0$, then $($by definition$)$ the distributions are antipodal. Now, $\pi^3_{27}(\Phi') + \tfrac{1}{7} H_{\Phi} = \pm \bbS^{\flat} \otimes \bbS^{\flat}$ for a unique choice of $\pm$ and a parallel tractor $\bbS \in \Gamma(\mcV)$ determined up to sign. If the equality holds for the sign $+$, then $\varepsilon = -1$ and $\phi' = \phi_{\pi}$. If the equality holds for $-$, then $\varepsilon = 1$ and $\phi' = \phi_0^+$.
\end{itemize}
\end{algorithm}

Since all of the involved tractor objects are parallel, the reconstruction problem is equivalent to the algebraic one recovering a normalized vector~$\bbS$ from $\G_2$-structures $\Phi, \Phi'$ on a $7$-dimensional real vector space inducing the same $\SO(3, 4)$-structure such that $\Phi' \in \mcF[\Phi; \bbS]$. One can thus verify the algorithm by computing in an adapted basis.

\section{The curved orbit decomposition}\label{section:curved-orbit-decomposition}

In this section, we treat the curved orbit decomposition of an oriented $(2, 3, 5)$ distribution determined by an almost Einstein scale, that is of a parabolic geometry of type $(\G_2, Q)$ to the stabilizer $S$ of a nonzero ray in the standard representation $\bbV$ of $\G_2$.

In Section~\ref{subsection:primer-curved-orbit-decompositions} we brief\/ly review the general theory of curved orbit decompositions and the decomposition of an oriented conformal manifold determined by an almost Einstein scale. In Section~\ref{subsection:orbit-decomposition-flat-model} we determine the orbit decomposition of the f\/lat model. In Section~\ref{subsection:characterizations-curved-orbits} we state and prove geometric characterizations of the curved orbits, both in terms of tractor data and in terms of data on the base manifold. In the remaining subsections we elaborate on the induced geometry determined on each of the curved orbits, which among other things yields proofs of Theorems~D$_-$, D$_+$, and D$_0$.

\subsection{The general theory of curved orbit decompositions}\label{subsection:primer-curved-orbit-decompositions}
Here we follow~\cite{CGH}. If the holonomy $\Hol(\omega)$ of a Cartan geometry $(\mcG, \omega)$ is a proper subgroup of~$G$, the principal connection $\hat\omega$ extending $\omega$ (see Section~\ref{subsubsection:holonomy}) can be reduced: If $H \leq G$ is a closed subgroup that contains any group in the conjugacy class $\Hol(\omega)$, $\smash{\hat\mcG} := \mcG \times_P G$ admits a reduction $j\colon \mcH \to \smash{\hat\mcG}$ of structure group to $H$, and $j^* \hat\omega$ is a principal connection on~$\mcH$. Such a reduction can be viewed equivalently as a section of the associated f\/iber bundle $\smash{\hat\mcG} / H := \smash{\hat\mcG} \times_G (G / H)$. We henceforth work with an abstract $G$-homogeneous space $\mcO$ instead of~$G / H$, which makes some exposition more convenient, and we call the corresponding $G$-equivariant section $s \colon \smash{\hat\mcG} \to \mcO$ a~\textit{holonomy reduction of type $\mcO$}. Note that we can identify $\hat\mcG \times_G \mcO$ with $\mcG \times_P \mcO$.

Given a Cartan geometry $(\mcG \to M, \omega)$ of type $(G, P)$ and a holonomy reduction thereof of type~$\mcO$ corresponding to a section $s\colon \hat\mcG \to \mcO$, we def\/ine for each $x \in M$ the \textit{$P$-type of $x$ $($with respect to~$s)$} to be the $P$-orbit $s(\mcG_x) \subseteq \mcO$. This partitions $M$ by $P$-type into a disjoint union $\smash{\bigcup_{a \in P \backslash \mcO} M_a}$ of so-called \textit{curved orbits} parameterized by the space $P \backslash \mcO$ of $P$-orbits of~$\mcO$. By construction, the $P$-type decomposition of the f\/lat model~$G / P$ coincides with the decomposition of~$G / P$ into $H$-orbits (for any particular choice of conjugacy class representative~$H$). Put another way, the $P$-types correspond to the possible intersections of~$H$ and~$P$ in~$G$ up to conjugacy.

The central result of the theory of curved orbit decompositions is that each curved orbit inherits from the Cartan connection $\omega$ and the holonomy reduction $s$ an appropriate Cartan geometry: We need some notation to state the result: Given a~$G$-homogeneous space $\mcO$ and elements $x, x' \in \mcO$, we have $x' = g \cdot x$ for some $g \in G$, and their respective stabilizer subgroups~$G_x$,~$G_{x'}$ are related by $G_{x'} = g G_x g^{-1}$. If~$x$,~$x'$ are in the same $P$-orbit, we can choose $g \in P$, and if we denote $P_x := G_x \cap P$, we likewise have $P_{x'} = g P_x g^{-1}$. Thus, as groups endowed with subgroups, $(G_x, P_x) \cong (G_{x'}, P_{x'})$. Given an orbit $a \in P \backslash \mcO$, we denote by $(H, P_a)$ an abstract representative of the isomorphism class of groups so endowed.

\begin{Theorem}[{\cite[Theorem 2.6]{CGH}}]\label{theorem:curved-orbit-decomposition}
Let $(\mcG \to M, \omega)$ be a parabolic $($more generally, Cartan$)$ geometry of type $(G, P)$ with a holonomy reduction of type $\mcO$. Then, for each orbit $a \in P \backslash \mcO$, there is a principal bundle embedding $j_a \colon \mcG_a \hookrightarrow \mcG\vert_{M_a}$, and $(\mcG_a \to M_a, \omega_a)$ is a Cartan geometry of type $(H_a, P_a)$ on the curved orbit~$M_a$, where $\omega_a := j_a^* \omega$.
\end{Theorem}

Informally, since each $P$-type corresponds to an intersection of $H$ and $P$ up to conjugacy in~$G$, for each such intersection $H \cap P$ (up to conjugacy) the induced Cartan geometry on the corresponding curved orbit has type~$(H, H \cap P)$.

\begin{Example}[almost Einstein scales, {\cite[Theorem~3.5]{CGH}}]\label{example:curved-orbit-decomposition-almost-Einstein}
Given a conformal structure $(M, \mbc)$ of signature $(p, q)$, $n := p + q \geq 4$, by Theorem \ref{theorem:almost-Einstein-bijection} a nonzero almost Einstein scale $\sigma \in \Gamma(\mcE[1])$ corresponds to a nonzero parallel standard tractor $\bbS := L_0^{\mcV}(\sigma)$ and hence determines a holonomy reduction to the stabilizer subgroup $\bar{S}$ of a nonzero vector in the standard representation $\bbV$ of $\SO(p + 1, q + 1)$; the conjugacy class of $\bar{S}$ depends on the causality type of $\bbS$.

If $\bbS$ is nonisotropic, there are three curved orbits, characterized by the sign of $\sigma$. The union of the open orbits is the complement $M - \Sigma$ of the zero locus \mbox{$\Sigma := \{x \in M \colon \sigma_x = 0\}$}, and the reduced Cartan geometries on these orbits are equivalent to the non-Ricci-f\/lat Einstein met\-ric~$\sigma^{-2} \mbg$ of signature~$(p, q)$. If $\bbS$ is spacelike (timelike) the reduced Cartan geometry on the hypersurface curved orbit~$\Sigma$ is a~normal parabolic geometry of type $(\SO(p, q + 1), \bar P)$ ($(\SO(p + 1, q), \bar P)$), which corresponds to an oriented conformal structure $\mbc_{\Sigma}$ of signature $(p - 1, q)$ \mbox{($(p, q - 1)$)}.

If $\bbS$ is isotropic, then again there are two open orbits, and on the union $M - \Sigma$ of these, the reduced Cartan geometry is equivalent to the Ricci-f\/lat metric $\sigma^{-2} \mbg$ of signature $(p, q)$. In this case, $\Sigma$ decomposes into three curved orbits: $\{x \in M \colon \sigma_x = 0, (\nabla \sigma)_x \neq 0\}$, $M_0^+ := \{x \in M \colon$ $\sigma_x = 0, (\nabla \sigma)_x = 0, (\Delta \sigma)_x < 0\}$, and $M_0^- := \{x \in M \colon \sigma_x = 0, (\nabla \sigma) = 0, (\Delta \sigma)_x > 0\}$; here~$\Delta \sigma$ denotes the Laplacian~$\sigma_{,a}{}^a$. The curved orbits $M_0^{\pm}$ are discrete, but the formermost curved orbit is a hypersurface that naturally (locally) f\/ibers by the integral curves of the line f\/ield~$\mbS$ spanned by $\sigma^{,a}$, and the (local) leaf space thereof inherits a conformal structure of signature $(p - 1, q - 1)$.
\end{Example}

\begin{Example}[$(2, 3, 5)$ conformal structures]
By Theorem \ref{theorem:2-3-5-holonomy-characterization}, an oriented conformal structure $(M, \mbc)$ of signature $(2, 3)$ is induced by a $(2, 3, 5)$ distribution if\/f it admits a holonomy reduction to $\G_2$. Since $\G_2$ acts transitively on the f\/lat model $\SO(3, 4) / \bar P \cong \bbS^2 \times \bbS^3$, the holonomy reduction to $\G_2$ determines only a single curved orbit.
\end{Example}

\subsection[The orbit decomposition of the flat model M]{The orbit decomposition of the f\/lat model $\boldsymbol{\mcM}$}\label{subsection:orbit-decomposition-flat-model}

In this subsection, we determine the orbits and stabilizer subgroups of the action of $S$ on the f\/lat model $\mcM := \G_2 / Q \cong \bbS^2 \times \bbS^3$, which by Theorem~\ref{theorem:curved-orbit-decomposition} determines the curved orbit decomposition of a parabolic geometry of type $(\G_2, Q)$.

\begin{Remark}Alternatively, as in the statements of Theorems D$_-$, D$_+$, and D$_0$, we could f\/ix a~conformal structure~$\mbc$, that is, a normal parabolic geometry of type $(\SO(3, 4), \bar P)$ (Section~\ref{subsubsection:oriented-conformal-structures}) equipped with a holonomy reduction to the intersection $S$ of a copy of $\G_2$ in $\SO(3, 4)$ and the stabilizer of a nonzero vector $\bbS \in \bbV$ (where we now temporarily view $\bbV$ as the standard representation of $\SO(3, 4)$). By Remark~\ref{remark:determined-by-S}, this determines a $1$-parameter family $\mcF \subset \Lambda^3 \bbV^*$ of compatible $\G_2$-structures but does not distinguish an element of this family. Transferring this statement to the setting of tractor bundles and then translating it into the setting of a~tangent bundle, such a~holonomy reduction determines a $1$-parameter family $\mcD$ of conformally isometric oriented $(2, 3, 5)$ distributions for which $\mbc = \mbc_{\mbD}$ for all $\mbD \in \mcD$, but does not distinguish a~distribution among them.
\end{Remark}

As usual by scaling assume $\bbS$ satisf\/ies $\varepsilon := -H_{\Phi}(\bbS, \bbS) \in \{-1, 0, +1\}$ and denote $\bbW := \langle \bbS \rangle^{\perp}$. The parabolic subgroup $Q$ preserving any ray $x \in \mcM$ (spanned by the isotropic weighted vector $X \in \bbV[1]$) preserves the f\/iltration $(\bbV_X^a)$ of $\bbV$ determined via \eqref{equation:isotropic-filtration} by $X$: Explicitly this is
\begin{gather}\label{equation:filtration-X}
	\begin{array}{@{}cccccccccccc@{}}
		^{-2} & & ^{-1} & & ^0 & & ^{+1} & & ^{+2} & & ^{+3} \\
		\bbV & \supset & \langle X \rangle^{\perp} & \supset & \im (X \times \,\cdot\,) & \supset & \ker (X \times \,\cdot\,) & \supset & \langle X \rangle & \supset & \{ 0 \}
	\end{array}
\end{gather}
Here, $X \times \,\cdot\,$ is the map $-X^C \Phi_C{}^A{}_B \in \End(\bbV)[1]$, and by the comments after \eqref{equation:isotropic-filtration} it satisf\/ies $X \times \bbV_X^a = \bbV_X^{a + 2}$. Since $Q$ preserves this f\/iltration, the corresponding set dif\/ferences $\{x \in \mcM \colon$ $\bbS_x \in \bbV_X^a - \bbV_X^{a + 1}\}$ are each unions of curved orbits.

\subsubsection{Ricci-negative case}

In this case, which corresponds to $\bbS$ spacelike ($\varepsilon = -1$), $S \cong \SU(1, 2)$ (Proposition \ref{proposition:stabilizer-subgroup}), and $\bbW$ inherits a complex structure $\bbK$ \eqref{equation:definition-K} and a complex volume form $\Psi$ (Proposition \ref{proposition:vareps-complex-volume-forms}). Since $\ker X$ is isotropic and $H_{\Phi}$ has signature $(3, 4)$, $\im X = (\ker X)^{\perp}$ is negative-semidef\/inite, so the only unions of orbits determined by the f\/iltrations are $\{x \colon \bbS \in \bbV - \langle X \rangle^{\perp}\}$ and $\{x \colon \bbS \in \langle X \rangle^{\perp} - \im X\}$.

Pick a nonzero vector $\ul X \in \bbV$ in the ray determined by $X \in \bbV[1]$. Since $\bbS$ is nonisotropic, $\bbV$~decomposes (as an $\SU(1, 2)$-module) as $\bbW \oplus \langle \bbS \rangle$, and with respect to this decomposition, $\ul X$~decomposes as $\ul w + \ul \sigma \bbS \in \bbW \oplus \langle \bbS \rangle$, where $\ul \sigma := H_{\Phi}(\ul X, \bbS)$; so, $0 = H_{\Phi}(\ul X, \ul X) = H_{\Phi}(\ul w, \ul w) + \ul \sigma^2$.

If $\ul \sigma > 0$, then $\ul \sigma^{-1} \ul w \in \bbW$ satisf\/ies $H_{\Phi}(\ul \sigma^{-1} \ul w, \ul \sigma^{-1} \ul w) = -1$. But the set of vectors $\ul w_0 \in \bbW$ satisfying $H(\ul w_0, \ul w_0) = -1$ is just the sphere $\bbS^{2, 3}$, and $\SU(1, 2)$ acts transitively on this space, and hence on the $5$-dimensional space $\smash{\mcM_5^+} \cong \bbS^{2, 3}$ of rays in $\mcM$ it subtends. The isotropy group of the ray spanned by $\ul w_0 + \bbS$ preserves the appropriate restrictions of $H$, $\bbK$, $\Psi$ to the four-dimensional subspace $\langle \ul w_0, \bbK \ul w_0 \rangle^{\perp} \subset \bbW$, so that space is neutral Hermitian and admits a~complex volume form, and hence the isotropy subgroup is contained in $\SU(1, 1) \cong \SL(2, \bbR)$. On the other hand, the isotropy subgroup has dimension $\dim S - \dim \smash{\mcM_5^+} = 8 - 5 = 3$, so it must coincide with~$\SU(1, 1)$.

If $\ul \sigma = 0$, then $\ul w =\ul X \in \bbW$ is isotropic. The set of such vectors is the intersection of the null cone of $H$ with $\bbW$. Again, $\SU(1, 2)$ acts transitively on the ray projectivization $\mcM_4 \cong \bbS^3 \times \bbS^1$ of this space. By construction, the isotropy subgroup $P_-$, which (since $\dim \mcM_4 = 4$) has dimension four, is contained in the $5$-dimensional stabilizer subgroup $P_{\SU(1, 2)}$ in $\SU(1, 2)$ of the complex line $\langle \ul X, \bbK \ul X \rangle \subset \bbW$ generated by $\ul X$; this latter group is (up to connectedness) the only parabolic subgroup of $\SU(1, 2)$.

The case $\ul \sigma < 0$ is essentially identical to the case $\ul \sigma > 0$, and we denote the correponding orbit $\smash{\mcM_5^-} \cong \bbS^{2, 3}$.

\subsubsection{Ricci-positive case}

In this case, which corresponds to $\bbS$ timelike ($\varepsilon := +1$), $S \cong \SL(3, \bbR)$, and $\bbW$ inherits a~paracomplex structure $\bbK$ and a paracomplex volume form $\Psi$. This case is similar to the Ricci-negative one, and we omit details that are similar thereto. Since all of the vectors in $\im X - \ker X$ are timelike, however, that set dif\/ference corresponds to a union of orbits that has no analogue for the other causality types of $\bbS$.

Similarly to the Ricci-negative case, $\bbV$ decomposes (as an $\SL(3, \bbR)$-module) as $\bbW \oplus \langle \bbS \rangle$, we can decompose $\ul X = \ul w + \ul \sigma \bbS \in \bbW \oplus \langle \bbS \rangle$, where $\sigma = H_{\Phi}(\ul X, \ul \bbS),$ and we have $0 = H_{\Phi}(\ul X, \ul X) = H_{\Phi}(\ul w, \ul w) - \ul \sigma^2$.

If $\ul \sigma > 0$, then the resulting curved orbit is $M_5^+ \cong \bbS^{2, 3}$, and the stabilizer subgroup is isomorphic to $\SL(2, \bbR)$.

If $\ul \sigma = 0$, then $\ul w = \ul X \in \bbW$ is isotropic. As mentioned at the beginning of this subsubsection, unlike in the Ricci-negative case, this subcase entails more than one orbit. To see this, let~$\bbE$ denote the $(+1)$-eigenspace of $\bbK$. Using (the restriction of) $H_{\Phi}$ we may identify the $-1$-eigenspace with $\bbE^*$, and so we can write $\ul X$ as $(0, \ul e, \ul \beta) \in \langle \bbS \rangle \oplus \bbE \oplus \bbE^*$. Since $\ul X$ is isotropic, $0 = \tfrac{1}{2} H_{\Phi}(\ul X, \ul X) = \beta(\ul e)$, and we can identify the ray projectivization of the set of such triples with $\bbS^2 \times \bbS^2$. Now, the action of~$S$ preserves whether each of the components $\ul e$, $\ul \beta$ is zero, giving three cases.

One can readily compute that $S$ acts transitively on pairs $(\ul e, \ul \beta)$ of nonzero elements with isotropy group $P_+ \cong \bbR_+ \ltimes \bbR^3$, which is characterized by its restriction to $\bbE$, and which in turn is given in a convenient basis by
\begin{gather*}
\left\{
	\begin{pmatrix}
		a & b & c \\
		0 & 1 & d \\
		0 & 0 & a^{-1}
	\end{pmatrix}
	\colon a \in \bbR_+;\, b, c, d \in \bbR
\right\} .
\end{gather*}
So, the corresponding orbit $\mcM_4$ has dimension $4$, and it follows from the remaining two cases that $\mcM_4 \cong \bbS^2 \times (\bbS^2 - \{\pm \ast\}) \cong \bbS^2 \times \bbS^1 \times \bbR$ for some point $\ast \in \bbS^2$. By construction, $P_+$~is contained in the $5$-dimensional stabilizer subgroup $P_{12}$ in $\SL(3, \bbR)$ of the paracomplex line $\langle \ul X, \bbK \ul X\rangle \subset \bbW$ generated by $\ul X$, and we may identify $P_{12}$ with the subgroup of the stabilizer subgroup in $\SL(3, \bbR)$ of a complete f\/lag in $\bbW$ that preserves either (equivalently, both of) the rays of the $1$-dimensional subspace in the f\/lag.

In the case that $\ul e \neq 0$ but $\ul \beta = 0$, $S$ again acts transitively, and this time the isotropy subgroup is isomorphic to the (parabolic) stabilizer subgroup $\smash{P_1 := \GL(2, \bbR) \ltimes (\bbR^2)^*}$ of a ray in~$\bbE$, which is the f\/irst parabolic subgroup in $\SL(3, \bbR)$, so the corresponding orbit is $\smash{\mcM_2^+} \cong \bbS^2$.

The dual case $\ul e = 0$, $\ul \beta \neq 0$ is similar: $S$ acts transitively, and we can identify the isotropy subgroup with the (parabolic) stabilizer subgroup $P_2 := \GL(2, \bbR) \ltimes \bbR^2 \cong P_1$ of a dual ray in~$\bbE^*$, so the corresponding orbit is $\smash{\mcM_2^-} \cong \bbS^2$.

Finally, the case $\ul \sigma < 0$ is essentially identical to the case $\ul \sigma > 0$ and we denote the corresponding orbit by $\smash{\mcM_5^-} \cong \bbS^{2, 3}$.

\subsubsection{Ricci-f\/lat case}

In this case, which corresponds to $\bbS$ isotropic ($\varepsilon = 0$), $S \cong \SL(2, \bbR) \ltimes Q_+$ and $\bbW$ inherits a~endomorphism $\bbK$ whose square is zero. Since $\bbS$ is isotropic, it determines a f\/iltration $(\bbV_{\bbS}^a)$ of~$\bbV$. By symmetry, we may identify the sets $\{x \in \mcM \colon \bbS_x \in \bbV_X^a - \bbV_X^{a + 1}\}$ that occur in this case with $\{x \in \mcM \colon X_x \in \bbV - \langle \bbS \rangle^{\perp}\}$, $\{x \in \mcM \colon X_x \in \langle \bbS \rangle^{\perp} - \im \bbK\}$, $\{x \in \mcM \colon X_x \in \ker \bbK - \langle \bbS \rangle\}$, and $\{x \in \mcM \colon X_x \in \langle \bbS \rangle - \{0\} \}$. (The dif\/ference $\{x \in \mcM \colon X_x \in \im \bbK - \ker \bbK\}$ does not occur here, as every vector in $\im \bbK - \ker \bbK$ is timelike.)

If $\ul \sigma > 0$, $S$ acts transitively on the $5$-dimensional space of rays. Computing directly shows that the isotropy subgroup is conjugate to the Levi factor $\SL(2, \bbR) < \G_2$, so we may identify $\smash{\mcM_5^+} \cong (\SL(2, \bbR) \ltimes Q_+) / \SL(2, \bbR) \cong Q_+ \cong \bbR^5$.

If $\ul \sigma = 0$, we see there are several possibilities. Since every vector in $\im \bbK - \ker \bbK$ is timelike, $\{x \in \mcM \colon X \in \langle \bbS \rangle^{\perp} - \im \bbK \}$ is the set of points $x \in \mcM$ such that $H_{\Phi}(X, \bbS) = 0$ but $X \times \bbS \neq 0$. Again, $S$ acts transitively on the $4$-dimensional space $\mcM_4$ of rays. In this case, computing gives that the isotropy subgroup is isomorphic to $\bbR \ltimes \bbR^3$.

Next, we consider the set of points $\{x \in \mcM \colon X \in \ker \bbK - \langle \bbS \rangle\}$. Since $\ker \bbK$ is totally isotropic, $\ker \bbK - \langle X \rangle$ is the set dif\/ference of a $3$-dimensional af\/f\/ine space and a linear subspace, so the corresponding orbit $\mcM_2$ of rays is a twice-punctured $2$-sphere. Again, computing directly gives that $S$ acts transitively on this space, and the stabilizer subgroup is a certain $6$-dimensional solvable group.

When $X \in \langle \bbS \rangle$, either $\bbS$ is in the ray determined by $X$ or its opposite, and these correspond respectively to $0$-dimensional orbits $\smash{\mcM_0^+}$ and $\smash{\mcM_0^-}$.

Finally, the case $\ul \sigma < 0$ is again essentially identical to the case $\ul \sigma > 0$, and again we denote the corresponding orbit $\smash{\mcM_5^-} \cong \bbR^5$.

\subsection{Characterizations of the curved orbits}\label{subsection:characterizations-curved-orbits}

In this subsection we give geometric characterizations of the curved orbits $M_{\bullet}$ determined by the holonomy reduction to $S$.

For the rest of this section, let $\mbD$ be an oriented $(2, 3, 5)$ distribution, denote the corresponding parallel tractor $\G_2$-structure by $\Phi$, and denote its components with respect to a~scale~$\tau$ by~$\phi$, $\chi$, $\theta$, $\psi$ as in \eqref{equation:G2-structure-splitting}. Also f\/ix an almost Einstein scale $\sigma$, denote the corresponding parallel tractor by $\bbS := L_0^{\mcV}(\sigma)$, and denote its components with respect to $\tau$ by $\sigma$, $\mu$, $\rho$ as in~\eqref{equation:standard-tractor-structure-splitting}. On the zero locus $\Sigma := \{x \in M \colon \sigma_x = 0\}$ of~$\sigma$,~$\mu$ is invariant (that is, independent of the choice of scale $\tau$), so on the set where moreover $\mu \neq 0$, $\mu$ determines a line f\/ield~$\mbS$. See also Appendix~\ref{appendix}.

\begin{Proposition}\label{proposition:curved-orbit-characterization}
The curved orbits are characterized $($separately$)$ by the following conditions on tractorial and tangent data. The bullets~$\bullet$ indicate which curved orbits occur for each causality type.
For the curved orbits $M_2^{\pm}$, $(\ast)$ indicates the following: On $M_2^+ \cup M_2^-$ we have $\mu_x \in [\mbD, \mbD]_x[-1] - \mbD_x[-1]$, so projecting to $([\mbD , \mbD]_x / \mbD_x)[-1]$ gives a nonzero element and via the Levi bracket we can regard this as an element of $\Lambda^2 \mbD_x[-1]$. Then, $x \in M_2^-$ $(M_2^+)$ iff this $($weighted$)$ bivector is oriented $($anti-oriented$)$. Also, $\Delta \sigma := \sigma_{,a}{}^a$.
\begin{center}
\def\arraystretch{1.2}
\begin{tabular}{|c|c|c|c|c|c|}
\hline
\multirow{2}{*}{$M_a$}
& \multicolumn{3}{c|}{$\varepsilon$} & \multirow{2}{*}{\begin{tabular}{c}tractor \\ condition\end{tabular}} & \multirow{2}{*}{\begin{tabular}{c}tangent \\ condition\end{tabular}}\\
\cline{2-4}
 & \parbox{0.5cm}{\centering $-1$} & \parbox{0.5cm}{\centering $0$} & \parbox{0.5cm}{\centering $+1$} & & \\
\hline\hline
$M_5^{\pm}$ & $\bullet$ & $\bullet$ & $\bullet$
	& $\pm H_{\Phi}(X, \bbS) > 0$
	& $\pm \sigma > 0$ \\
\hline
$M_4 $ & $\bullet$ & $\bullet$ & $\bullet$
	& \begin{tabular}{c}$H_{\Phi}(X, \bbS) = 0$ \\ $(X \times \bbS) \wedge X \neq 0$\end{tabular}
	& \begin{tabular}{c}$\sigma = 0$ \\ $\xi \neq 0$\end{tabular} \\
\hline
$M_2^{\pm}$ & & & $\bullet$
	& $X \times \bbS = \pm X$
	& \begin{tabular}{c}$\xi = 0$ \\ $(\ast)$\end{tabular} \\
\hline
$M_2 $ & & $\bullet$ &
	& \begin{tabular}{c}$X \times \bbS = 0$ \\ $X \wedge \bbS \neq 0$\end{tabular}
	& $\mbS \subset \mbD$ \\
\hline
$M_0^{\pm}$ & & $\bullet$ &
	& $\bbS \in (\pm X \cdot \bbR_+)[-1]$
	& \begin{tabular}{c}$\sigma = 0$ \\ $\nabla \sigma = 0$ \\ $\mp \Delta \sigma > 0$\end{tabular} \\
\hline
\end{tabular}
\end{center}
\end{Proposition}

\begin{proof}
The characterizations of $M_5^{\pm}$ are immediate from the descriptions in Section~\ref{subsection:orbit-decomposition-flat-model}.

Passing to the tractor setting, a point $x \in M$ is in $M_4$ if $\bbS_x \in \langle\bbS_x\rangle^{\perp} - \im (X_x \times \,\cdot\,)$ (for readability, in this proof we herein sometimes suppress the subscript $_x$). By the discussion after \eqref{equation:filtration-X}, $X \times (\im X) = \langle X \rangle$, so $\bbS \in \im (X \times \,\cdot\,)$ if\/f $(X \times \bbS) \wedge X = 0$, yielding the tractor characterization of $M_4$. Next, $x \in M_2^{\pm}$ if $\bbS \in \im (X \times \,\cdot\,) - \ker (X \times \,\cdot\,)$, so this curved orbit is characterized by $X \times \bbS \in \langle X \rangle - \{ 0 \}$; in fact, since $X \times \bbS = -\bbS \times X = \bbK(X)$, $X$ is an eigenvalue of $\bbK$, and we acutally have $X \times \bbS = \pm X$. Finally, $x \in M_0^{\pm}$ if\/f $\bbS \in \langle X \rangle$, that is, if $X \wedge \bbS = 0$, so $x \in M_2 = \ker X - \langle X \rangle$ if\/f $X \times \bbS = 0$ but $X \wedge \bbS \neq 0$.

In the splitting determined by a scale $\tau$,
\begin{gather*}
	(X \times \bbS)^A
		= \bbK^A{}_B X^B
		\stackrel{\tau}{=} \tractorT{0}{\xi^a}{\alpha}
		= \tractorT{0}{\sigma \theta^a + \mu_c \phi^{ca}}{\mu^c \theta_c}
		\in \Gamma\tractorT{\mcE[2]}{TM}{\mcE}.
\end{gather*}
So, the only nonzero component of $(X \times \bbS) \wedge X$ is $\xi^a$, which together with the tractor characterization gives the tangent bundle characterization of~$M_4$. Since $\sigma = 0$ on $M_4$, we have on that curved orbit that $0 \neq \xi^a = \mu_c \phi^{ca}$, so $\mu^b \not\in [\mbD, \mbD]$ and hence $\mbS \cap [\mbD, \mbD] = \{ 0 \}$ (including this one, the assertions about the components of~$\bbS$ and~$\Phi$ in this proof all follow from Proposition~\ref{proposition:identites-g2-structure-components}).

For $M_2^{\pm}$, comparing the components of $X \times \bbS = \pm X$ gives $\xi^a = 0$ and $\alpha = \mu^c \theta_c = \pm 1$. Together with $\sigma = 0$ the f\/irst condition gives $\phi^{ab} \mu_b = 0$, which is equivalent to $\mu^b \in [\mbD, \mbD][-1]$; on the other hand, $\mu^c \theta_c \neq 0$ gives that $\mu^b \not\in \mbD[-1]$, so $\mu^b$ projects to a nonzero element of $([\mbD, \mbD] / \mbD)[-1] \cong \Lambda^2 \mbD[-1]$ (the isomorphism is the one given by the Levi bracket). Since $\theta^c \in [\mbD, \mbD][-1] - \mbD[-1]$, it also determines a nonzero (and by construction, oriented) element of $\Lambda^2 \mbD[-1]$. Thus, since $\theta^c \theta_c = -1$, we have $\mu^c \theta_c = -1$ (and hence $x \in M_2^-$) if\/f the element of $\Lambda^2 \mbD[-1]$ determined by $\mu$ is oriented. (The above gives $\mbS \subset [\mbD, \mbD]$ and $\bbS \cap \mbD = \{ 0 \}$.)

For $M_2$, if $\bbS \in \ker(X \times \,\cdot\,)$ we have $\sigma = 0$, $\xi^a = 0$, and $\alpha = 0$, so by the argument in the previous case we have $\mbS \subset \mbD$. Since $\bbS \not\in \langle X \rangle$, we have $\bbS \wedge X \neq 0$, which is equivalent to $\mu^a = 0$, and by \eqref{equation:splitting-operator-standard} $\mu = \nabla \sigma$.

Finally, if $\bbS \in \langle X \rangle$, which by the previous case comprises the points where $\mu = 0$. Again using $L_0^{\mcV}$ (and that $\sigma = 0$) gives that $\pm \rho = \mp \tfrac{1}{5} \Delta \sigma$, and so $\pm \rho > 0$ (and hence $x \in M_0^{\pm}$) if\/f $\mp \Delta \sigma > 0$.
\end{proof}

\begin{Corollary}
\label{corollary:xi-behavior-curved-orbits}
Let $\mcD$ be a $1$-parameter family of conformally isometric oriented $(2, 3, 5)$ distributions related by the almost Einstein scale $\sigma$, and let $\xi$ denote the corresponding conformal Killing field. Then:
\begin{enumerate}\itemsep=0pt
	\item[$1.$] The set $M_{\xi} := \{x \in M \colon \xi_x \neq 0\}$ is the union of the open and hypersurface curved orbits:
\begin{gather*}
	M_{\xi} = M_5 \cup M_4 \cup M_5^- .
\end{gather*}
In particular, $(a)$ $M_{\xi} \supseteq M_5^+ \cup M_5^-$, $(b)$ the complement $M - M_{\xi}$ $($if nonempty$)$ has codimension $3$, and $(c)$ if $\sigma$ is Ricci-negative then $M_{\xi} = M$.
	\item[$2.$] 
The curved orbits $M_5^{\pm}$ and $M_4$ are preserved by the flow of $\xi$.
	\item[$3.$]
If $x \in M_5^+ \cup M_5^-$, then for every distribution $\mbD \in \mcD$, $\mbL_x \subset [\mbD, \mbD]_x$ and $\mbL_x$ is transverse to~$\mbD_x$. In particular, $\mbL_x$ is timelike.
	\item[$4.$] If $x \in M_4$, then $\mbL_x \subset \mbD_x$ for every distribution $\mbD \in \mcD$. In particular, $\mbL_x$ is isotropic.
\end{enumerate}
\end{Corollary}
\begin{proof}
Only (2) is not immediate: It follows from the characterizations $M_5^{\pm} = \{\pm \sigma > 0\}$ and $M_4 = \{\sigma = 0\} \cap M_{\xi}$ together with the fact~\eqref{equation:Lie-derivatives-objects} that the f\/low of $\xi$ preserves $\sigma$.
\end{proof}

\begin{Corollary}\label{corollary:null-complementary-distribution-set-of-definition}
Let $\mcD$ be a $1$-parameter family of conformally isometric oriented $(2, 3, 5)$ distributions related by the almost Ricci-flat scale~$\sigma$. Then, the limiting distribution $\mbD_{\infty}$ vanishes precisely on $M_2 \cup M_0^+ \cup M_0^-$, which $($if nonempty$)$ has codimension~$3$, and so the null-complementary distribution~$\mbE$ it determines is defined precisely on $M_{\xi} = M_5^+ \cup M_4 \cup M_5^-$.
\end{Corollary}
\begin{proof}
On $M_5^+ \cup M_5^-$, \eqref{equation:IJK-open-orbit} below gives $\phi_{\infty} = I \stackrel{\sigma}{=} -\psi$ , which vanishes nowhere by Proposi\-tion~\ref{proposition:identites-g2-structure-components}(8) (here, $I$ is the $2$-form determined by any $\mbD \in \mcD$; in the Ricci-f\/lat case, it does not depend on that choice). On $M - (M_5^+ \cup M_5^-)$, \eqref{equation:I-hypersurface} below gives that $I = \xi^{\flat} \wedge \mu^{\flat}$. In the proof of Proposition~\ref{proposition:curved-orbit-characterization} we saw that on $M_4$, $\xi \in \mbD$ and $\mu \not\in [\mbD, \mbD] \supset \mbD$, so $\xi^{\flat}$ and $\mu^{\flat}$ are linearly independent there. On $M_2 \cup M_0^+ \cup M_0^-$, that proposition gives $\xi = 0$, and hence $I = 0$ there.
\end{proof}

\subsection[The open curved orbits M5+/-]{The open curved orbits $\boldsymbol{M_5^{\pm}}$}\label{subsection:open-curved-orbits}

In this section, we f\/ix an oriented $(2, 3, 5)$ distribution $\mbD$ and a nonzero almost Einstein scale $\sigma$ of $\mbc_{\mbD}$ and restrict our attention to the union
\begin{gather*}
	M_5 := M_5^+ \cup M_5^- = \{x \in M \colon \sigma_x \neq 0\} = M - \Sigma
\end{gather*}
determined by the corresponding holonomy reduction; in particular $M_5$ is open, and moreover, by the discussion after Theorem~\ref{theorem:almost-Einstein-bijection}, it is dense in $M$. Since $\sigma$ is nowhere zero on $M_5$, we can work in the scale $\sigma\vert_{M_5}$ itself, in which many earlier formulae simplify. (Henceforth in this subsection, we suppress the restriction notation $\vert_{M_5}$.)

As usual, by rescaling we may assume that the parallel tractor $\bbS$ corresponding to $\sigma$ satisf\/ies $\varepsilon := -H_{\Phi}(\bbS, \bbS) \in \{-1, 0, +1\}$. Then, from the discussion after Theorem \ref{theorem:almost-Einstein-bijection}, $g_{ab} := \sigma^{-2} \mbg_{ab}$ is almost Einstein and has Schouten tensor $\sfP_{ab} = \frac{1}{2} \varepsilon g_{ab}$, or equivalently, Ricci tensor $R_{ab} = 4 \varepsilon g_{ab}$, and hence scalar curvature $R = 20 \varepsilon$. In the scale $\sigma$, the components of the parallel tractor $\bbS$ itself are
\begin{gather}\label{equation:components-S-scale-sigma}
	\sigma \stackrel{\sigma}{=} 1, \qquad
	 \mu^e \stackrel{\sigma}{=} 0, \qquad
	 \rho \stackrel{\sigma}{=} -\tfrac{1}{2} \varepsilon ,
\end{gather}
and substituting in \eqref{equation:components-of-K} gives that
\begin{gather}\label{equation:K-components-scale-sigma}
	\bbK^A{}_B
		:= -\bbS^C \Phi_C{}^A{}_B
		 \stackrel{\sigma}{=} \tractorQ
			{\ul \theta^b}
			{\tfrac{1}{2} (\varepsilon \ul \phi_{ab} + \ul \barphi_{ab})}
			{0}
			{\tfrac{1}{2} \varepsilon \ul \theta_b} .
\end{gather}
We have introduced the $2$-form $\barphi_{ab} := -2 \psi_{ab} \in \Gamma(\Lambda^2 T^*M[1])$ for notational convenience. Note that $\xi^b \stackrel{\sigma}{=} \ul \theta^b$.

Substituting \eqref{equation:components-S-scale-sigma} in \eqref{equation:I}, \eqref{equation:J}, \eqref{equation:K} gives that $I$, $J$, $K$ simplify to
\begin{gather}\label{equation:IJK-open-orbit}
	\ul I_{ab} = \tfrac{1}{2} \big({-}\varepsilon \ul\phi_{ab} + \ul\barphi_{ab}\big) , \qquad
	\ul J_{ab} = \xi^c \ul\chi_{cab} \qquad
	\ul K_{ab} = \tfrac{1}{2} \big({-}\varepsilon \ul\phi_{ab} - \ul\barphi_{ab}\big) .
\end{gather}
The endomorphism component of $\bbK$ in the scale $\sigma$ coincides with $-\ul K^a{}_b$.

\subsubsection{The canonical splitting}

The components of canonical tractor objects on $M_5$ in the splitting determined by $\sigma$ are themselves canonical (and so are, just as well, their trivializations); in particular this includes the components $\chi_{abc}$, $\theta_c$, $\psi_{bc}$ of~$\Phi_{ABC}$. Moreover, via Proposition \ref{proposition:identites-g2-structure-components}, $\mbD$ and the scale $\sigma$ together determine a canonical splitting of the canonical f\/iltration $\mbD \subset [\mbD, \mbD] \subset TM_5$:
\begin{gather}\label{equation:open-orbit-splitting}
	TM_5 = \mbD \oplus \mbL \oplus \mbE .
\end{gather}

The fact that $\xi = \ul\theta$ gives the following:

\begin{Proposition}
\label{proposition:coincidence-line-fields}
The line field $\mbL$ in the splitting \eqref{equation:open-orbit-splitting} determined by $\sigma$ coincides with $($the restriction to $M_5$ of$)$ the line field of the same name spanned by $\xi$ in Section~{\rm \ref{subsection:conformal-Killing-field}}.
\end{Proposition}

The splitting \eqref{equation:open-orbit-splitting} entails isomorphisms $\mbL \cong [\mbD, \mbD] / \mbD$ and $\mbE \cong TM_5 / [\mbD, \mbD]$, and so the components $\smash{\Lambda^2 \mbD \stackrel{\cong}{\to} [\mbD, \mbD]}$ and $\smash{\mbD \otimes ([\mbD, \mbD] / \mbD) \stackrel{\cong}{\to} TM_5 / [\mbD, \mbD]}$ of the Levi bracket $\mcL$ give rise to isomorphisms $\smash{\Lambda^2 \mbD \stackrel{\cong}{\to} \mbL}$ and $\smash{\mbD \otimes \mbL \stackrel{\cong}{\to} \mbE}$. The preferred nonvanishing section $\smash{\xi = \ul\theta \in \Gamma(\mbL)}$ thus yields isomorphisms $\smash{\Lambda^2 \mbD \stackrel{\cong}{\to} \bbR}$ (equivalently, a volume form on $\mbD$) and $\smash{\mbD \stackrel{\cong}{\to} \mbE}$. We may use the volume form to identify $\smash{\mbD \cong \Lambda^2 \mbD \otimes \mbD^* \cong \mbD^*}$ and dualize the previous isomorphism to yield a bilinear pairing $\mbD \times \mbE \to \bbR$.

This splitting is closely connected with the notions of antipodal and null-complementary distributions introduced in Section~\ref{subsubsection:additional-distributions}.

\begin{Theorem}\label{theorem:canonical-splitting}
The canonical distribution $\mbE$ spanned by the tractor component $\psi$ in the splitting $\mbD \oplus \mbL \oplus \mbE$ determined by $\sigma$ $($equivalently, the distribution determined by~$\barphi)$, is $($the restriction to~$M_5$ of$)$ the distribution antipodal or null-complementary to~$\mbD$.
\end{Theorem}
\begin{proof}
From Section~\ref{subsubsection:additional-distributions} the antipodal or null-complementary distribution is spanned by $I - K$, and substituting using \eqref{equation:IJK-open-orbit} gives that this is $\bar\phi$.
\end{proof}

\begin{Remark}\label{remark:generalized-path-geometry}
Together $\mbD$ and $\mbL$ comprise the underlying data of another parabolic geometry on~$M_5$, namely of type $(\SL(4, \bbR), P_{12})$, where $P_{12}$ is the stabilizer subgroup of a partial f\/lag in~$\bbR^4$ of signature $(1, 2)$ under the action induced by the standard action. The underlying structure for a regular, normal geometry of this type is a \textit{generalized path geometry} in dimension $5$, which consists of a $5$-manifold $M_5$, a line f\/ield $\mbL \subset TM_5$, and a $2$-plane distribution $\mbD \subset TM$ such that (1) $\mbL \cap \mbD = \{ 0 \}$, (2) $[\mbD, \mbD] \subseteq \mbD \oplus \mbL$, and (3) if $\eta \in \Gamma(\mbD)$, $\xi' \in \Gamma(\mbL)$, and $x \in M$ together satisfy $[\eta, \xi']_x = 0$, then $\eta_x = 0$ or $\xi'_x = 0$ \cite[Section~4.3.3]{CapSlovak} (in our case, these conditions follow from the properties of the Levi bracket $\mcL$). In this dimension, this geometry is sometimes called XXO geometry, in reference to the marked Dynkin diagram that encodes it. The restriction of~$\mbL$ to $M_{\xi} - M_5 = M_4$ is contained in $\mbD \vert_{M_4}$, so~$M_5$ is the largest set on which this construction yields a generalized path geometry.
\end{Remark}

\subsubsection{The canonical hyperplane distribution}

Denote by $\mbC \subset TM_{\xi}$ the hyperplane distribution orthogonal to $\mbL := \langle \xi \rangle\vert_{M_{\xi}}$.

\begin{Proposition}
Let $\mcD$ be a family of conformally isometric oriented $(2, 3, 5)$ distributions related by an almost Einstein scale $\sigma$. Then, on~$M_5$:
\begin{enumerate}\itemsep=0pt
	\item[$1.$] The pullback $g_{\mbC}$ of the metric $g := \sigma^{-2} \mbg$ to the hyperplane distribution $\mbC$ has neutral signature.
	\item[$2.$] The hyperplane distribution $\mbC$ is a contact distribution iff $\sigma$ is not Ricci-flat.
	\item[$3.$] If we fix $\mbD \in \mcD$, then $\mbC = \mbD \oplus \mbE$, where $\mbE$ is the distribution antipodal or null-complementary to $\mbD$.
	\item[$4.$] The canonical pairing $\mbD \times \mbE \to \bbR$ is nondegenerate, and the bilinear form it induces on~$\mbC$ via the direct sum decomposition $\mbC = \mbD \oplus \mbE$ is~$g_{\mbC}$.
\end{enumerate}
\end{Proposition}
\begin{proof} (1) The conformal class has signature $(2, 3)$, and Corollary~\ref{corollary:xi-behavior-curved-orbits}(3) gives that $\mbL = \mbC^{\perp}$ is timelike.

 (2) In the scale $\sigma$, $\xi \stackrel{\sigma}{=} \ul\theta$ and $d\xi^{\flat} \stackrel{\sigma}{=} \ul K^{\flat} = - \tfrac{1}{2} (\varepsilon \ul \phi + \ul \barphi)$. The decomposability of~$\phi$ and~$\barphi$ (the latter follows from Proposition~\ref{proposition:identites-g2-structure-components}(5)) implies $\xi^{\flat} \wedge d\xi^{\flat} \wedge d\xi^{\flat}
				\stackrel{\sigma}{=} -\tfrac{1}{2} \varepsilon \ul \phi \wedge \ul \theta \wedge \ul \barphi$. By Proposition~\ref{proposition:identites-g2-structure-components}(8), $\phi \wedge \theta \wedge \barphi$ is a nonzero multiple of the conformal volume form $\epsilon_{\mbg}$ and so vanishes nowhere if\/f $\varepsilon \neq 0$.

 (3) By Theorem \ref{theorem:canonical-splitting}, $\mbD$ and $\mbE$ are transverse, and by Corollary~\ref{corollary:containment-family-hyperplane-distribution} they are both contained in $\mbC$, so the claim follows from counting dimensions.

(4) This follows from computing in an adapted frame.
\end{proof}

Computing in an adapted frame gives the following pointwise description:

\begin{Proposition}Let $\mcD$ be a $1$-parameter family of conformally isometric oriented $(2, 3, 5)$ distributions related by an almost Einstein scale $\sigma$. Then, for any $x \in M_5 = \{x \in M \colon \sigma_x \neq 0\}$, the family $\mcD_x := \{\mbD_x \colon \mbD \in \mcD\}$ is precisely the set of totally isotropic $2$-planes in $\mbC_x$ self-dual with respect to the $($weighted$)$ bilinear form $\mbg_{\mbC}$ and the orientation determined by~$\epsilon_{\mbC}$.
\end{Proposition}

Computing in an adapted frame shows that the images of $I, J, K \in \Gamma(\End(TM)[1])$ are contained in~$\mbC[1]$, so they restrict sections of $\End(\mbC)[1]$, which by mild abuse of notation we denote $I^{\alpha}{}_{\beta}$, $J^{\alpha}{}_{\beta}$, $K^{\alpha}{}_{\beta}$ (here and henceforth, we use lowercase Greek indices $\alpha, \beta, \gamma, \ldots$ for tensorial objects on $M$). It also gives that these maps satisfy, for example, $I^{\alpha}{}_{\gamma} I^{\gamma}{}_{\beta} = -\varepsilon \sigma^2 \delta^{\alpha}{}_{\beta} \in \Gamma(\End(\mbC)[2])$ and Proposition~\ref{proposition:properties-IJK} below. In the scale $\sigma$, this and the remaining equations become:
\begin{alignat}{3}
&	\ul I^{\alpha}{}_{\gamma} \ul I^{\gamma}{}_{\beta} \stackrel{\sigma}{=} - \varepsilon \delta^a{}_{\beta}, \qquad &&
\ul J^{\alpha}{}_{\gamma} \ul K^{\gamma}{}_{\beta} = -\ul K^{\alpha}{}_{\gamma} \ul J^{\gamma}{}_{\beta} \stackrel{\sigma}{=} \ul I^{\alpha}{}_{\beta},&\nonumber \\
 &\ul J^{\alpha}{}_{\gamma} \ul J^{\gamma}{}_{\beta} \stackrel{\sigma}{=} \delta^a{}_{\beta},\qquad &&
\ul K^{\alpha}{}_{\gamma} \ul I^{\gamma}{}_{\beta} = -\ul I^{\alpha}{}_{\gamma} \ul K^{\gamma}{}_{\beta} \stackrel{\sigma}{=} - \varepsilon \ul J^{\alpha}{}_{\beta}, &\nonumber \\
&	\ul K^{\alpha}{}_{\gamma} \ul K^{\gamma}{}_{\beta} \stackrel{\sigma}{=} \varepsilon \delta^a{}_{\beta},\qquad &&
	\ul I^{\alpha}{}_{\gamma} \ul J^{\gamma}{}_{\beta} = -\ul J^{\alpha}{}_{\gamma} \ul I^{\gamma}{}_{\beta} \stackrel{\sigma}{=} - \ul K^{\alpha}{}_{\beta} .& \label{equation:ijk-compositions}
\end{alignat}

\begin{Proposition}
\label{proposition:properties-IJK}
Let $(M, \mbD)$ be an oriented $(2, 3, 5)$ distribution and $\sigma$ an almost Einstein scale $\sigma$ for $\mbc_{\mbD}$, and let $\mbE$ be the distribution antipodal or null-complementary to $\mbD$ determined by $\sigma$. Then, on $M_5$:
\begin{enumerate}\itemsep=0pt
	\item[$1.$] $\ul I \in \End(C)$ is an almost $(-\varepsilon)$-complex structure, and $-\ul I\vert_{\mbD}$ is the isomorphism $\smash{\mbD \stackrel{\cong}{\to} \mbE}$ determined by $\sigma$ introduced at the beignning of the subsection. If $\sigma$ is Ricci-flat, then $\ul I\vert_{\mbE} = 0$, and if $\sigma$ is not Ricci-flat, then $\ul I\vert_{\mbE}$ is an isomorphism $\smash{\mbE \stackrel{\cong}{\to} \mbD}$.
	\item[$2.$] $\ul J \in \End(\mbC)$ is an almost paracomplex structure, and its eigenspaces are $\mbD$ and $\mbE$: $\ul J\vert_{\mbD} = \id_{\mbD}$, $\ul J\vert_{\mbE} = -\id_{\mbE}$.
	\item[$3.$] $\ul K \in \End(\mbC)$ is an almost $\varepsilon$-complex structure, $\ul K\vert_{\mbD} = -\ul I\vert_{\mbD}$, and $\ul K\vert_{\mbE} = \ul I\vert_{\mbE}$. If $\varepsilon = +1$, the $(\mp 1)$-eigenspaces of $\ul \mbK$ are the limiting $2$-plane distributions $\mbD_{\mp}$ defined in Section~{\rm \ref{subsubsection:additional-distributions}}.
\end{enumerate}
\end{Proposition}

In the non-Ricci-f\/lat case, we can identify the pointwise $\U(1, 1)$-structure on $M_5$ as follows:

\begin{Proposition}
Let $\mcD$ be a $1$-parameter family of conformally isometric oriented $(2, 3, 5)$ distributions related by a non-Ricci-flat almost Einstein scale $\sigma$.
\begin{enumerate}\itemsep=0pt
	\item[$1.$] For any $\mbD \in \mcD$, the endomorphisms $\ul I$, $\ul J$, $\ul K$ determine an almost split-quaternionic structure on $($the restriction to $M_5$ of$)$~$\mbC$, that is, an injective ring homomor\-phism $\smash{\widetilde{\bbH} \!\hookrightarrow\! \End(\mbC_x)}$ for each $x \in M_5$, where $\smash{\widetilde{\bbH}}$ is the ring of split quaternions.	
\item[$2.$] The almost split-quaternionic structure in $(1)$ depends only on~$\mcD$.
\end{enumerate}
\end{Proposition}
\begin{proof}
(1) This follows immediately from the identities~\eqref{equation:ijk-compositions}.
(2) This follows from computing in an adapted frame.
\end{proof}

\subsubsection[The varepsilon-Sasaki structure]{The $\boldsymbol{\varepsilon}$-Sasaki structure}\label{subsubsection:vareps-Sasaki-structure}

The union $M_5$ of the open orbits turns out also to inherit an $\varepsilon$-Sasaki structure, the odd-dimensional analogue of a K\"ahler structure.

\begin{Definition}\label{definition:vareps-Sasaki-structure}
For $\varepsilon \in \{-1, 0, +1\}$, an \textit{$\varepsilon$-Sasaki structure} on a (necessarily odd-di\-men\-sio\-nal) manifold $M$ is a pair $(h, \xi)$, where $h \in \Gamma(S^2 T^*M)$ is a pseudo-Riemannian metric on $M$ and $\xi \in \Gamma(TM)$ is a vector f\/ield on $M$ such that
\begin{enumerate}\itemsep=0pt
	\item[1)] $h_{ab} \xi^a \xi^b = 1$,
	\item[2)] $\xi_{(a, b)} = 0$ (or equivalently, $(\mcL_{\xi} h)_{ab} = 0$, that is, $\xi$ is a Killing f\/ield for $h$), and
	\item[3)] $\xi^a{}_{, bc} = \varepsilon (\xi^a h_{bc} - \delta^a{}_c \xi_b)$.
\end{enumerate}
An $\varepsilon$-Sasaki--Einstein structure is an $\varepsilon$-Sasaki structure $(h, \xi)$ for which $h$ is Einstein.
\end{Definition}

It follows quickly from the def\/initions that the restriction of $\xi^a{}_{,b}$ is an almost $\varepsilon$-complex structure on the subbundle $\langle \xi \rangle^{\perp}$, and that if $\varepsilon = +1$, the $(\pm 1)$-eigenbundles of this restriction (which have equal, constant rank) are integrable and totally isotropic.

\begin{Theorem}\label{theorem:open-orbit-vareps-Sasaki-structure}
Let $\mcD$ be a $1$-parameter family of conformally isometric oriented $(2, 3, 5)$ distributions related by an almost Einstein scale $\sigma$. On $M_5 := \{\sigma \neq 0\}$, the signature-$(3, 2)$ Einstein metric
\begin{gather*}
	-g = -\sigma^{-2} \mbc_{\mbD}
\end{gather*}
and the Killing field $\xi$ together comprise an $\varepsilon$-Sasaki--Einstein structure.
\end{Theorem}

\begin{proof}Substituting in \eqref{equation:K-squared-identity} the components of $\bbK$ in the scale $\sigma$ \eqref{equation:K-components-scale-sigma}, using that the endomorphism component of $\bbK$ in that scale is $-\ul K^a{}_b$, and simplifying leaves
\begin{gather}\label{equation:K-squared-identity-components-scale-sigma}
	-\xi_c \xi^c \stackrel{\sigma}{=} 1 ,
		\qquad
	\ul K^a{}_c \xi^c \stackrel{\sigma}{=} 0,
		\qquad
	\ul K^a{}_c \ul K^c{}_b \stackrel{\sigma}{=} \varepsilon (\delta^a{}_b + \xi^a \xi_b) ,
\end{gather}
together with $\xi^c K^b{}_c = 0$, but this last equation follows from the second equation and the $g$-skewness of $\ul K$. Similarly, expanding the left-hand side of $\nabla \bbK = 0$ using \eqref{equation:K-components-scale-sigma} with respect to the scale $\sigma$, eliminating duplicated equations, and rearranging gives
\begin{gather}\label{equation:K-parallel-identity-components-scale-sigma}
	\xi^b{}_{, c} \stackrel{\sigma}{=} \ul K^b{}_c ,		\qquad
	\ul K^a{}_{b, c} \stackrel{\sigma}{=} -\varepsilon (\xi^a g_{bc} - \xi_b \delta^a{}_c) .
\end{gather}

1.~Rearranging the f\/irst equation in \eqref{equation:K-squared-identity-components-scale-sigma} gives $(-g_{ab}) \xi^a \xi^b = 1$.

2.~Since $\ul K^a{}_c$ is $g$-skew, symmetrizing the f\/irst equation in \eqref{equation:K-parallel-identity-components-scale-sigma} with $g_{ab}$ gives $\xi_{(b, c)} = \ul K_{(bc)}$ $= 0$; equivalently, $\xi$ is a Killing f\/ield for $g$ and hence for $-g$.

3.~This follows immediately from substituting the f\/irst equation in \eqref{equation:K-parallel-identity-components-scale-sigma} into the second. (Here indices are raised and lowered with $g$ and not with the candidate Sasaki metric~$-g$.)
\end{proof}

\begin{Corollary}\label{corollary:null-complementary-intgrable}
Let $\mbD$ be an oriented $(2, 3, 5)$ distribution and $\sigma$ a nonzero Ricci-flat almost Einstein scale for $\mbc_{\mbD}$. The null-complementary distribution $\mbE$ they determine is integrable.
\end{Corollary}
\begin{proof}Since $\varepsilon = 0$, the second equation of~\eqref{equation:K-parallel-identity-components-scale-sigma} says that $\ul K$ is parallel. By Section~\ref{subsubsection:additional-distributions}, $\ul K$ is decomposable and (as a decomposable bivector f\/ield) spans $\mbE\vert_{M_5}$, so that distribution is $\nabla$-parallel and hence integrable. But $M_5$ is dense in $M$, and integrability is a closed condition, so $\mbE$ is integrable (on all of $M$).
\end{proof}

\begin{Proposition}\label{proposition:Einstein-Sasaki-to-235}
Let $(g, \xi)$ be an oriented $\varepsilon$-Sasaki--Einstein structure $($with $\varepsilon=\pm 1)$ of signature $(2,3)$ on a $5$-manifold~$M$. Then, locally, there is a canonical $1$-parameter family~$\mcD$ of oriented $(2,3,5)$ distributions related by an almost Einstein scale for which the associated conformal structure is $\mathbf{c}=[-g]$.
\end{Proposition}

\begin{proof}Set $\mathbf{c}=[-g]$, let $\sigma \in \Gamma(\mcE[1])$ be the unique section such that $-g = \sigma^{-2} \mbg$, and $\bbS := L_0^{\mcV}(\sigma)$ the corresponding parallel tractor. Def\/ine the adjoint tractor $\bbK := L_0^{\mcA}(\xi)$. It is known that a~Sasaki--Einstein metric $g$ satisf\/ies $\sfP_{ab} = \frac{1}{2} \varepsilon g_{ab}$, and so by \eqref{equation:Einstein-constant} $\bbS$ satisf\/ies $H(\bbS, \bbS) = -\varepsilon$, where~$H$ is the tractor metric determined by $\mbc$. Thus, the proof of Theorem~\ref{theorem:open-orbit-vareps-Sasaki-structure} gives
\begin{align*}
\nabla^{\mcV} \mathbb{K}=0, \qquad\mathbb{K}^2 = {\varepsilon} \,\mathrm{id} + {\mathbb{S} \otimes \mathbb{S}^{\flat}}, \qquad \mathrm{and} \qquad \mathbb{S} \hook\mathbb{K}=0.
\end{align*}
Transferring the content of Section~\ref{subsubsection:vareps-hermitian-structure} to the tractor bundle setting then shows that the parallel subbundle $\mcW := \langle \bbS \rangle^{\perp} \subset \mcV$ inherits a parallel almost $\varepsilon$-Hermitian structure.

Denote the curvature of the normal tractor connection by ${{\Omega_{ab}}^C}_D \in \Gamma(\Lambda^2 T^*M \otimes \End(\mcV))$. The curvature of the induced connection on the bundle $\smash{\Lambda^3_{\mathbb{C}_{\varepsilon}} \mathcal{W}}$ of $\varepsilon$-complex volume forms on~$\mcW$
is given by $\smash{\Omega_{ab}{}^C{}_C + \varepsilon i_{\varepsilon} \Omega_{ab}{}^C{}_D \mathbb{K}^D{}_C}$.
Now $\Omega_{ab}{}^C{}_C=0$ by skew-symmetry, and, since $\mathbb{K}$ is parallel, $\Omega_{ab}{}^C{}_D\mathbb{K}^D{}_C=0$ by \cite[Proposition~2.1]{CapGoverHolonomyCharacterization}. Thus, the induced connection on $\smash{\Lambda^3_{\mathbb{C}_{\varepsilon}}\mathcal{W}}$ is f\/lat, so it admits local parallel sections. Let $\Psi$ be such a (local) parallel section normalized so that $\Psi \wedge \bar\Psi = -\frac{4}{3}i_{\varepsilon}\bbK \wedge \bbK\wedge \bbK$. Denote $\Re \Psi$ the pullback to $\mathcal{V}$ of the real part of $\Psi$.
Then, by Proposition \ref{proposition:epsilon-complex-volume-form-g2-structure}, the parallel tractor $3$-form
\begin{gather*}
\Phi=\mathrm{Re} \Psi+\varepsilon\, \mathbb{S}^{\flat}\wedge\mathbb{K}\in\Gamma\big(\Lambda^3\mathcal{V}\big)
\end{gather*}
def\/ines a parallel $\G_2$-structure on $\mathcal{V}$ compatible with $H$. By the discussion before Proposition~\ref{proposition:identites-g2-structure-components}, its projecting slot def\/ines a $(2,3,5)$ distribution with associated conformal structure $\smash{\mathbf{c}=[-g]}$. Finally, parallel sections of $\smash{\Lambda^3_{\mathbb{C}_{\epsilon}}\mathcal{W}}$ satisfying $\smash{\Psi\wedge \bar{\Psi}= -\frac{4}{3}i_{\varepsilon}\bbK \wedge \bbK\wedge \bbK}$ are parametrized by $\{z\in\mathbb{C}_{\varepsilon}\colon z\bar{z}=1\}$ (that is, $\mathbb{S}^1$ if $\varepsilon=-1$ and $\SO(1, 1)$ if~$\varepsilon=1$).
\end{proof}

\subsubsection{Projective geometry}

On the complement $M_5$ of the zero locus $\Sigma$ of $\sigma$, we may canonically identify (the restriction of) the parallel subbundle $\mcW := \langle \bbS \rangle^{\perp}$ with the \textit{projective} tractor bundle of the projective structure~$[\nabla^g]$, where $g$ is the Einstein metric~$\sigma^{-2} \mbg$, and the connection $\nabla^{\mcW}$ that $\nabla^{\mcV}$ induces on $\mcW$ with the normal projective tractor connection \cite[Section~8]{GoverMacbeth}.

This compatibility determines a holonomy reduction of the latter connection to $S$, and one can analyze separately the consequences of this projective reduction. For example, if $\sigma$ is non-Ricci-f\/lat, then lowering an index of the parallel complex structure $\bbK \in \Gamma(\End(\mcW))$ with $H\vert_{\mcW}$ yields a parallel symplectic form on $\mcW$. A holonomy reduction of the normal projective tractor connection on a $(2 m + 1)$-dimensional projective manifold $M$ to the stabilizer $\Sp(2 m + 2, \bbR)$ of a~symplectic form on a~$(2 m + 2)$-dimensional real vector space determines precisely a~torsion-free contact projective structure \cite[Section~4.2.6]{CapSlovak} on $M$ suitably compatible with the projective structure~\cite{Fox}.

This also leads to an alternative proof that the open curved orbits inherit a Sasaki--Einstein structure in the Ricci-negative case: The holonomy of $\nabla^{\mcW}$ is reduced to $\SU(1, 2)$, but \cite[Section~4.2.2]{Armstrong} identif\/ies $\mfsu(p', q')$ as the Lie algebra to which the projective holonomy connection determined by an Sasaki--Einstein structure is reduced.

The upcoming article \cite{GNW} discusses the consequences of a holonomy reduction of (the normal projective tractor connection of) a $(2m + 1)$-dimensional projective structure to the special unitary group $\SU(p', q')$, $p' + q' = m + 1$.

\subsubsection[The open leaf space L4]{The open leaf space $\boldsymbol{L_4}$}\label{subsubsection:open-leaf-space}

As in Section~\ref{subsection:local-leaf-space}, we assume that we have replaced $M$ by an open subset so that $\pi_L$ is a locally trivial f\/ibration over a smooth $4$-manifold. Def\/ine $L_4 := \pi_L(M_5)$: By Corollary~\ref{corollary:xi-behavior-curved-orbits}(2) $M_5$ is a~union of $\pi_L$-f\/ibers, so $L_3 := \pi_L(M_4) = L - L_4$ is a hypersurface.

Since $\xi\vert_{M_5}$ is a nonisotropic Killing f\/ield, $-g := -\sigma^{-2} \mbc_{\mbD}\vert_{M_5}$ descends to a metric $\hatg$ on~$L_4$ (henceforth in this subsection we sometimes suppress the restriction notation $\vert_{M_5}$). By Proposition~\ref{proposition:descent} $\mcL_{\xi} \sigma = 0$ and $\mcL_{\xi} K = 0$, so the trivialization $\ul K \in \Gamma(\End(TM_5))$ is invariant under the f\/low of $\xi$. Since it annihilates~$\xi$, it descends to an endormorphism f\/ield we denote $\hatK \in \End(TL_4)$. Then, Proposition~\ref{equation:ijk-compositions} implies $\hatK^2 = -\varepsilon \id_{TL_4}$, that is, $K$~is an almost $\varepsilon$-complex structure on~$L_4$. This yields a specialization to our setting of a well-known result in Sasaki geometry.

\begin{Theorem}\label{theorem:Kahler-Einstein}
The triple $(L_4, \hatg, \hatK)$ is an $\varepsilon$-K\"ahler--Einstein structure with $\smash{\hat R} = -6 \varepsilon \hatg$.
\end{Theorem}
\begin{proof}Since $\ul K^a{}_b \xi^b = 0$, the $g$-skewness of $\ul K$ implies the $\hatg$-skewness of $\hatK$. Thus, $\hatg$ and $\hatK$ together comprise an almost K\"ahler structure on $L_4$; the integrability of $\hatK$ is proved, for example, in~\cite{Blair}, so they in fact consistute a K\"ahler structure. Since $\pi_L\vert_{L_4}$ is a~(pseudo-)Riemannian submersion, we can relate the curvatures of $g$ and $\hatg$ via the O'Neill formula, which gives that~$\hatg$ is Einstein and determines the Einstein constant.
\end{proof}

\subsubsection[The varepsilon-K\"ahler--Einstein Fef\/ferman construction]{The $\boldsymbol{\varepsilon}$-K\"ahler--Einstein Fef\/ferman construction}\label{subsubsection:twistor-construction}

The well known construction of Sasaki--Einstein structures from K\"ahler--Einstein structures immediately generalizes to the $\varepsilon$-K\"ahler--Einstein setting; see, for example, \cite{HabilKath} (in this subsubsection, we restrict to $\varepsilon \in \{\pm1\}$). Here we brief\/ly describe the passage from $\varepsilon$-K\"ahler--Einstein structures to almost Einstein~$(2,3,5)$ conformal structures as a generalized Fef\/ferman construction \cite[Section~4.5]{CapSlovak} between the respective Cartan geometries. Further details will be discussed in an article in preparation~\cite{SagerschnigWillseTwistor}.

An $\varepsilon$-K\"ahler structure $(\hat{g},\hat{K})$ of signature $(2,2)$ on a manifold $L_4$ can be equivalently encoded in a torsion-free Cartan geometry $(\mathcal{S}\to L_4,\omega)$ of type $(S,A)$, where $(S,A) =(\SU(1,2),\U(1,1) )$ if $\varepsilon = -1$ and $(S,A)=(\SL(3,\mathbb{R}), \GL(2,\mathbb{R}) )$ if $\varepsilon = 1$, see, for example, \cite{CGH} for the K\"ahler case.
We realize $A$ within $S$ as block diagonal matrices $\begin{pmatrix}\mathrm{det}A^{-1}&0\\ 0& A\end{pmatrix}.$ The action of $A$ preserves the decomposition $\mathfrak{s}=\mathfrak{a}\oplus\mathfrak{m}=\begin{pmatrix}\mathfrak{a}&\mathfrak{m}\\ \mathfrak{m}& \mathfrak{a}\end{pmatrix}$ and is given on $\mathfrak{m}\subset\mathfrak{s}$ by $X\mapsto \mathrm{det}(A) AX$; in particular it preserves an $\varepsilon$-Hermitian structure (unique up to multiples) on $\mathfrak{m}$ and we f\/ix a (standard) choice. The $\mathfrak{m}$-part $\theta$ of a Cartan connection $\hat{\omega}$ of type~$(S,A)$ determines an isomorphism $TL_4 \cong \mathcal{S}\times_{A}\mathfrak{m}$ and (via this isomorphism) an $\varepsilon$-Hermitian structure on~$TL_4$. The $\mathfrak{a}$-part $\gamma$ of the Cartan connection def\/ines a linear connection $\nabla$ preserving this $\varepsilon$-Hermitian structure. If $\omega$ is torsion-free then $\nabla$ is torsion-free, and thus the $\varepsilon$-Hermitian structure is $\varepsilon$-K\"ahler. Conversely, given an $\varepsilon$-K\"ahler structure, the Cartan bundle $\mathcal{S}\to L_4$ is the reduction of structure group of the frame bundle to $A\subset\mathrm{SO}(2,2)$ def\/ined by the parallel $\varepsilon$-Hermitian structure and the (reductive) Cartan connection $\hat{\omega}\in\Omega^1(\mathcal{S},\mathfrak{s})$ is given by the sum $\hat{\omega}=\gamma+\theta$ of the pullback of the Levi-Civita connection form $\gamma\in\Omega^1(\mathcal{S},\mathfrak{a})$ and the soldering form $\theta\in\Omega^1(\mathcal{S},\mathfrak{m})$.

For the construction we f\/irst build the \emph{correspondence space}
\begin{gather*}\mathcal{C}L_4 := \mathcal{S}/A_0 \cong \mathcal{S}\times_{A} (A/A_0),\end{gather*}
where $A_0= \SU(1,1)$ if $\varepsilon=-1$ and $A_0= \SL(2,\mathbb{R})$ if $\varepsilon=1$. Then, $\mathcal{C}L_4\to L_4$ is an $\mathbb{S}^1$-bundle if $\varepsilon=-1$ and an $\SO(1, 1)$-bundle if $\varepsilon=1$. We can view $\hat{\omega}\in\Omega^1(\mathcal{S},\mathfrak{s})$ as a Cartan connection on the $A_0$-principal bundle $\mathcal{S}\to \mathcal{C}L_4$.
Next we f\/ix inclusions
\begin{gather*}S\hookrightarrow \G_2\hookrightarrow \SO(3,4),\end{gather*}
such that $S$ stabilizes a vector $\mathbb{S}$ satisfying $H(\mathbb{S},\mathbb{S})=-\varepsilon$ in the standard representation~$\mathbb{V}$ of~$\G_2$ (here $H$ is the bilinear form the representation determines on $\bbV$), the $S$-orbit in $\G_2/Q\cong \SO(3,4)/\bar P$ is open and $A_0= S\cap Q= S\cap \bar P$.
 Consider the extended Cartan bundles $\mathcal{G}=\mathcal{S}\times_{A_0}Q $ and $\bar{\mathcal{G}}=\mathcal{S}\times_{A_0}\bar{P} $. There exist unique Cartan connections $\omega\in\Omega^1(\mathcal{G}, \mathfrak{g}_2)$ and $\bar{\omega}\in\Omega^1(\bar{\mathcal{G}}, \mathfrak{so}(3,4))$ extending $\hat{\omega}$~\cite{CapSlovak}. Thus one obtains Cartan geometries of type $(\G_2, Q)$ and $(\SO(3,4),\bar{P})$, respectively, on $\mathcal{C}L_4$
that are non-f\/lat whenever one applies the construction to a torsion-free non-f\/lat Cartan connection $\omega$ of type~$(S,A)$.

\begin{Proposition} Let $(\hat{g},\hat{K})$ be an $\varepsilon$-K\"ahler--Einstein structure, $\varepsilon \in \{\pm1\}$, of signature $(2,2)$ on~$L_4$ such that $\smash{\hat{R}_{ab}=-6 \varepsilon \hat{g}_{ab}}$. Then the induced conformal structure $\mathbf{c} := \smash{\hat g}$ on the correspondence space $\mathcal{C}L_4$ is a $(2,3,5)$ conformal structure equipped with a parallel standard tractor~$\mathbb{S}$, $H(\mathbb{S},\mathbb{S})=-\varepsilon$, which corresponds to a non-Ricci-flat Einstein metric in~$\mathbf{c}$.\ $($Here $H$ is the canonical tractor metric determined by $\mbc.)$

Conversely, locally, all $(2,3,5)$ conformal structures containing non-Ricci-flat Einstein metrics arise via this construction from $\varepsilon$-K\"ahler--Einstein structures.
\end{Proposition}

\begin{proof}We f\/irst show that the conformal Cartan geometry $(\bar{\mathcal{G}}\to \mathcal{C}L_4,\bar{\omega})$ on the correspondence space is normal if and only if the $\varepsilon$-K\"ahler structure on $L_4$ is Einstein with $\hat{R}_{ab}=-6 \varepsilon \hat{g}_{ab}$.
The curvature of the Cartan connection $\hat{\omega}=\gamma+\theta$ is given by
\begin{gather*}\hat{\Omega}=\mathrm{d}\gamma +\tfrac{1}{2}[\gamma,\gamma]+\tfrac{1}{2}[\theta,\theta]\in\Omega^2(\mathcal{S},\mathfrak{a}).\end{gather*}
Computing the Lie bracket $[[X,Y],Z]$ for $X,Y,Z\in\mathfrak{m}$, and interpreting the curvature as a tensor f\/ield on $L_4$ shows that it can be expressed as
\begin{align*}
\hat{\Omega}_{ij}{}^k{}_l=\hat{R}_{ij}{}^k{}_l+\varepsilon \hat{g}_{jl}\delta^k{}_i-\varepsilon \hat{g}_{il}\delta^k{}_j +\hat{K}_{jl}\hat{K}^{k}{}_i -\hat{K}_{il}\hat{K}^k{}_j-2\, \hat{K}_{ij}\hat{K}^k{}_{l},
\end{align*}
where $\hat{g}_{ij}$ denotes the metric, $\hat{R}_{ij}{}^k{}_l$ its Riemannian curvature tensor and $\hat{K}_{ij}=\hat{g}_{ik}\hat{K}^k{}_j$ the K\"ahler form. Tracing over $i$ and $k$ shows that
 \begin{align*}
 \hat{\Omega}_{kj}{}^k{}_{l}=\hat{R}_{kj}{}^k{}_{l}+ 6 \varepsilon \hat{g}_{jl}.
 \end{align*}
 Thus $\hat{\Omega}_{kj}{}^k{}_l=0$ if and only if $\hat{R}_{ab}=-6 \varepsilon \hat{g}_{ab}$.
 Further,
 for an $\varepsilon$-K\"ahler--Einstein structure
\begin{align*}
\hat{R}_{ij}{}^k{}_l \hat{K}^l{}_k=2\, \hat{K}^l{}_i \hat{R}_{kl}{}^k{}_j
\end{align*}
holds.
If $\hat{R}_{kj}{}^k{}_l=- 6 \varepsilon \hat{g}_{jl},$ this implies that
$\hat{\Omega}_{ij}{}^k{}_l \hat{K}^l{}_k=0.$
This precisely means that the Cartan curvature takes values in the subalgebra $\mathfrak{a}_0\subset\mathfrak{a}$
 of matrices with vanishing complex trace. Since $\mathfrak{a}_0\subset\bar{\mathfrak{p}},$ the resulting conformal Cartan connection $\bar{\omega}$ is torsion-free. Vanishing of the Ricci-type contraction of $\hat{\Omega}$, i.e., $\hat{\Omega}_{kj}{}^k{}_l=0$, then further implies that that the conformal Cartan connection is normal. Conversely, normality of the conformal Cartan connection implies $\hat{\Omega}_{kj}{}^k{}_l=0$ and thus $\hat{R}_{ab}=-6 \varepsilon \hat{g}_{ab}$.

By construction and normality of the conformal Cartan geometry $(\bar{\mathcal{G}}\to \mathcal{C}L_4,\bar{\omega})$, the induced conformal structure on $\mathcal{C}L_4$ is a $(2,3,5)$ conformal structure that admits a parallel standard tractor $\mathbb{S}$, $H(\mathbb{S},\mathbb{S})=-\varepsilon$, with underlying nowhere-vanishing Einstein scale $\sigma$. The vertical bundle $V\mathcal{C} L_4$ for $\mathcal{C} L_4\to L_4$ corresponds to the subspace $\mathfrak{a}/\mathfrak{a}_0\subset \mathfrak{s}/\mathfrak{a}_0$, i.e., $V\mathcal{C}L_4=\mathcal{S}\times_{A_0} (\mathfrak{a}/\mathfrak{a}_0)$. Since this is the unique $A_0$-invariant $1$-dimensional subspace in $\mathfrak{s}/\mathfrak{a}_0$, the vertical bundle coincides with the subbundle $\mathbf{L}$ spanned by the Killing f\/ield~$\xi$.

We now prove the converse. Let $(\mathcal{G}_5\to M_5, \omega_5)$ the Cartan geometry of type $(S,A_0)$ on $M_5$ determined (according to Theorem~\ref{theorem:curved-orbit-decomposition}) by the holonomy reduction coresponding to a parallel standard tractor $\mathbb{S}$, $H(\mathbb{S},\mathbb{S})=-\varepsilon$, of a $(2,3,5)$ conformal structure. Restrict to an open subset so that $\pi_L\colon M_5\to L_4$ is a f\/ibration over the leaf space determined by the corresponding Killing f\/ield~$\xi$. Since $\xi$ is a normal conformal Killing f\/ield, it inserts trivially into the curvature of $\omega_5$. Since~$\xi$ spans $VM_5$, by \cite[Theorem~1.5.14]{CapSlovak} this guarantees that on a suf\/f\/iciently small leaf space~$L_4$ one obtains a Cartan geometry of type~$(S,A)$ such that the restriction of $(\mathcal{G}_5\to M_5, \omega_5)$ is locally isomorphic to the canonical geometry on the correspondence space over~$L_4$. Normality of the conformal Cartan connection implies that the Cartan geometry of type $(S,A)$ is torsion-free and the corresponding $\varepsilon$-K\"ahler metric is non-Ricci-f\/lat Einstein.
\end{proof}

\begin{Remark}It is interesting to note the following geometric interpretation of the correspondence spaces: If $\varepsilon=-1$, the bundle $\mathcal{C} L_4\to L_4$ can be identif\/ied with the twistor bundle $\mathbb{T}L_4\to L_4$ whose f\/iber over a point $x\in L_4$ comprises all self-dual totally isotropic $2$-planes in $T_x \mathcal{C} L_4$. If $\varepsilon=1$, $\mathcal{C}L_4\to L_4$ can be identif\/ied with the subbundle of the twistor bundle whose f\/iber over a point $x\in L_4$ comprises all self-dual $2$-planes in $T_x \mathcal{C}L_4$ except the eigenspaces of the endomorphism $\hat{K}_x$. The total space~$\mathcal{C} L_4$ carries a tautological rank $2$-distribution obtained by lifting each self-dual totally isotropic $2$-plane horizontally to its point in the f\/iber, and it was observed~\cite{AnNurowski, BLN} that, provided the self-dual Weyl tensor of the metric on~$L_4$ vanishes nowhere, this distribution is $(2,3,5)$ almost everywhere. This suggests a relation of the present work to the An--Nurowski twistor construction (and recent work of Bor and Nurowski).
\end{Remark}

\subsection[The hypersurface curved orbit M4]{The hypersurface curved orbit $\boldsymbol{M_4}$}

On the complement $M - M_5$, $\sigma = 0$ (and so $\mu$ is invariant); this simplif\/ies many formulae there. First, $\varepsilon = -H_{\Phi}(\bbS, \bbS) = -\mu_a \mu^a$. Substituting in \eqref{equation:definition-xi}, \eqref{equation:I}, \eqref{equation:J}, \eqref{equation:K} and using that expression for $\varepsilon$ yields (on~$M_4$)
\begin{gather}
\xi^a = \mu_b \phi^{ba}, \label{equation:xi-hypersurface} \\
I_{ab}= 3 \mu^c \mu_{[c} \phi_{ab]} = -\varepsilon \phi_{ab} - 2 \mu_{[a} \xi_{b]}, \label{equation:I-hypersurface} \\
J_{ab} = 3 \mu^c \phi_{[ca} \theta_{b]} = \mu^c (\ast \phi)_{cab}, \label{equation:J-hypersurface} \\
K_{ab} = -2 \mu^c \mu_{[a} \phi_{b]c} = 2 \mu_{[a} \xi_{b]} . \label{equation:K-hypersurface}
\end{gather}

Denote by $M_{\mbS}$ the set on which the line f\/ield $\mbS$ is def\/ined (recall from Section~\ref{subsection:characterizations-curved-orbits} that $\mbS := \langle\mu\rangle$ on the space where $\sigma = 0$ and $\mu \neq 0$). By the proof of Proposition \ref{proposition:curved-orbit-characterization}, this is $M_4$ in the Ricci-negative case, $M_4 \cup M_2^+ \cup M_2^-$ in the Ricci-positive case, and $M_4 \cup M_2$ in the Ricci-f\/lat case.

\begin{Proposition}
On the submanifold $M_{\mbS}$, $\mbS^{\perp} = TM_{\mbS} \subset TM\vert_{M_{\mbS}}$.
\end{Proposition}
\begin{proof}
The set $M_{\mbS}$ is precisely where $\sigma = 0$ and $\mu^a = \sigma^{,a} \neq 0$, so $\sigma$ is a def\/ining function for~$M_{\mbS}$ and hence $TM_{\mbS} = \ker \mu$ there.
\end{proof}

\subsubsection{The canonical lattices of hypersurface distributions}

Recall that if $\bbS$ is nonisotropic (if $\sigma$ is not Ricci-f\/lat) it determines a direct sum decomposition $\mcV = \mcW \oplus \langle \bbS \rangle$, where $\mcW := \langle \bbS \rangle^{\perp}$, and if $\bbS$ is isotropic (if $\sigma$ is Ricci-f\/lat), it determines a f\/iltration associated to~\eqref{equation:isotropic-filtration}: $\mcV \supset \mcW \supset \im \bbK \supset \ker K \supset \langle \bbS \rangle \supset \{ 0 \}$. On $M_4$, $\bbS \in \langle X \rangle^{\perp} - \im (X \times \,\cdot\,)$, so $X \times \bbS \in \ker (X \times \,\cdot\,) - \langle X \rangle$. In particular, $X \times \bbS$ is isotropic but nonzero, and hence it determines an analogous f\/iltration of~$\mcV$.

Forming the intersections and spans of the components of the f\/iltrations determined by $\bbS$ and $X \times \bbS$ gives a lattice of vector subbundles of $\mcV$ under the operations of span and intersection (in fact, it is a graded lattice graded by rank). It has $22$ elements in the non-Ricci-f\/lat case and $26$ in the Ricci-f\/lat case, so for space reasons we do not reproduce these here. However, since $\mcW / \langle X \rangle \cong TM[-1]$, the sublattice of vector bundles $\mathcal{N}$ satisfying $\langle X \rangle \preceq \mathcal{N} \preceq \mcW$ descends to a natural lattice of subbundles of $TM\vert_{M_4}$. We record these lattices (they are dif\/ferent in the Ricci-f\/lat and non-Ricci-f\/lat cases), which ef\/f\/iciently encode the incidence relationships among the subbundles, in the following proposition. (We omit the proof, which is tedious but straightforward, and which can be achieved by working in an adapted frame.)

\begin{Proposition}\label{proposition:M4-lattice}
The bundle $TM\vert_{M_4}$ admits a natural lattice of vector subbundles under the operations of span and intersection: If $\sigma$ is non-Ricci-flat, the lattice is
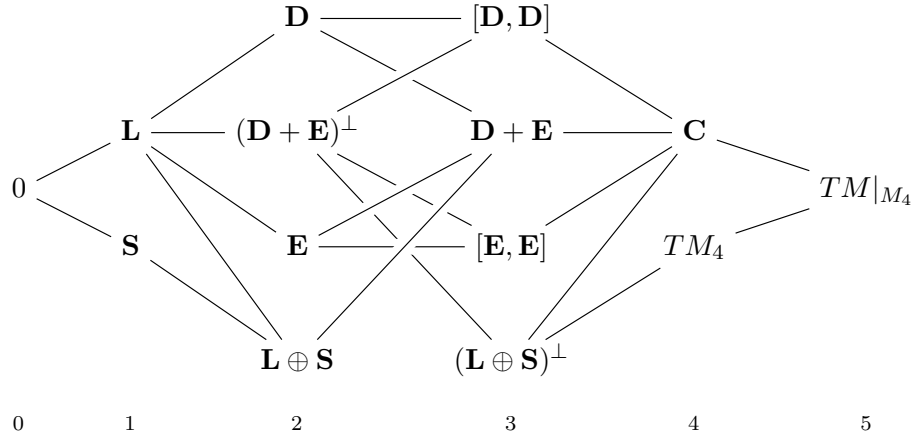
\begin{center}
\begin{tikzcd}[row sep = tiny]
	& & \mbD \arrow[-]{r} \arrow[-]{rdd} & {[\mbD, \mbD]} \arrow[-]{rdd} \\
	\phantom{0} \\
	& \mbL \arrow[-]{ruu} \arrow[-]{r} \arrow[-]{rdd} \arrow[-]{rdddd} & (\mbD + \mbE)^{\perp} \arrow[-, crossing over]{ruu} \arrow[-]{rdd} \arrow[-]{rdddd} & \mbD + \mbE \arrow[-]{r} & \mbC \arrow[-]{rd} \\
	0 \arrow[-]{ru} \arrow[-]{rd} & & & & & TM\vert_{M_4} \\
	& \mbS \arrow[-]{rdd} & \mbE \arrow[-, crossing over]{ruu} \arrow[-, crossing over]{r} & {[\mbE, \mbE]} \arrow[-]{ruu} & TM_4 \arrow[-]{ru} \\
	\phantom{0} \\
	& & \mbL \oplus \mbS \arrow[-, crossing over]{ruuuu} & (\mbL \oplus \mbS)^{\perp} \arrow[-]{ruuuu} \arrow[-]{ruu} \\
	_0 & _1 & _2 & _3 & _4 & _5
	\end{tikzcd} .
\end{center}
In particular, this contains a full flag field
\begin{gather*}
	0 \subset \mbL \subset (\mbD + \mbE)^{\perp} \subset (\mbL \oplus \mbS)^{\perp} \subset TM_4
\end{gather*}
on~$TM_4$. The subbundles in the lattice that depend only on $\mcD[\mbD; \sigma]$ and not $\mbD$ are $0$, $\mbL$, $\mbS$, $\mbL \oplus \mbS$, $(\mbL \oplus \mbS)^{\perp}$, $\mbC$, $TM_4$, $TM\vert_{M_4}$. If $\sigma$ is Ricci-flat, the lattice is
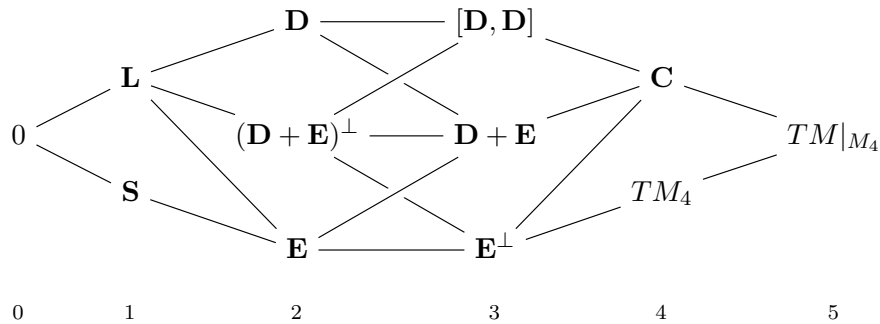
\begin{center}
\begin{tikzcd}[row sep = tiny]
	& & \mbD \arrow[-]{r} \arrow[-]{rdd} & {[\mbD, \mbD]} \arrow[-]{rd} \\
	& \mbL \arrow[-]{ru} \arrow[-]{rd} \arrow[-]{rddd} & & & \mbC \arrow[-]{rd} \\
 	0 \arrow[-]{ru} \arrow[-]{rd} & & (\mbD + \mbE)^{\perp} \arrow[-, crossing over]{ruu} \arrow[-]{r} \arrow[-]{rdd} & \mbD + \mbE \arrow[-]{ru} & & TM\vert_{M_4} \\
	& \mbS \arrow[-]{rd} & & & TM_4 \arrow[-]{ru} \\
	& & \mbE \arrow[-, crossing over]{ruu} \arrow[-]{r} & \mbE^{\perp} \arrow[-]{ruuu} \arrow[-]{ru} \\
	_0 & _1 & _2 & _3 & _4 & _5
\end{tikzcd}
\end{center}
This determines a natural sublattice
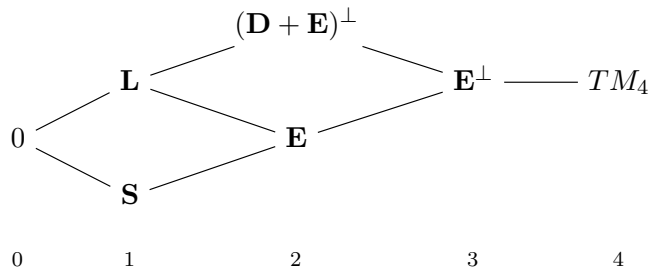
\begin{center}
\begin{tikzcd}[row sep = tiny]
	& & (\mbD + \mbE)^{\perp} \arrow[-]{rd} \\
	& \mbL \arrow[-]{ru} \arrow[-]{rd} & & \mbE^{\perp} \arrow[-]{r} & TM_4 \\
	0 \arrow[-]{ru} \arrow[-]{rd} & & \mbE \arrow[-]{ru} \\
	& \mbS \arrow[-]{ru} \\
	_0 & _1 & _2 & _3 & _4
\end{tikzcd}
\end{center}
of vector subbundles of $TM_4$. The subbundles in the first lattice that depend only on the $1$-parameter family $\mcD[\mbD; \sigma]$ and not $\mbD$ are $0$, $\mbL$, $\mbS$, $\mbE$, $\mbE^{\perp}$, $\mbC$, $TM_4$, $TM\vert_{M_4}$.

In both cases, the restriction $\mbL\vert_{M_4}$ $($which depends only on $\mcD[\mbD; \sigma])$ is the intersection of the distribution $\mbD$ and the tangent space $TM_4$ of the hypersurface.

In the lattices, all bundles are implicitly restricted to $M_4$, the numbers indicate the ranks of the bundles in their respective columns, and $($in the case of the two large lattices$)$ the diagram is arranged so that each bundle is positioned horizontally opposite its $\mbg$-orthogonal bundle.
\end{Proposition}

\subsubsection[The hypersurface leaf space L3]{The hypersurface leaf space $\boldsymbol{L_3}$}\label{subsubsection:hypersurface-leaf-space}

Recall that $L_3 := \pi_L(M_4)$. Since $\mbS$ is spanned by the invariant component $\mu$ of $\bbS$ and $\mcL_{\xi} \bbS = 0$, $\mbS$ descends to a line f\/ield $\smash{\hat\mbS \subset TL\vert_{L_3}}$. This line f\/ield is contained in $TL_3$ if\/f $\mbS$ is contained in $TM_4 = \mbS^{\perp}$, that is (by Proposition \ref{proposition:M4-lattice}) if\/f $\sigma$ is Ricci-f\/lat.

Similarly, since the f\/low of $\xi$ preserves $\mbg$, it also preserves $\mbC \cap TM_4 = (\mbL \oplus \mbS)^{\perp}$. Then, because $\mbL \subset \mbC \cap TM_4 \subset \mbS^{\perp} = TM_4$, $\mbC \cap TM_4$ descends to a $2$-plane distribution $\mbH \subset TL_3$.

\begin{Proposition}
The $2$-plane distribution $\mbH \subset TL_3$ defined as above is contact.
\end{Proposition}
\begin{proof}
Since $\mbC \cap TM_4 = (\ker \xi^{\flat}) \cap TM_4$, we can write this bundle as $\smash{\ker \iota_{M_4}^* \xi^{\flat}} $, where $\iota_{M_4} \colon$ \mbox{$M_4 \hookrightarrow M$} denotes inclusion. By construction, $\smash{\iota_{M_4}^* \xi^{\flat}}$ (where we have trivialized $\xi^{\flat}$ with respect to an arbitrary scale $\tau$) is also the pullback $\pi_L^* \beta$ of a def\/ining $1$-form $\beta \in \Gamma(TL_3)$ for $\mbH$. Thus, $\smash{\pi_L^* (\beta \wedge d\beta)} = \smash{\pi_L^* \beta \wedge d(\pi_L^* \beta)} = \smash{\iota_{M_4}^* \xi^{\flat} \wedge d(\iota_{M_4}^* \xi^{\flat})} = \smash{\iota_{M_4}^* (\xi^{\flat} \wedge d\xi^{\flat})}$, but computing in an adapted frame shows that $\xi^{\flat} \wedge d\xi^{\flat}$ vanishes nowhere, and hence the same holds for $\beta \wedge d\beta$; equivalently, $\mbH$ is contact.
\end{proof}

Now, consider the component $\zeta^a{}_b := -\sigma \psi^a{}_b - \mu_c \chi^{ca}{}_b - \rho \phi^a{}_b\in \Gamma(\End_{\circ}(TM))$ of $\bbK^A{}_B$ in the splitting \eqref{equation:components-of-K} with respect to a scale $\tau$. Let $\mbJ \colon \mbH \to TL\vert_{L_3}$ be the map that lifts a vector $\smash{\hat\eta} \in \mbH_{\hat x}$ to any $\eta \in T_x M_4$ for arbitrary $x \in \pi_L^{-1}(\hat x)$, applies $\zeta$, and then pushes forward back to~$T_{\hat x} L$ by~$\pi_L$.

We show that this map is well-def\/ined, that it is independent of the choice of $\tau$, and that we may regard it as an endomorphism of $\mbH$: By Lemma~\ref{lemma:Lie-derivative-tractor}, $\mcL_{\xi} \bbK = -[L_0^{\mcA}(\xi), \bbK] = -[\bbK, \bbK] = 0$, so~$\bbK$ and hence $\zeta$ is itself invariant under the f\/low of $\xi$, and hence that $\zeta$ is independent of choice of basepoint $x$ of the lift. Now, any two lifts $\eta, \eta' \in T_x M_4$ dif\/fer by an element of $\ker T_x \pi_L = \langle \xi_x \rangle$; on the other hand, expanding \eqref{equation:K-squared-identity} in terms of the splitting determined by $\tau$, taking a particular component equation, and evaluating at $\sigma = 0$ gives the identity $\zeta^b{}_a \xi^a = \alpha \xi^b$ for some smooth function $\alpha$, so $T_x \pi_L \cdot \zeta(\xi) = 0$, and hence $\mbJ$ is well-def\/ined. Finally, under a change of scale, $\zeta$~is transformed to $\zeta^a{}_b \mapsto \zeta^a{}_b + \Upsilon^a \xi_b - \xi^a \Upsilon_b$ for some form $\Upsilon_a \in \Gamma(T^*M)$~\cite{BEG}.
A lift $\eta^b$ of $\hat\eta \in \mbH$ is an element of $\mbC \cap TM_4 \subset \mbC = \ker \xi^{\flat}$, $\Upsilon^a \xi_b \eta^b = 0$. The term $\xi^a \Upsilon_b \eta^b$ is again in $\ker T_x \pi_L$, and we conclude that $\mbJ$ is independent of the scale~$\tau$.

Now, in the notation of the previous paragraph, we have $\mu_b \zeta^b{}_c \eta^c = \mu_b(-\mu_d \chi^{db}{}_c - \rho \phi^b{}_c) \eta^c = - \rho \nu_b \phi^b{}_c \eta^c$. This is $-\rho \xi_c \eta^c$, and we saw above that $\xi_c \eta^c = 0$, so $\mu_b \zeta^b{}_c \eta^c = 0$, that is, $\zeta^b{}_c \eta^c \in \ker \mu = \mbS^{\perp}$. Using the $\mbg$-skewness of $\zeta$ gives $\xi_b \zeta^b{}_c \eta^c = -\eta_b \zeta^b{}_c \xi^c = -\eta_b (\alpha \xi^b) = -\alpha \eta_b \xi^b$, but again $\eta_b \xi^b = 0$, so we also have $\zeta^b{}_c \eta^c \in \ker \xi^{\flat} = \mbL^{\perp}$. Thus, $\zeta(\eta) \in \mbL^{\perp} \cap \mbS^{\perp} = \mbC \cap TM_4$, and pushing forward by $\pi_L$ gives $\mbJ(\hat\eta) \in \mbH$, so we may view $\mbJ$ as an endomorphism of~$\mbH$.

\begin{Proposition}
The endomorphism $\mbJ \in \Gamma(\End(\mbH))$ defined as above is an $\varepsilon$-complex structure.
\end{Proposition}
\begin{proof}In the above notation, unwinding (twice) the def\/inition of $\mbJ$ gives that $\mbJ^2(\hat\eta) = T_x \pi_L \cdot \zeta^2(\eta)$. Now, another component equation of \eqref{equation:K-squared-identity} is $\zeta^a{}_c \zeta^c{}_b - \xi^a \nu_b - \nu^a \xi_b = \varepsilon \delta^a{}_b + \mu^a \mu_b$ for some $\nu \in \Gamma(T^*M)$. The above observations about the terms $\Upsilon^a \xi_b$ and $\xi^a \Upsilon_b$ apply just as well to $\xi^a \nu_b$ and $\nu^a \xi_b$, and since $\eta \in \mbC \cap TM_4 \subset \mbS^{\perp} = \ker \mu$, we have $\mu^a \mu_b \eta^b = 0$, so the above component equation implies $\mbJ^2 = \varepsilon \id_{\mbH}$. In the case $\varepsilon = +1$ one can verify that the $(\pm 1)$-eigenspaces of~$\mbJ$ are both $1$-dimensional, that is, $\mbJ$ is an almost $\varepsilon$-complex structure on~$\mbH$.
\end{proof}

In the Ricci-negative case, this shows precisely that $(L_3, \mbH, \mbJ)$ is an almost CR structure (in fact, it turns out to be integrable, see the next subsubsection), and one might call the resulting structure in the general case an almost $\varepsilon$-CR structure. The three signs of $\varepsilon$ (equivalently, the three signs of the Einstein constant) give three qualitatively distinct structures, so we treat them separately.

\subsubsection{Ricci-negative case: The classical Fef\/ferman conformal structure}\label{subsubsection:curved-orbit-negative-hypersurface}

If $\sigma$ is Ricci-negative, then by Example \ref{example:curved-orbit-decomposition-almost-Einstein}, $M - M_5 = M_4$ inherts a conformal structure $\mbc_{\mbS}$ of signature $(1, 3)$. We can identify the standard tractor bundle $\mcV_{\mbS}$ of $\mbc_{\mbS}$ with the restriction $\mcW\vert_{\Sigma}$ of the $\nabla^{\mcV}$ parallel subbundle $\mcW := \langle \bbS \rangle^{\perp}$, and under this identif\/ication the normal tractor connection on $\mcV_{\mbS}$ coincides with the restriction of $\nabla^{\mcV}$ to $\mcW\vert_{\Sigma}$ \cite{Gover}. In particular, $\Hol(\mbc_{\mbS}) \leq \SU(1, 2)$, but this containment characterizes (locally) the $4$-dimensional conformal structures that arise from the classical Fef\/ferman conformal construction \cite{CapGoverHolonomyCharacterization, LeitnerHolonomyCharacterization}, which canonically associates to any nondegenerate partially integrable almost CR structure of hypersurface type on a manifold a conformal structure on a natural $\bbS^1$-bundle over that manifold \cite{Fefferman}, \cite[Example~3.1.7, Section~4.2.4]{CapSlovak}.

\begin{Proposition}\label{proposition:Fefferman-conformal-structure}
Let $\mcD$ denote a $1$-parameter family of conformally isometric oriented $(2, 3, 5)$ distributions related by a Ricci-negative almost Einstein scale $\sigma$.
\begin{enumerate}\itemsep=0pt
	\item[$1.$] The conformal structure $\mbc_{\mbS}$ of signature $(1, 3)$ determined on the hypersurface curved orbit~$M_4$ is a Fefferman conformal structure.
	\item[$2.$] The infinitesimal generator of the $($local$)$ $\bbS^1$-action is $\xi\vert_{M_4} = \iota_7(\sigma)\vert_{M_4}$, so the line field it spans is $\mbL\vert_{M_4} = \mbD \cap TM_4$ for every $\mbD \in \mcD$.
	\item[$3.$] The $3$-dimensional CR-structure underlying the Fefferman conformal structure $(M_4, \mbc_{\mbS})$ is $(L_3, \mbH, \mbJ)$.
\end{enumerate}
\end{Proposition}
\begin{proof}
The f\/irst claim is deduced in the paragraph before the proposition. The latter claims follow from (1), unwinding def\/initions, and the proof of \cite[Corollary~2.3]{CapGoverHolonomyCharacterization}.
\end{proof}

\subsubsection[Ricci-positive case: A paracomplex analogue of the Fef\/ferman conformal structure]{Ricci-positive case: A paracomplex analogue\\ of the Fef\/ferman conformal structure}\label{subsubsection:curved-orbit-positive-hypersurface}

This case is similar to the Ricci-negative case, but dif\/fers qualitatively in two ways.

First, the endomorphism $\mbJ \in \End(\mbH)$ is a paracomplex structure rather than a complex one; let $\mbH_{\pm}$ denote its $(\pm 1)$-eigendistributions, which are both line f\/ields. A contact distribution on a $3$-manifold equipped with a direct sum decomposition into line f\/ields is the $3$-dimensional specialization of a \textit{Legendrean contact structure}, the paracomplex analogue of a partially integrable almost CR structure of hypersurface type \cite[Section~4.2.3]{CapSlovak}. These correspond to regular, normal parabolic geometries of type $(\SL(3, \bbR), P_{12})$ where $P_{12} < \SL(3, \bbR)$ is a Borel subgroup. The analog of the classical Fef\/ferman conformal structure associates to any Legendrean contact structure on a manifold $N$ a neutral conformal structure on a natural $\SO(1, 1)$-bundle over the manifold \cite{HSSTZ, NurowskiSparling}. By analogy with the construction discussed in Section~\ref{subsubsection:curved-orbit-negative-hypersurface}, we call a~conformal structure that locally arises this way a \textit{para-Fefferman conformal structure}.

Second, the conformal structure $\mbc_{\mbS}$ is def\/ined on the union $M_4 \cup M_2^+ \cup M_2^-$, but only its restriction to $M_4$ is induced by the analogue of the classical Fef\/ferman construction (indeed, recall from Section~\ref{proposition:curved-orbit-characterization} that the vector f\/ield $\xi$ whose integral curves comprise the leaf space $L$ vanishes on $M_2^{\pm}$).

The paracomplex analogue of Proposition \ref{proposition:Fefferman-conformal-structure} is the following:
\begin{Proposition}
Let $\mcD$ denote a $1$-parameter family of conformally isometric $(2, 3, 5)$ distributions related by a Ricci-positive almost Einstein scale $\sigma$.
\begin{enumerate}\itemsep=0pt
	\item[$1.$] The conformal structure $\mbc_{\mbS} \vert_{M_4}$ of signature $(2, 2)$ determined on the hypersurface curved orbit $M_4$ is a para-Fefferman conformal structure.
	\item[$2.$] The infinitesimal generator of the $($local$)$ $\SO(1, 1)$-action is $\xi\vert_{M_4} = \iota_7(\sigma)\vert_{M_4}$, so the line field it spans is $\mbL\vert_{M_4} = \mbD \cap TM_4$ for every $\mbD \in \mcD$.
	\item[$3.$] The $3$-dimensional Legendrean contact structure underlying $(M_4, \mbc_{\mbS}\vert_{M_4})$ is $(L_3, \mbH_+ \oplus \mbH_-)$, where $\mbH_{\pm}$ are the $(\pm 1)$-eigenspaces of the paracomplex structure~$\mbJ$ on~$\mbH$.
\end{enumerate}
\end{Proposition}

The geometry of $3$-dimensional Legendrean contact structures admits another concrete, and indeed classical (local) interpretation, namely as that of second-order ordinary dif\/ferential equations (ODEs) modulo point transformations: We can regard a second-order ODE $\ddot{y} = F(x, y, \dot{y})$ as a function $F(x, y, p)$ on the jet space $J^1 := J^1(\bbR, \bbR)$, and the vector f\/ields $D_x := \partial_x + p \partial_y + F(x, y, p) \partial_p$ and $\partial_p$ span a contact distribution (namely the kernel of $dy - p \,dx \in \Gamma(T^* J^1)$), so $\langle D_x \rangle \oplus \langle \partial_p \rangle$ is a Legendrean contact structure on $J^1$. Point transformations of the ODE, namely those given by prolonging to $J^1$ (local) coordinate transformations of $\bbR_{xy}^2$, are precisely those that preserve the Legendrean contact structure (up to dif\/feomorphism) \cite{DoubrovKomrakov}.

\subsubsection{Ricci-f\/lat case: A f\/ibration over a special conformal structure}\label{subsubsection:curved-orbit-flat-hypersurface}

In this case, Example \ref{example:curved-orbit-decomposition-almost-Einstein} gives that the hypersurface curved orbit $M_4$ locally f\/ibers over the space $\smash{\wt L}$ of integral curves of $\mbS$ (nota bene the f\/ibrations $\pi_L\vert_{M_4} \colon M_4 \to L_3$ in the non-Ricci-f\/lat cases above are instead along the integral curves of $\mbL$), and that $\smash{\wt L}$ inherits a conformal structure~$\mbc_{\wt L}$ of signature~$(1, 2)$. Considering the sublattice of the last lattice in Proposition \ref{proposition:M4-lattice} of the distributions containing $\mbS$ and forming the quotient bundles modulo $\mbS$ yields a complete f\/lag f\/ield of $T \wt L$ that we write as $0 \subset \mbE / \mbS \subset \mbE^{\perp} / \mbS \subset T \wt L$ it depends only on $\mcD$. Since $\mbE$ is totally $\mbc$-isotropic, the line f\/ield $\mbE / \mbS$ is $\mbc_{\wt L}$-isotropic, and by construction it is orthogonal to~$\mbE^{\perp} / \mbS$ with respect to $\mbc_{\wt L}$. Thus, we may regard the induced structure on~$\smash{\wt L}$ as a Lorentzian conformal structure equipped with an isotropic line f\/ield.

Similarly, the f\/ibration along the integral curves of $\mbL$ determines a complete f\/lag f\/ield that we denote $0 \subset \mbE / \mbL \subset \mbH \subset TL_3$. Computing in a local frame gives $\mbE / \mbL = \ker \mbJ = \im \mbJ$, and this line the kernel of the (degenerate) conformal (negative semidef\/inite) bilinear form $\mbc$ determines on~$\mbH$.

\subsection[The high-codimension curved orbits M2+/-, M2, M0+/-]{The high-codimension curved orbits $\boldsymbol{M_2^{\pm}}$, $\boldsymbol{M_2}$, $\boldsymbol{M_0^{\pm}}$}\label{subsection:high-codimension-curved-orbits}

Recall that on these orbits, $\xi = 0$ and $K = 0$, and hence $\mbE$ is not def\/ined. Recall also that if $\sigma$ is Ricci-negative, all three of these curved orbits are empty. If $\sigma$ is Ricci-f\/lat, only $M_2$ and $M_0^{\pm}$ occur, and if $\sigma$ is Ricci-positive, only $M_2^{\pm}$ occur. Since the curved orbits $M_0^{\pm}$ are $0$-dimensional, they inherit no structure.

\subsubsection[The curved orbits M2+/-: Projective surfaces]{The curved orbits $\boldsymbol{M_2^{\pm}}$: Projective surfaces}\label{subsubsection:curved-orbit-M2pm}

By Theorem \ref{theorem:curved-orbit-decomposition} and the orbit decomposition of the f\/lat model in Section~\ref{subsection:orbit-decomposition-flat-model}, the holonomy reduction determines parabolic geometries of type $(\SL(3, \bbR), P_1)$ and $(\SL(3, \bbR), P_2)$ on $M_2^{\pm}$. Torsion-freeness of the normal conformal Cartan connection immediately implies that these parabolic geometries are torsion-free and hence determine underlying torsion-free projective structures (that is, equivalence classes of torsion-free af\/f\/ine connections having the same unparametrized geodesics).

Again, the formulae for various objects simplify on this orbit: By the proof of Proposition~\ref{proposition:curved-orbit-characterization} we have $\sigma = 0$, $\xi = 0$, and $\mu^c \theta_c = \pm 1$ here, and substituting in~\eqref{equation:I}, \eqref{equation:J}, \eqref{equation:K} gives
\begin{gather*}
	I = -\phi, \qquad J = \pm \phi, \qquad K = 0 .
\end{gather*}

These specializations immediately give the Ricci-positive analog of Corollary \ref{corollary:null-complementary-distribution-set-of-definition}:

\begin{Proposition}
\label{proposition:vanishing-phi-pm-infinity}
Let $(M, \mbD)$ be an oriented $(2, 3, 5)$ distribution and $\sigma$ a Ricci-positive scale for $\mbc_{\mbD}$. Then, the limiting normal conformal Killing forms $\phi_{\mp\infty} := \pm I + J$ respectively vanish precisely on $M_2^{\pm}$, so the distributions $\mbD_{\mp\infty}$ they respectively determine are respectively defined precisely on $M_5 \cup M_4 \cup M_2^{\mp}$.
\end{Proposition}

\begin{Proposition}\label{proposition:TM2pm-D}
For all $x \in M_2^{\pm}$, $T_x M_2^{\pm} = \mbD_x$ $($for every $\mbD \in \mcD)$.
\end{Proposition}
\begin{proof}By Proposition \ref{proposition:curved-orbit-characterization}, $M_2^{\pm} = \{ x \in M \colon \xi_x = 0 \}$ and so $TM_2^{\pm} \subseteq \ker \nabla \xi$. On the other hand, as in the proof of that proposition we have $\xi^a{}_{,b} = -\zeta^a{}_b - \theta_c \mu^c \delta^a{}_b$, and computing in an adapated frame shows that on $M_2^{\pm}$, $\xi^a{}_{,b}$ has rank $3$. Equivalently, the kernel has dimension $2 = \dim TM_2^{\pm} = 2$, so $\ker \nabla \xi = T_x M_2^{\pm}$.

Writing $\nabla^{\mcV}_c \bbK^A{}_B = 0$ in components gives $\xi^b{}_{,c} = -\zeta^b{}_{,c} - \mu^d \theta_d \delta^b{}_c$, and as in the proof of Proposition \ref{proposition:curved-orbit-characterization}, $-\zeta^b{}_c = \sigma \psi^b{}_c + \mu_d \chi^{db}{}_c + \rho \phi^b{}_c$. Since $x \in M_2^{\pm}$, $\mu^d \phi_d = \pm 1$ and $\sigma = 0$. For $\eta \in \mbD_x$, Proposition \ref{proposition:identites-g2-structure-components}(2) gives that $\phi^b{}_c \eta^c = 0$, and computing in an adapted frame gives that $\mu_d \chi^{db}{}_c$ restricts to $\id_{\mbD}$ on $M_2^{\pm}$. Substituting then gives $\xi^b{}_{,c} \eta^c = 0$, so by dimension count $T_x M_2^{\pm} = \mbD_x$.
\end{proof}

\subsubsection[The curved orbit M2]{The curved orbit $\boldsymbol{M_2}$}\label{subsubsection:curved-orbit-M2}

As for the hypersurface curved orbits, forming the intersections and spans of the components of the f\/iltrations determined by $\bbS$ and $X \times \bbS$ in this case yields a lattice of (14) vector subbundles of $\mcV$, and determining the lattice of (10) vector subbundles of $TM\vert_{M_2}$ this induces shows in particular that one has a distinguished line f\/ield $\mbS = \mbD \cap TM_2$ on~$M_2$. Specializing the formulae for $I$, $J$, $K$ as in the previous cases gives that on $M_2$, $I = J = K = 0$.

\section{Examples}\label{section:examples}

In this section, we give three conformally nonf\/lat examples, one for each sign of the Einstein constant; each is produced using a dif\/ferent method. To the knowledge of the authors, before the present work there were no examples in the literature of nonf\/lat $(2, 3, 5)$ conformal structures known to admit a~non-Ricci-f\/lat almost Einstein scale.\footnote{We recently learned from Bor and Nurowski that they have, in work in progress, also constructed examples~\cite{BorNurowskiPrivate}.} In particular, these examples show that none of the holonomy reductions considered in this article force local f\/latness of the underlying conformal structure.

\begin{Example}[a distinguished rolling distribution]\label{example:distinguished-rolling-distribution}
We construct a homogeneous Sasaki--Einstein metric of signature $(3, 2)$ whose negative determines a Ricci-negative conformal structure. Each $(2, 3, 5)$ distribution in the corresponding family is dif\/feomorphic to a particular special so-called \textit{rolling distribution}.

Let $(\bbS^2, h_+, \bbJ_+)$ and $(\bbH^2, h_-, \bbJ_-)$ respectively denote the round sphere and hyperbolic plane with their usual K\"ahler structures, rescaled so that their respective scalar curvatures are~$\pm 12$. In the usual respective polar coordinates $(r, \varphi)$ and $(s, \psi)$,
\begin{alignat*}{3}
	&h_+ := \frac{2}{3} \cdot \frac{1}{\big(r^2 + 1\big)^2} \big(dr^2 + r^2 d\varphi^2\big), \qquad&&
\bbJ_+:= r \partial_r \otimes d\varphi - \tfrac{1}{r} \partial_{\varphi} \otimes dr ,&\\
 &	h_- := \frac{2}{3} \cdot \frac{1}{\big(s^2 - 1\big)^2} \big(ds^2 + s^2 d\psi^2\big), \qquad &&
\bbJ_- := s \partial_s \otimes d\psi - \tfrac{1}{s} \partial_{\psi} \otimes ds .
\end{alignat*}
Then, the triple $(\bbS^2 \times \bbH^2, \hatg, \hatK)$, where $\hatg := h_+ \oplus -h_-$ and $\hatK := \bbJ_+ \oplus \bbJ_-$, is a K\"ahler structure satisfying $\smash{\hat R_{ab} = 6 \hatg_{ab}}$. The K\"ahler form $\hatg_{ac} \hatK^c{}_b$ is equal to $(d\alpha)_{ab}$, where
\begin{gather*}
	\alpha := \frac{2}{3}\left(-\frac{\varphi r\,dr}{(r^2 + 1)^2} + \frac{\psi s\,ds}{(s^2 - 1)^2}\right) .
\end{gather*}
The inf\/initesimal symmetries of the K\"ahler structure are spanned by the lifts of the inf\/in\-te\-si\-mal symmetries of $(\bbS^2, h_+, \bbJ_+)$ and $(\bbH^2, h_-, \bbJ_-)$, and so the inf\/initesimal symmetry algebra is $\mfaut(\hatg, \hatK) \cong \mfso(3, \bbR) \oplus \mfsl(2, \bbR)$.

On the canonical $\bbS^1$-bundle $\pi \colon M \to \bbS^2 \times \bbH^2$ def\/ined in Section~\ref{subsubsection:twistor-construction}, with standard f\/iber coordinate $\lambda$, def\/ine $\beta := d \lambda - 2 \pi^* \alpha$. Then, the associated Sasaki--Einstein structure is $(M, g, \partial_{\lambda})$, where $g := \pi^* \hatg + \beta^2$. The normalizations of the scalar curvatures of the sphere and hyperbolic plane were chosen so that $R_{ab} = 4 g_{ab}$. The $1$-parameter family $\{\mbD_{\upsilon}\}$ of corresponding oriented $(2, 3, 5)$ distributions, which in particular induce the conformal class $[-g]$, is
\begin{gather*}
	\mbD_{\upsilon}	=\bigg\langle
		3 \big(r^2\! + 1\big) s \partial_r + 3 \big(s^2\! - 1\big) s \cos \gamma \partial_s + 3 \big(s^2\! - 1\big) \sin \gamma \partial_{\psi} + 4 s \left(\frac{s \psi \cos \gamma}{s^2 - 1} - \frac{r \varphi}{r^2 + 1}\right) \partial_{\lambda}, \\
\hphantom{\mbD_{\upsilon}	=\bigg\langle}{} 3 \big(r^2 + 1\big) s \partial_{\varphi} + 3 \big(s^2 - 1\big) s \sin \gamma \partial_s - 3 r \big(s^2 - 1\big) \cos \gamma \partial_{\psi} + \frac{4 s}{s^2 - 1} r \psi \sin \gamma \partial_{\lambda}	\bigg\rangle,
\end{gather*}
where
\begin{gather*}
	\gamma := \frac{r^2 - 1}{r^2 + 1} \varphi + \frac{s^2 + 1}{s^2 - 1} \psi - 3 \lambda + \upsilon .
\end{gather*}
One can compute the tractor connection explicitly (the explicit expression is unwieldy, so we do not reproduce it here) and use it to compute that the conformal holonomy $\Hol([-g])$ is the full group $\SU(1, 2)$. In particular, this shows that in the Ricci-negative case the holonomy reduction considered in this case does not automatically entail a holonomy reduction to a smaller group. Since almost Einstein scales are in bijective correspondence with parallel standard tractors, the space of Einstein scales is $1$-dimensional (as an independent parallel standard tractor would further reduce the holonomy). One can compute that $\mfaut(\mbD_{\upsilon}) \cong \mfaut(\hatg, \hatK) \cong \mfso(3, \bbR) \oplus \mfsl(2, \bbR)$ and $\mfaut([-g]) \cong \mfaut(g) \cong \mfaut(\hatg, \hatK) \oplus \langle \xi \rangle \cong \mfso(3, \bbR) \oplus \mfsl(2, \bbR) \oplus \bbR$.

One can show that every distribution $\mbD_{\upsilon}$ is equivalent to the so-called rolling distribution for the Riemannian surfaces $(\bbS^2, g_+)$ and $(\bbH^2, g_-)$. The underlying space of this distribution, which we can informally regard as the space of relative conf\/igurations of $\bbS^2$ and $\bbH^2$ in which the surfaces are tangent at a single point, is the twistor bundle \cite{AnNurowski} over $\bbS^2 \times \bbH^2$ whose f\/iber over $(x_+, x_-)$ is the circle $\operatorname{Iso}(T_{x_+} \bbS^2, T_{x_-} \bbH^2) \cong \bbS^1$ of isometries. The distribution is the one characterized by the so-called no-slip, no-twist conditions on the relative motions of the two surfaces \cite[Section~3]{BryantHsu}.

We can produce a para-Sasaki analogue of this example, which in particular has full holonomy group $\SL(3, \bbR)$ and hence shows the holonomy reduction to that group again does not auto\-matically entail a reduction to a smaller group. Let $(\bbL^2, h, \bbJ)$ denote the para-K\"ahler Lorenztian surface with
\begin{gather*}
	h := \frac{2}{3 \big(r^2 + 1\big)^2}\big({-}dr^2 + r^2 d\varphi^2\big), \qquad \bbJ := r \partial_r \otimes d\varphi + \tfrac{1}{r} \partial_{\varphi} \otimes dr .
\end{gather*}
Then, the triple $(\bbL^2 \times \bbL^2, h \oplus h, \bbJ \oplus \bbJ)$, is a suitably normalized para-K\"ahler structure and we can proceed as before. Every $(2, 3, 5)$ distribution in the determined family is dif\/feomorphic to the Lorentzian analogue of the rolling distribution for the surfaces $(\bbL^2, h)$ and $(\bbL^2, -h)$.\footnote{This para-K\"ahler--Einstein structure is isometric to \cite[equation~(4.21)]{Chudecki}, which is attributed there to Nurowski.}
\end{Example}

\begin{Example}[a cohomogeneity $1$ distribution from a homogeneous projective surface]\label{example:Dirichlet-Ricci-positive}
We construct an example of a Ricci-positive almost Einstein $(2, 3, 5)$ conformal structure by specifying a para-Fef\/ferman conformal structure $\mbc_N$ on a $4$-manifold $N$ and solving a natural geometric Dirichlet problem: We produce a conformal structure $\mbc$ on $N \times \bbR$ equipped with a holonomy reduction to $\SL(3, \bbR)$ for which the hypersurface curved orbit is $N$ and the induced structure there is $\mbc_N$. In particular, this yields an example of an almost Einstein $(2, 3, 5)$ distribution for which the zero locus of the almost Einstein scale is nonempty, and hence for which the curved orbit decomposition has more than one nonempty curved orbit.

Consider the projective structure $[\nabla]$ on $\bbR^2_{xy}$ containing the torsion-free connection $\nabla$ characterized by
\begin{gather*}
	\nabla_{\partial_x} \partial_x = 3 x y^2 \partial_x + x^3 \partial_y , \qquad
	\nabla_{\partial_x} \partial_y = \nabla_{\partial_y} \partial_x = 0 , \qquad
	\nabla_{\partial_y} \partial_y = x^3 \partial_x - 3 x^2 y \partial_y .
\end{gather*}
Eliminating the parameter in the geodesic equations for $\nabla$ yields the ODE $\ddot{y} = (x \dot{y} - y)^3$, which corresponds (recall Section~\ref{subsubsection:curved-orbit-positive-hypersurface}) to the function $F(x, y, p) = (x p - y)^3$.
The point symmetry algebra of the ODE (that is, the symmetry algebra of the Legendrean contact structure on $J := \{x p - y > 0\} \subset J^1(\bbR, \bbR)$) is $\mfsl(2, \bbR)$ and acts inf\/initesimally transitively. Hence, we may identify $J$ with an open subset of (some cover of) $\SL(2, \bbR)$. With respect to the left-invariant local frame
\begin{gather*}
	E_X := -(x p - y)^2 \partial_p, \qquad
	E_H := x \partial_x + y \partial_y, \qquad
	E_Y := \frac{1}{x p - y} (\partial_x + p \partial_y) .
\end{gather*}
of $J$, the line f\/ields spanning the contact distribution are $\langle \partial_p \rangle = \langle E_X \rangle$, and $\langle D_x \rangle = \langle E_Y - 3 E_X \rangle$. The Fef\/ferman conformal structure $(N, \mbc_N)$ is again homogeneous: Its ($5$-dimensional) symmetry algebra $\mfaut(\mbc_N)$ contains an inf\/initesimally transitive subalgebra isomorphic to $\mfgl(2, \bbR)$. A (local) left-invariant frame of $N$ realizing this subalgebra is given by
\begin{gather*}
	\hat E_X = E_X + x (x p - y) \partial_a , \qquad	\hat E_H = E_H - \partial_a , \qquad	\hat E_Y = E_Y , \qquad	\partial_a ,
\end{gather*}
where $a$ is the standard coordinate on the f\/iber of $N \to J$ and our notation uses the natural (local) decomposition $N \cong J \times \bbR_a$. In the dual left-invariant coframe $\{\chi, \eta, \upsilon, \alpha\}$, the conformal structure $\mbc_N$ has left-invariant representative $g_N := - \chi \upsilon - \eta^2 + \eta \alpha - \upsilon^2$.

The scale $\sigma_N := e^{a / 2} \sqrt{x p - y}$ (given here with respect to the scale corresponding to $g_N$) is an almost Einstein scale, and hence $\smash{g_E := \sigma_N^{-2} g_N}$ is Einstein (in fact, Ricci-f\/lat). The conformal class $\mbc$ on $M := N \times \bbR_r$ containing $g' := g_E - dr^2$ admits the almost Einstein scale $r$ (here given with respect to $g'$): $\smash{g := r^{-2} g'\vert_{\{\pm r > 0\}} \in \mbc\vert_{\{\pm r > 0\}}}$ is a \textit{Poincar\'e--Einstein metric} for $\mbc_N$, and in particular is Ricci-positive, and $\mbc_N$ is a \textit{conformal infinity} for $g$; see \cite[Section~4]{FeffermanGraham}. (We suppress the notation for the pullback by the canonical projection $M = N \times \bbR \to N$.)

So, the curved orbits are $M_5^{\pm} = \{(p, r) \in N \times \bbR \colon \pm r > 0\}$, $M_4 = N \times \{0\} \leftrightarrow N$, and $M_2^{\pm} = \varnothing$. On $M_5$, $g := r^{-2} g'$ is Ricci-positive, and $(N, \mbc_N)$ is a conformal inf\/inity for either of $\smash{(M_5^{\pm}, g\vert_{M_5^{\pm}})}$. The inf\/initesimal symmetry algebra $\mfaut(\mbc)$ of $\mbc$ has dimension $6$, and is spanned by $\mcX := y \partial_x - p^2 \partial_p + p \partial_a$, $\mcH := -x \partial_x + y \partial_y + 2 p \partial_p - \partial_a$, $\mcY := x \partial_y + \partial_p$, $\mcZ := e^{-a}[(x p - y) \partial_p - x \partial_a]$, $\mcA := - 2 \partial_a + r \partial_r$, $\partial_r$. Now, $\mcX \wedge \mcH \wedge \mcY \wedge \mcA \wedge \partial_r = -2 (x p - y)^2 \partial_x \wedge \partial_y \wedge \partial_p \wedge \partial_a \wedge \partial_r$, which vanishes nowhere on $M$, so $(M, \mbc)$ is homogeneous.

Computing the compatible parallel tractor $3$-forms, and in particular using \eqref{equation:phi-parameterization-Ricci-positive}, gives that one $1$-parameter family of conformally isometric oriented $(2, 3, 5)$ distributions~$\mbD_t^{\mp}$ that induce $\mbc$ and are related by the Einstein scale~$r$ is given on~$M_5$ as
\begin{gather*}
	\mbD_t^{\mp} :=		\bigg\langle \pm \frac{r e^{-a \mp t}}{x p - y} \hat E_X + 2 \partial_a ,\mp \left[2 e^{2 a \pm t} (x p - y)^2 + \frac{1}{2} e^{\mp t} r^2\right] \hat E_X \\
\hphantom{\mbD_t^{\mp} :=		\bigg\langle}{} + e^a r (x p - y) \hat E_H \pm 2 (x p - y)^2 e^{2 a \pm t} \hat E_Y + \frac{1}{r} \partial_a + \partial_r \bigg\rangle .
\end{gather*}
Computing the wedge product of the two spanning f\/ields shows that this span extends smoothly across $M_4$ to a $(2, 3, 5)$ distribution on all of $M$. By def\/inition this family is $\mcD(\mbD_0^-; r)$, and the corresponding conformal Killing f\/ield is $\iota_7(r) = \mcA$. The inf\/initesimal symmetry algebra of $\smash{\mbD_t^{\pm}}$ is $\smash{\mfaut(\mbD_t^{\pm})} = \smash{\langle \mcX, \mcH, \mcY, \pm e^{\pm t} \mcZ - 2 \partial_r \rangle} \cong \smash{\mfgl(2, \bbR)}$. In particular, this furnishes an example of an inhomogeneous $(2, 3, 5)$ distribution that induces a homogeneous conformal structures.

The metric $g'$ is itself Ricci-f\/lat, so the conformal structure $\mbc$ admits two linearly independent almost Einstein scales. In the scale of $\mbc$ determined by $g'$, $\aEs(\mbc) = \langle 1, r \rangle$, and the corresponding conformal Killing f\/ields are spanned by $\iota_7(1) = \mcZ + \partial_r$ and $\iota_7(r) = \mcA$. These scales correspond to two linearly independent parallel tractors, which reduces the conformal holonomy $\Hol(\mbc)$ to a proper subgroup of $\SL(3,\bbR)$; computing gives $\Hol(\mbc) \cong \SL(2, \bbR) \ltimes \bbR^2$.
\end{Example}

\begin{Example}[submaximally symmetric $(2, 3, 5)$ distributions]
In Cartan's analysis \cite{CartanFiveVariables} of the equivalence problem for $(2, 3, 5)$ distributions, he showed that if the dimension of the inf\/initesimal symmetry algebra of a $(2, 3, 5)$ distribution $\mbD$ has inf\/initesimal symmetry algebra of dimension $< 14$ (equivalently, if it is not locally f\/lat) and satisf\/ies a natural uniformity condition, then $\dim \mfaut(\mbD) \leq 7$. (It was shown much more recently, in~\cite{KruglikovThe}, that the uniformity condition is unnecessary.) Moreover, equality holds if\/f the distribution is locally equivalent, up to a suitable notion of complexif\/ication, to the distribution
\begin{gather}\label{equation:submaximal-distributions}
\mbD_I :=	\big\langle		\partial_q ,\partial_x + p \partial_y + q \partial_p - \tfrac{1}{2}\big[q^2 + \tfrac{10}{3} I p^2 + \big(1 + I^2\big) y^2\big] \partial_z	\big\rangle
\end{gather}
on $\bbR^5_{xypqz}$ for some constant $I$.\footnote{The coef\/f\/icient $\smash{\frac{10}{3}}$ corrects an arithmetic error in \cite[Section~9, equation~(6)]{CartanFiveVariables}. Also, note that we have specialized the formula given there to constant~$I$.}

The almost Einstein geometry of the distributions $\mbD_I$ is discussed in detail in \cite{Willse}: The induced conformal structure $\mbc_I := \mbc_{\mbD_I}$ contains the representative metric
\begin{gather*}
	g_I := \left[-\tfrac{3}{2} \big(I^2 + 1\big) y^2 + 2 I p^2 - \tfrac{1}{2} q^2\right] \! dx^2 - 4 I p \,dx \,dy + q \,dx \,dp \\
\hphantom{g_I :=}{}- 3 p \,dx \,dq - 3 \,dx \,dz - 3 I \,dy^2 + 3 \,dy \,dq - 2 \,dp^2 .
\end{gather*}
The trivializations by $g_I$ of the almost Einstein scales of $\mbc_I$ are the pullbacks by the projection $\bbR^5_{xypqz} \to \bbR_x$ of the solutions of the homogeneous ODE $\sigma'' - \tfrac{1}{3} I \sigma = 0$ in $x$, and all of these turn out to be Ricci-f\/lat. In particular the vector space of almost Einstein scales of $\mbc_I$ is $2$-dimensional, so by Theorem \ref{theorem:conformal-Killing-field-decomposition} $\dim \mfaut(\mbc_I) = \dim \mfaut(\mbD_I) + \dim \aEs(\mbc_I) = 9$. For all $I$, $\Hol(\mbc_I)$ is isomorphic to the $5$-dimensional Heisenberg group. Unlike for the non-Ricci-f\/lat cases, the authors are aware of no example of a $(2, 3, 5)$ distribution $\mbD$ for which $\mbc_{\mbD}$ is equal to the full ($8$-dimensional) stabilizer $\SL(2, \bbR) \ltimes Q_+$ in $\G_2$ of an isotropic vector in the standard representation.

These distributions are contained in the f\/irst class of examples of $(2, 3, 5)$ distributions whose induced conformal structures locally admit Einstein representatives \cite[Example~6]{Nurowski}.\footnote{In that reference, these distributions were given in a form not immediately recognizable as dif\/feomorphic to those in \eqref{equation:submaximal-distributions}. For $I \neq \pm \frac{3}{4}$, the distribution $\mbD_I$ is dif\/feomorphic to the distribution def\/ined via \cite[equation~(55)]{Nurowski} by the function $F(q) = q^m$, where $k = 2 m - 1$ is any value that satisf\/ies
\begin{gather*}
	I^2 = \frac{(k^2 + 1)^2}{(k^2 - 9)(\tfrac{1}{9} - k^2)} ;
\end{gather*} when $I = \pm \frac{3}{4}$, one may take $F(q) = \log q$ \cite{DoubrovKruglikov}.}
\end{Example}

\begin{landscape}
\appendix

\section{Tabular summary of the curved orbit decomposition.}
\label{appendix}

\vspace{0.5cm}

\begin{centering}
\scalebox{0.86}{
\setlength{\tabcolsep}{.2em}
\def\arraystretch{1.2}
\begin{tabular}{|c|ccc|c|c|c|c|c|c|c|c|c|}
\hline
\multirow{2}{*}{$M_a$} & \multicolumn{3}{c|}{$\varepsilon$} & \multirow{2}{*}{Section} & \multirow{2}{*}{$S$} & \multirow{2}{*}{$S \cap P$} & \multirow{2}{*}{\small structure} & \multirow{2}{*}{$A$} & \multirow{2}{*}{\small leaf space structure} & \multirow{2}{*}{$X$} & \multirow{2}{*}{$\mbL$} & \multirow{2}{*}{$\mbS$} \\
\cline{2-4}
& \parbox{0.5cm}{\centering $-1$} & \parbox{0.5cm}{\centering $ 0$} & \parbox{0.5cm}{\centering $+1$} & & & & & & & & & \\
\hline\hline
\multirow{3}{*}{$M_5^{\pm}$} & $\bullet$ & & & \multirow{3}{*}{Section~\ref{subsection:open-curved-orbits}} & $\SU(1, 2)$ & $\SU(1, 1)$ & {\small Sasaki--Einstein} & $\U(1, 1)$ & \small{K\"ahler--Einstein} & \multirow{3}{*}{$\pm H_{\Phi}(X, \bbS) > 0$} & \multirow{3}{*}{\begin{tabular}{c}$\mbL \subset [\mbD, \mbD]$ \\ $\mbL \pitchfork \mbD$\end{tabular}} & \multirow{3}{*}{--} \\
\cline{2-4} \cline{6-10}
& & $\bullet$ & & & $\SL(2, \bbR) \ltimes Q_+$ & $\SL(2, \bbR)$ & {\small null-Sasaki--Einstein} & $\GL(2, \bbR)$ & {\small null-K\"ahler--Einstein} & & & \\
\cline{2-4} \cline{6-10}
& & & $\bullet$ & & $\SL(3, \bbR)$ & $\SL(2, \bbR)$ & {\small para-Sasaki--Einstein} & $\GL(2, \bbR)$ & {\small para-K\"ahler--Einstein} & & & \\
\hline\hline
\multirow{6}{*}{$M_4$} & $\bullet$ & & & Section~\ref{subsubsection:curved-orbit-negative-hypersurface} & $\SU(1, 2)$ & $P_-$ & {\small \begin{tabular}{c} Fef\/ferman \\ conformal (sig.\ $(1, 3)$) \end{tabular}} & $P_{\SU(1, 2)}$ & {\small 3-dim. CR structure} & \multirow{6}{*}{\begin{tabular}{c} $H(X, \bbS) = 0$ \\ $(X \times \bbS) \wedge X \neq 0$ \end{tabular}} & \multirow{6}{*}{$\mbL \subset \mbD$} & \multirow{6}{*}{$\mbS \pitchfork [\mbD, \mbD]$} \\
\cline{2-10}
& & $\bullet$ & & Section~\ref{subsubsection:curved-orbit-flat-hypersurface} & $\SL(2, \bbR) \ltimes Q_+$ & $\bbR \ltimes \bbR^3$ & {\small \begin{tabular}{c} f\/ibration over \\ conformal (sig.\ $(1, 2)$) \\ + isotropic line f\/ield \end{tabular}} & $\ast^{\dagger}$ & complete f\/lag f\/ield$^{\ddagger}$ & & & \\
\cline{2-10}
& & & $\bullet$ & Section~\ref{subsubsection:curved-orbit-positive-hypersurface} & $\SL(3, \bbR)$ & $P_+$ & \begin{tabular}{c}para-Fef\/ferman \\ conformal (sig.\ $(2, 2)$)\end{tabular} & $P_{12}$ & second-order ODE & & & \\
\hline\hline
$M_2^{\pm}$ & & & $\bullet$ & Section~\ref{subsubsection:curved-orbit-M2pm} & $\SL(3, \bbR)$ & $P_1$ {\small or} $P_2$ & \small{$2$-dim. projective} & -- & -- & $X \times \bbS = \pm X$ & -- & \begin{tabular}{c} $\mbS \subset [\mbD, \mbD]$ \\ $\mbS \pitchfork \mbD$ \end{tabular} \\
\hline\hline
$M_2$ & & $\bullet$ & & Section~\ref{subsubsection:curved-orbit-M2} & $\SL(2, \bbR) \ltimes Q_+$ & $\bbR \ltimes (\bbR^4 \ltimes \bbR)$ & {\small line f\/ield} & -- & -- & \begin{tabular}{c} $X \times \bbS = 0$ \\ $X \wedge \bbS \neq 0$ \end{tabular} & -- & $\mbS \subset \mbD$ \\
\hline\hline
$M_0^{\pm}$ & & $\bullet$ & & -- & $\SL(2, \bbR) \ltimes Q_+$ & $\SL(2, \bbR) \ltimes Q_+$ & {\small trivial} & -- & -- & $X \wedge \bbS = 0$ & -- & -- \\
\hline
\end{tabular}
}

\scalebox{0.86}
{
\small

\begin{tabular}{cp{24cm}}
 $\dagger$ & This is the solvable $5$-dimensional group $\bbR \ltimes ((\bbR^2 \ltimes \bbR) \oplus \bbR)$. \\
$\ddagger$ & The leaf space also inherits a degenerate conformal bilinear form with respect to which the line f\/ield is isotropic and the plane distribution is isotropic but not totally isotropic.
\end{tabular}
}

\vspace{1cm}

\scalebox{0.82}{
\setlength{\tabcolsep}{.2em}
\def\arraystretch{1.2}
\begin{tabular}{|c|c|c|c|c|}
\hline
$M_a$
	& $\xi^a$
	& $I_{ab}$
	& $J_{ab}$
	& $K_{ab}$ \\
\hline\hline
$M_5^{\pm}$
	& \eqmakebox[IJK]{$\theta^a$}
	& \eqmakebox[IJK]{$\tfrac{1}{2} (-\varepsilon \phi_{ab} + \barphi_{ab})$}
	& \eqmakebox[IJK]{$-\theta^c \chi_{cab}$}
	& \eqmakebox[IJK]{$\tfrac{1}{2} (-\varepsilon \phi_{ab} - \barphi_{ab})$} \\
\hline
$M_4 $
	& $\mu_b \phi^{ba}$
	& $-\varepsilon \phi_{ab} - 2 \mu_{[a} \xi_{b]}$
	& $3 \mu^c \phi_{[ca} \theta_{b]}$
	& $2 \mu_{[a} \xi_{b]}$ \\
\hline
$M_2^{\pm}$
	& $0$
	& $- \phi_{ab}$
	& $ \phi_{ab}$
	& $0$ \\
\hline
$M_2 $
	& $0$
	& $0$
	& $0$
	& $0$ \\
\hline
$M_0^{\pm}$
	& $0$
	& $0$
	& $0$
	& $0$ \\
\hline
\end{tabular}
}

\scalebox{0.82}
{
\small
\begin{tabular}{c}
The formulae for the open orbits $M_5^{\pm}$ are given in the scale $\sigma$.
\end{tabular}
}

\end{centering}

\end{landscape}

\subsection*{Acknowledgements}
It is a pleasure to thank Andreas \v{C}ap for discussions about curved orbit decompositions and natural operators on $3$-dimensional CR and Legendrean contact structures, Boris Doubrov and Boris Kruglikov for discussions about the geometry of second-order ODEs modulo point transformations, Rod Gover for comments about conformal tractor geometry, John Huerta for comments about the algebra of $\G_2$, Pawe\l{} Nurowski for a suggestion that gave rise to Example~\ref{example:distinguished-rolling-distribution}, and Michael Eastwood and Dennis The for comments about various aspects of the project. Ian Anderson's Maple package~\texttt{DifferentialGeometry} was used extensively, including for the derivation of Proposition~\ref{proposition:compatible-g2-structures} and Algorithm \ref{algorithm:recovery} and the preparation of Example~\ref{example:Dirichlet-Ricci-positive}, and it is again a pleasure to thank him for helpful comments about the package's usage. Finally, the authors thank the referees for several helpful comments and suggestions.

The f\/irst author is an INdAM (Istituto Nazionale di Alta Matematica) research fellow. She gratefully acknowledges support from the Austrian Science Fund (FWF) via project J3071--N13 and support from project FIR--2013 Geometria delle equazioni dif\/ferenziali. The second author gratefully acknowledges support from the Australian Research Council and the Austrian Science Fund (FWF), the latter via project P27072--N25.

\pdfbookmark[1]{References}{ref}
\LastPageEnding

\end{document}